\documentclass[reqno,twoside]{amsart}
\usepackage{amsfonts}
\usepackage{amsmath,amssymb}
\usepackage[dvips,bottom=1.4in,right=1in,top=1in, left=1in]{geometry}
\usepackage{epic}
\usepackage{pstricks}
\usepackage{graphicx}

\newtheorem{theorem}{Theorem}[section]
\newtheorem{acknowledgement}[theorem]{Acknowledgement}

\newtheorem{corollary}[theorem]{Corollary}

\newtheorem{definition}[theorem]{Definition}

\newtheorem{lemma}[theorem]{Lemma}

\newtheorem{proposition}[theorem]{Proposition}

\newtheorem{remark}[theorem]{Remark}

\numberwithin{equation}{section}

\def\bOm{\overline{\Omega}}

\input{tcilatex}

\begin{document}
\title[Transmission problems]{Transmission problems with nonlocal boundary
conditions and rough dynamic interfaces}
\author{Ciprian G. Gal}
\address{C. G. Gal, Department of Mathematics, Florida International
University, Miami, 33199 (USA)}
\email{cgal@fiu.edu}
\author{Mahamadi Warma}
\address{M.~Warma, University of Puerto Rico, Faculty of Natural Sciences,
Department of Mathematics (Rio Piedras Campus), PO Box 70377 San Juan PR
00936-8377 (USA)}
\email{mahamadi.warma1@upr.edu, mjwarma@gmail.com}
\thanks{The research of the second author was partially supported by the
AFOSR Grant FA9550-15-1-0027}

\begin{abstract}
We consider a transmission problem consisting of a semilinear parabolic
equation in a general non-smooth setting with emphasis on rough interfaces
which bear a fractal-like geometry and nonlinear dynamic (possibly,
nonlocal)\ boundary conditions along the interface. We give a unified
framework for existence of strong solutions, existence of finite dimensional
attractors and blow-up phenomena for solutions under general conditions on
the bulk and interfacial nonlinearities with competing behavior at infinity.
\end{abstract}

\maketitle

\section{Introduction}

In this paper we aim to give a unified framework for a general class of
transmission problems of the form%
\begin{equation}
\partial _{t}u-\text{div}\left( \mathbf{D}\nabla u\right) +f\left( u\right)
=0,\;\;\;\text{ in }\;\,J\times (\Omega \backslash \Sigma ),  \label{p1}
\end{equation}%
with $J=(0,T)$, $f$ is a nonlinear function which can be either a source or
a sink, and the matrix $\mathbf{D=}\left( d_{i,j}\left( x\right) \right) $
is symmetric, bounded, measurable and non-degenerate such that
\begin{equation*}
\left\langle \mathbf{Dv,v}\right\rangle _{\mathbb{R}^{N}}\geq
d_{0}\left\vert \mathbf{v}\right\vert _{\mathbb{R}^{N}}^{2},\;%
\mbox{ for any
}\;\mathbf{v}\in {\mathbb{R}}^{N}\;\mbox{ with some
constant }\;d_{0}>0.
\end{equation*}%
In \eqref{p1}, $\Omega \subset {\mathbb{R}}^{N}$ is a bounded domain (open
and connected) with Lipschitz continuous boundary $\partial \Omega $ that is
disjointly decomposed into a Dirichlet part and a Neumann part, $\partial
\Omega =\Gamma _{N}\cup \Gamma _{D}.$ More precisely, on $\partial \Omega $
we consider Dirichlet and Neumann boundary conditions:%
\begin{equation}
u=0\text{ on }J\times \Gamma _{D},\text{ }(\mathbf{D}\nabla u)\cdot \nu =0%
\text{ on }J\times \Gamma _{N},  \label{p2}
\end{equation}%
where $\nu $ denotes the unit outer normal vector on $\partial \Omega $.
Moreover, in \eqref{p1}, $\Sigma $ is a $d$-dimensional fractal-like
"surface" contained in $\Omega \subset {\mathbb{R}}^{N}$ with $d\in
(N-2,N)\cap (0,N)$. We assume that $\mathcal{H}^{d}(\Sigma )<\infty $ where $%
\mathcal{H}^{d}$ denotes the $d$-dimensional Hausdorff measure and we denote
by $\mu _{\Sigma }$ the restriction of $\mathcal{H}^{d}$ to the set $\Sigma $%
. Our bounded open set $\Omega \subset \mathbb{R}^{N}$ and $\Sigma $ are
such that $\Omega =\Omega _{1}\cup \Sigma \cup \Omega _{2}$ with $\Sigma =%
\overline{\Omega }_{1}\cap \overline{\Omega }_{2}$. That is, $\Omega $ is
divided into two domains $\Omega _{1}$ and $\Omega _{2}$ with $\Omega
_{1}\cap \Omega _{2}=\emptyset $ and $\Sigma $ is a $d$-dimensional surface
lying strictly in $\Omega $. We notice that $\partial (\Omega \setminus
\Sigma )=\partial \Omega \cup \Sigma $. We make the following convention for
a function $u$ defined on $\Omega \setminus \Sigma $:
\begin{equation*}
u(x)=u_{+}(x),\;x\in \Omega _{+}:=\Omega _{1}\;\;\mbox{ and }%
\;\;u(x)=u_{-}(x),\;\;x\in \Omega _{-}:=\Omega _{2}
\end{equation*}%
and we will apply the same principle to other functions defined over $\Omega
\setminus \Sigma $.

Before we give the boundary conditions that we shall consider on $\Sigma$,
we have to introduce first a generalized version of a weak normal derivative
of a function $u$ on the interface $\Sigma$. We notice that the definition
given below is a version of the one introduced in \cite{BW2,Wa}. Let $\sigma
$ denote the surface measure on $\partial \Omega =\Gamma _{D}\cup \Gamma
_{N} $, that is, the restriction to $\partial\Omega$ of the $(N-1)$%
-dimensional Hausdorff measure, and recall that $\nu $ is well defined as a
unit normal vector on $\partial \Omega $. Since by assumption $\Sigma $ may
be so irregular that no normal vector $\nu _{\Sigma }$ can be defined on $%
\Sigma $, we will use the following generalized version of a normal
derivative of a function on $\Sigma$. Let $\eta$ be a signed Radon measure
on $\Sigma$ and $F:\overline{\Omega }\rightarrow {\mathbb{R}}^{N}$ a
measurable function. If there exists a function $g\in L_{loc}^{1}({\mathbb{R}%
}^{N})$ satisfying%
\begin{equation}
\int_{\Omega \backslash \Sigma }F\cdot \nabla \varphi dx=\int_{\Omega
\backslash \Sigma }g\varphi dx+\int_{\partial \Omega }(F\cdot \nu)\varphi\;
d\sigma +\int_{\Sigma }\varphi d\eta  \label{NE}
\end{equation}%
for all $\varphi \in C^{1}(\overline{\Omega })$, then we say that $\eta$ is
the normal measure of $F$ and we denote $N_{\Sigma}^\star(F):=\eta$. If the
normal measure $N_{\Sigma }^{\star }(F)$ exists, then it is unique and $%
dN_{\Sigma }^{\star }(\psi F)=\psi dN_{\Sigma }^{\star }(F)$ for all $\psi
\in C^{1}(\overline{\Omega })$. If $u\in W_{loc}^{1,1}(\Omega \backslash
\Sigma )$ and $N_{\Sigma }^{\star }(\mathbf{D}\nabla u)$ exists, then we
will denote by $N_{\mathbf{D}}(u):=N_{\Sigma }^{\star }(\mathbf{D}\nabla u)$
and call it the generalized $\mathbf{D}$-normal derivative of $u$ on $\Sigma$%
.

\begin{figure}[h]
\centering
\includegraphics[scale=0.50]{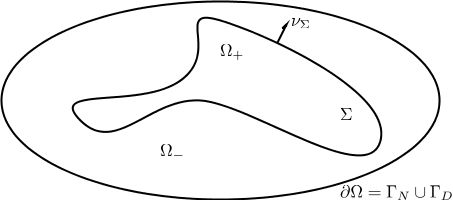}
\caption{Example of Domain $\Omega =\Omega _+\cup \Sigma \cup \Omega _-$.}
\end{figure}

To justify this definition, assume that $\Sigma $ is a Lipschitz
hypersurface of dimension $\left( N-1\right) $. Then the measure $\mu
_{\Sigma }=\sigma _{\Sigma }$, the usual Lebesgue surface measure on $\Sigma$%
. Let $\nu_\Sigma$ be the exterior normal vector to $\Omega_1:=\Omega_+$ at $%
\Sigma$. Then the exterior normal vector to $\Omega_2:=\Omega_-$ at $\Sigma$
should be $-\nu_{\Sigma}$ (see Figure 1). Let $u\in
W^{2,2}(\Omega\setminus\Sigma)$ be such that there are $g\in L_{loc}^{1}({%
\mathbb{R}}^{N})$ and $\tilde g\in L^{1}(\Sigma ,\sigma _{\Sigma })$
satisfying
\begin{align}  \label{NE-2}
\int_{\Omega \backslash \Sigma }\mathbf{D}\nabla u\cdot \nabla \varphi dx=
\int_{\Omega \backslash \Sigma }g\varphi dx+\int_{\partial \Omega }\left(
\mathbf{D}\nabla u\cdot \nu \right) \varphi d\sigma +\int_{\Sigma }\tilde
g\varphi\;d\sigma _{\Sigma }
\end{align}%
for all $\varphi \in C^{1}(\overline{\Omega })$. On the other hand, since $%
u\in W^{2,2}(\Omega\setminus\Sigma)$, then using the classical Green formula
(recall that we are in the situation where $\Omega\setminus\Sigma$ has a
Lipschitz continuous boundary), we have that for all $\varphi \in C^{1}(%
\overline{\Omega })$,
\begin{align}  \label{NE-3}
\int_{\Omega \backslash \Sigma }\mathbf{D}\nabla u\cdot \nabla \varphi
dx=&-\int_{\Omega\setminus\Sigma }\mbox{div}(\mathbf{D}\nabla u)\varphi
dx+\int_{\partial \Omega }\left( \mathbf{D}\nabla u\cdot \nu \right) \varphi
d\sigma \\
&+\int_{\Sigma }\varphi \left( \mathbf{D}_+\nabla u_{+}-\mathbf{D}_-\nabla
u_{-}\right)\cdot \nu _{\Sigma } d\sigma _{\Sigma } .  \notag
\end{align}
Moreover, $\mbox{div}(\mathbf{D}\nabla u)\in L^2(\Omega\setminus\Sigma)$, $%
\left( \mathbf{D}\nabla u\cdot \nu \right)\in L^2(\partial\Omega)$ and $%
\left( \mathbf{D}_+\nabla u_{+}-\mathbf{D}_-\nabla u_{-}\right)\cdot \nu
_{\Sigma }\in L^2(\Sigma)$. It follows from \eqref{NE-2} and \eqref{NE-3}
that in this case, $g=-\mbox{div}(\mathbf{D}\nabla u)$ and $\tilde g=\left(%
\mathbf{D}\nabla u_{+}-\mathbf{D}\nabla u_{-}\right)\cdot \nu _{\Sigma }$ so
that
\begin{align*}
dN_{\mathbf{D }}(u):=dN_{\Sigma}^\star(\mathbf{D}\nabla u)= \tilde gd\sigma
_{\Sigma }=\left(\left( \mathbf{D}_+\nabla u_{+}-\mathbf{D}_-\nabla
u_{-}\right)\cdot \nu _{\Sigma }\right)d\sigma _{\Sigma }\;\;\mbox{ on }%
\;\Sigma,
\end{align*}
and hence,
\begin{align*}
\frac{dN_{\mathbf{D }}(u)}{d\sigma_{\Sigma}}=\left(\mathbf{D}_+\nabla u_{+}-%
\mathbf{D}_-\nabla u_{-}\right)\cdot \nu _{\Sigma }\;\;\mbox{ on }\;\Sigma.
\end{align*}
We will always use the generalized Green type identity \eqref{NE} and the
generalized $\mathbf{D}$-normal derivative of a function $u$ on $\Sigma$
introduced above.

Now, on the interfacial region $\Sigma $ we impose the dynamic boundary
condition formally given by
\begin{equation}
dN_{\mathbf{D}}(u)+\left( \delta \partial _{t}u+\beta \left( x\right)
u+\Theta _{{\Sigma }}\left( u\right) \right) d\mu _{\Sigma }=h\left(
u\right) d\mu _{\Sigma },\;\;\text{ on }\;\,J\times \Sigma ,  \label{m2}
\end{equation}%
where we are primarily interested in nonlinear sources at the interface $%
\Sigma $ such that%
\begin{equation}
\underset{\left\vert \tau \right\vert \rightarrow \infty }{\lim \inf }\frac{%
h\left( \tau \right) }{\tau }=\infty .  \label{bad-h}
\end{equation}%
The function $\beta \in L^{\infty }(\Sigma ,\mu _{\Sigma })$ satisfies $%
\beta (x)\geq 0$ for $\mu _{\Sigma }$-a.e. $x\in \Sigma $ but it is \emph{%
not identically} equal to zero on $\Sigma $ and $\delta \geq 0$ is a real
number. Note that the condition on $h$ implies that $h$ has a bad sign and
is of superlinear growth at infinity (i.e., $h\left( \tau \right) \sim
c_{h}\left\vert \tau \right\vert ^{\alpha }\tau ,$ as $\left\vert \tau
\right\vert \rightarrow \infty $, for some $\alpha >0$ and $c_{h}>0$). In (%
\ref{m2}), we define $\Theta _{{\Sigma }}$ as a non-local operator in the
following fashion (see Section \ref{sub-22} below):
\begin{equation}
\langle \Theta _{{\Sigma }}(u),v\rangle :=\int \int_{\Sigma \times \Sigma
}K(x,y)(u(x)-u(y))(v(x)-v(y))d\mu _{\Sigma }\left( x\right) d\mu _{\Sigma
}\left( y\right) ,  \label{nonlocal-op}
\end{equation}%
for functions $u,v\in \widetilde{W}^{1,2}(\Omega )$ such that $\Theta
_{\Sigma }(u)\in L^{2}(\Sigma ,d\mu _{\Sigma })$ and
\begin{equation*}
\int \int_{\Sigma \times \Sigma }K(x,y)|u(x)-u(y)|^{2}d\mu _{\Sigma }(x)d\mu
_{\Sigma }(y),\;\;\int \int_{\Sigma \times \Sigma }K(x,y)|v(x)-v(y)|^{2}d\mu
_{\Sigma }(x)d\mu _{\Sigma }(y)<\infty ,
\end{equation*}%
where the kernel $K:\Sigma \times \Sigma \rightarrow \mathbb{R}_{+}$ is
symmetric and satisfies
\begin{equation*}
c_{0}\leq K\left( x,y\right) \left\vert x-y\right\vert ^{d+2s}\leq c_{1},
\end{equation*}%
for all $x,y\in \Sigma $, $x\neq y$, for some constants $c_{0},c_{1}>0$ and $%
s\in \left( 0,1\right) $. A basic example is $K\left( x,y\right) =\left\vert
x-y\right\vert ^{-\left( d+2s\right) }$. In this case, $\Theta _{\Sigma
}=\left( -\Delta _{\Sigma }\right) ^{s}$ is a nonlocal operator
characterizing the presence of anomalous "fractional" diffusion along $%
\Sigma $. The $\mathbf{D}$-normal derivative $N_{\mathbf{D}}(u)$ of $u$ in %
\eqref{m2} is understood in the sense of \eqref{NE}. We mention that if $%
\Sigma $ is a Lipschitz hypersurface of dimension $(N-1)$, then the
condition \eqref{m2} reads
\begin{equation*}
\left( \mathbf{D}_{+}\nabla u_{+}-\mathbf{D}_{-}\nabla u_{-}\right) \cdot
\nu _{\Sigma }+\left( \delta \partial _{t}u+\beta \left( x\right) u+\Theta _{%
{\Sigma }}\left( u\right) \right) =h\left( u\right) ,\;\;\text{ on }%
\,\;J\times \Sigma .
\end{equation*}%
Finally, initial conditions for \eqref{p1}, \eqref{p2} and \eqref{m2} must
also be prescribed on $\Omega \backslash \Sigma $ and on $\Sigma ,$
respectively.

The interface problem \eqref{p1}, \eqref{p2} and \eqref{m2} with a more
simple transmission condition (\ref{m2}) on $\Sigma $, i.e. when $\delta =0,$
$\Theta _{{\Sigma }}\equiv 0$ and $h\equiv 0$, is often encountered in the
literature in heat transfer phenomena, fluid dynamics and material science,
and in electrostatics and magnetostatics. It is the usual case when two
bodies materials or fluids with different conductivities or diffusions are
involved. The corresponding transmission problem for the linear elliptic
equation assuming that the interface $\Sigma $ is at least Lipschitz smooth
and $\Theta _{{\Sigma }}\equiv 0$, $h\equiv 0$, has been considered by a
large group of mathematicians including O.A. Ladyzhenskaya and N.N.
Ural'tseva since the 1970's (see \cite{ER, HZ}). We also refer the reader to
the recent book of Borsuk \cite{B}. Under the same foregoing assumptions the
(linear) parabolic transmission problem has also been investigated recently
in \cite{BR, GK,HZ, KK, Nom} (and the references therein) in the case when $%
\delta =0$ and $\Theta _{{\Sigma }}\equiv 0,h\equiv 0$. The transmission
problem \eqref{p1} with a linear dynamic (i.e., when $\delta >0$)
transmission condition on a smooth surface $\Sigma $ (so it can be
flattened) and its influence of the solution was studied in \cite{BN}\ in
the autonomous case. Finally, a linear interface problem with linear dynamic
transmission condition involving the surface diffusion $\Delta _{\Sigma }$
on $\Sigma $ and allowing both cases in which $\mathbf{D}$ is nondegenerate
and degenerate, and when the interface $\Sigma $ is a $\left(N-1\right) $%
-dimensional Lipschitz hypersurface was also considered in \cite{EMR2,EMR}.
When $\Sigma $ is rough and fractal-like has also attracted considerable
interest over the recent years due to their importance in various
engineering, physical and natural processes, such as, \textquotedblleft
hydraulic fracturing\textquotedblright , a frequently used engineering
method to increase the flow of oil from a reservoir into a producing oil
well \cite{Ca, HP}, current flow through rough electrodes in
electrochemistry \cite{Cap, FS, TV} or diffusion and biological processes
across irregular physiological membranes \cite{La1, La2, La3, La4}.
%But so far the present
%aforementioned quoted literature assumes that $\Sigma $ is apriori given as
%either a $\left( d-1\right) $-dimensional Lipschitz surface or for the most
%part $\Sigma $ is the well-known Koch curve (snowflake) $\mathcal{C}_{k},$
%which is known to have finite Hausdorff dimension.

It is our goal to give a unified analysis of these transmission problems for
a large class of fractal-like interfaces, to go beyond the present studies
which have focused mainly on the linear interface problem with linear
transmission conditions and mainly well-posedness like results. We will
derive stronger and sharper results in terms of existence, regularity and
stability of bounded solutions. Then we also show the existence of finite
dimensional attractors for the nonlinear transmission problem \eqref{p1}, %
\eqref{p2} and \eqref{m2} especially in non-dissipative situations in which
a bad source $h$ (via (\ref{bad-h})) is present along $\Sigma $ such that
energy is always fed in all of $\Omega \backslash \Sigma $ through $\Sigma $
but eventually dissipated completely by a nonlinear "good" source $f$. It
turns out that blow-up phenomena and global existence are strictly related
to competing conditions and the behavior of the nonlinearities $f,h$ at
infinity as $\left\vert \tau\right\vert \rightarrow \infty $. We state sharp
balance conditions between the bulk source/sink $f$ and interfacial source $%
h $ allowing us to directly compare them even when they are acting on
separate parts of the domain $\Omega ,$ on $\Omega \backslash \Sigma $ and
on $\Sigma $, respectively. Notably, similar ideas have been used by Gal
\cite{Gal0} in the treatment of parabolic $p$-Laplace equations with
nonlinear reactive-like dynamic boundary conditions (cf. also Bernal and
Tajdine \cite{BT} for parabolic problems subject to nonlinear Robin
conditions), and by Gal and Warma in \cite{GW2} to give a complete
characterization of the long-time behavior as time goes to infinity (in
terms of a finite dimensional global attractor, $\omega $-limit sets and
Lyapunov functions) for semilinear parabolic equations on rough domains
subject to nonlocal Robin boundary conditions. Since most of the
aforementioned applications have a real physical meaning in non-reflexive
Banach spaces, like $L^{1}\left( \Omega \right) $ or $L^{\infty }\left(
\Omega \right) $, our prefered notion of generalized solutions and nonlinear
solution semigroups will naturally be given in $L^{\infty }$-type spaces. It
is this context in which in fact our balance conditions on the
nonlinearities become also optimal and sharp in a certain sense. In this
work, we also develop a unified framework in order to resolve the
difficulties coming from having to deal with a rough fractal-like interface $%
\Sigma $, and develop new tools based on potential analysis to handle the
present case. This new approach allows us to not only overcome the
difficulties mentioned earlier, but also to do so using only elementary
tools from Sobolev spaces, nonlinear semigroup theory and dynamical systems
theory avoiding the use of sophisticated tools from harmonic analysis.

The remainder of the paper is structured as follows. In Section \ref{prelim}%
, we establish our notation and give some basic preliminary results for the
operators and spaces appearing in our transmission model. In Section \ref%
{strong-sol}, we prove some well-posedness results for this model; in
particular, we establish existence and stability results for strong
solutions. In Section \ref{gl}, we prove results which establish the
existence of global and exponential attractors for \eqref{p1}, \eqref{p2}
and \eqref
{m2}. Section \ref{bl} contains some new results on the blow-up of strong
solutions.

\section{Some generation of semigroup results}

\label{prelim}

In this section we introduce the functional framework associated with the
transmission problem in question and then derive semigroup type results for
the linear operator corresponding to the linear problem. All these tools are
necessary in the study of the nonlinear transmission problem \eqref{p1}, %
\eqref{p2} and \eqref{m2}.

\subsection{Functional framework}

Let $\Omega \subset {\mathbb{R}}^{N}$ be a bounded open set with boundary $%
\partial \Omega $ and $1\leq p<\infty $. We denote by
\begin{equation*}
W^{1,p}(\Omega ):=\{u\in L^{p}(\Omega ):\;\int_{\Omega }|\nabla
u|^{p}dx<\infty \}
\end{equation*}%
the first order Sobolev space endowed with the norm
\begin{equation*}
\Vert u\Vert _{W^{1,p}(\Omega )}:=\left( \int_{\Omega
}|u|^{p}\;dx+\int_{\Omega }|\nabla u|^{p}dx\right) ^{\frac{1}{p}}.
\end{equation*}%
We let
\begin{equation*}
\widetilde{W}^{1,p}(\Omega ):=\overline{W^{1,p}(\Omega )\cap C(\overline{%
\Omega })}^{W^{1,p}(\Omega )}\;\;\mbox{ and }\;W_{0}^{1,p}(\Omega ):=%
\overline{\mathcal{D}(\Omega )}^{W^{1,p}(\Omega )},
\end{equation*}
where $\mathcal{D}(\Omega)$ denotes the space of test functions on $\Omega$.
It is well-known that $\widetilde{W}^{1,p}(\Omega )$ is a proper closed
subspace of $W^{1,p}(\Omega )$ but the two spaces coincide if $\Omega $ has
a continuous boundary (see., e.g. \cite[Theorem 1 p.23]{MaPo}). Moreover, by
definition, $W_{0}^{1,p}(\Omega )$ is a closed subspace of $\widetilde{W}%
^{1,p}(\Omega )$ and hence, of $W^{1,p}(\Omega )$. Next, let $\Gamma \subset
\partial \Omega $ be a closed set. We denote by $W_{0,\Gamma }^{1,p}(\Omega
) $ the closure of the set%
\begin{equation*}
\{u\in W^{1,p}(\Omega )\cap C(\overline{\Omega }):\;u=0\;\mbox{ on }\;\Gamma
\}\;\;\text{ in }\;\;W^{1,p}(\Omega )
\end{equation*}%
so that $W_{0}^{1,p}(\Omega )$ coincides with $W_{0,\partial \Omega
}^{1,p}(\Omega )$.

\begin{definition}
Let $p\in [1,\infty)$ be fixed. We will say that a given open set $\Omega
\subset {\mathbb{R}}^{N}$ has the ${W}^{1,p}$-extension property, if for
every $u\in {W}^{1,p}(\Omega )$, there exists $U\in W^{1,p}({\mathbb{R}}%
^{N}) $ such that $U|_{\Omega }=u$. Similarly, we will say that $\Omega $
has the $\widetilde{W}^{1,p}$-extension property, if for every $u\in
\widetilde{W}^{1,p}(\Omega )$, there exists $U\in W^{1,p}({\mathbb{R}}^{N})$
such that $U|_{\Omega }=u$. In that case, the extension operator $%
E:\;W^{1,p}(\Omega )\rightarrow W^{1,p}({\mathbb{R}}^{N}) $ (resp., $%
\widetilde{E}:\;\widetilde{W}^{1,p}(\Omega )\rightarrow W^{1,p}({\mathbb{R}}%
^{N})$ is linear and bounded, that is, there exists a constant $C>0$ such
that for every $u\in W^{1,p}(\Omega )$ (resp., $u\in \widetilde{W}%
^{1,p}(\Omega )$),
\begin{equation*}
\Vert Eu\Vert _{W^{1,p}({\mathbb{R}}^{N})}=\Vert U\Vert _{W^{1,p}({\mathbb{R}%
}^{N})}\leq C\Vert u\Vert _{W^{1,p}(\Omega )}.
\end{equation*}%
(resp., $\Vert \widetilde{E}u\Vert _{W^{1,p}({\mathbb{R}}^{N})}=\Vert U\Vert
_{W^{1,p}({\mathbb{R}}^{N})}\leq C\Vert u\Vert _{W^{1,p}(\Omega )}$). See
\cite{HKT} for further details.
\end{definition}

If $\Omega $ has the $W^{1,p}$-extension property, then there exists a
constant $C>0$ such that for every $u\in W^{1,p}(\Omega )$,
\begin{equation}
\Vert u\Vert _{L^{q }(\Omega )}\leq C\Vert u\Vert _{W^{1,p}(\Omega )},\;\;q
\in \lbrack 1,p^{\star }],\;p^{\star }:=\frac{Np}{N-p},\;\mbox{ if }%
\;N>p,\;\;q \in \lbrack 1,\infty )\;\mbox{ if }\;N=p.  \label{sobo}
\end{equation}%
Moreover,
\begin{equation}  \label{sob-co}
W^{1,p}(\Omega )\hookrightarrow C^{0,1-\frac Np}(\overline{\Omega })\;\;%
\mbox{ if }\; N<p.
\end{equation}
The embedding \eqref{sobo} and \eqref{sob-co} remain valid with $%
W^{1,p}(\Omega )$ replaced by $\widetilde W^{1,p}(\Omega )$, if $\Omega $ is
assumed to have the $\widetilde W^{1,p}$-extension property. In particular,
all these statements are true if $\Omega $ has a Lipschitz continuous
boundary $\partial \Omega .$

\begin{remark}
\emph{We mention that if $\Omega$ has the $W^{1,p}$-extension property, then
$\widetilde{W}^{1,p}(\Omega)=W^{1,p}(\Omega)$ and hence, $\Omega$ has the $%
\widetilde W^{1,p}$-extension property. For instance, in two dimensions the
open set $\Omega :=B(0,1)\backslash \lbrack -1,1]\times \{0\}\subset\mathbb{R%
}^2$ enjoys the $\widetilde{W}^{1,p}$-extension property for any $p\in
\lbrack 1,\infty )$ but does not possess the $W^{1,p}$-extension property
(see e.g. \cite{BW}). In general if $p\in (1,\infty)$ and $\Omega \subset {%
\mathbb{R}}^{N}$ has the $W^{1,p}$-extension property and $F\subset \Omega $
is a relatively closed set with $\mathcal{H}^{N-1}(F)<\infty $ (i.e., the $%
(N-1)$-dimensional Hausdorff measure $\mathcal{H}^{N-1}$ of $F$ is finite),
then the open set $\Omega \backslash F$ has the $\widetilde{W}^{1,p}$%
-extension property. However, generally it may happen (as the case of the
two dimensional example given above) that $\Omega \backslash F$ does not
possess the $W^{1,p}$-extension property (see e.g. \cite{BW} for other
examples and for more details on this subject).}
\end{remark}

\begin{definition}
Let $F\subset {\mathbb{R}}^{N}$ be a compact set, $d\in (0,N]$ and $\mu $ a
regular Borel measure on $F$. We say that $\mu$ is an \emph{upper $d$%
-Ahlfors measure} if there exists a constant $C>0 $ such that for every $%
x\in F$ and every $r\in (0,1]$, one has
\begin{equation*}
\mu (B(x,r)\cap F)\leq Cr^{d}.
\end{equation*}
\end{definition}

\begin{remark}
\emph{Let $\Omega \subset \mathbb{R}^{N}$ be an arbitrary bounded domain
with boundary $\partial\Omega$. Let $1<p<\infty$ and $d\in (N-p,N)\cap (0,N)$
be the Hausdorff dimension of $\partial\Omega$. Assume also that $\Omega$
has the $W^{1,p}$-extension property. Then $\mathcal{H}_{\mid \partial
\Omega }^{d}$ is an upper $d$-Ahlfors measure. For instance, if $\Omega
\subset \mathbb{R}^{2}$\ is the bounded domain enclosed by the Koch curve,
then $\mathcal{H}_{\mid \partial \Omega }^{d}\left( \partial \Omega \right)
<\infty $ where $d=\ln (4)/\ln (3)$ is the Hausdorff dimension of $\partial
\Omega $. Moreover $\Omega $ has the $W^{1,p}$-extension property and the
restriction of $\mathcal{H}^{d}$ to $\partial \Omega $ is an upper $d$%
-Ahlfors measure (see \cite{Bie,GW,JW}).}
\end{remark}

Next, let $S\subset {\mathbb{R}}^{N}$ be a compact set, $\mu $ an upper $d$%
-Ahlfors measure on $S$ for some $d\in (N-2,N)\cap (0,N)$ and let $0<s<1$.
We define the fractional order Sobolev space
\begin{equation*}
\mathbb{B}_{d,s}^{2}(S,\mu ):=\{u\in L^{2}(S,\mu ):\;\int_{S}\int_{S}\frac{%
|u(x)-u(y)|^{2}}{|x-y|^{d+2s}}d\mu (x)d\mu (y)<\infty \}
\end{equation*}%
and we endow it with the norm
\begin{equation*}
\Vert u\Vert _{\mathbb{B}_{d,s}^{2}(S,\mu )}:=\left( \int_{S}|u|^{2}d\mu
+\int_{S}\int_{S}\frac{|u(x)-u(y)|^{2}}{|x-y|^{d +2s}}d\mu (x)d\mu
(y)\right) ^{\frac{1}{2}}.
\end{equation*}

The following result is taken from \cite[Theorems 1.1 and 6.7]{Bie} (see
also \cite{Dan}).

\begin{theorem}
\label{Theo-Bie} Let $\Omega \subset {\mathbb{R}}^{N}$ be a bounded open set
and assume that it has the $W^{1,2}$-extension property. Let $S\subset
\overline{\Omega }$ be a compact set and $\mu $ an upper $d$-Ahlfors measure
on $S$ for some $d\in (N-2,N)\cap (0,N)$. Then the following assertions hold.

\begin{enumerate}
\item There exists a constant $C>0$ such that for every $u\in W^{1,2}(\Omega
)$,
\begin{equation}
\Vert u\Vert _{L^{2_\star}(S,\mu )}\leq C\Vert u\Vert _{W^{1,2}(\Omega
)},\;\;2_\star:=\frac{2d}{N-2}>2.  \label{sob-meas}
\end{equation}

\item For every $0<s<1-\frac{N-d}{2}$, there exists a constant $C>0$ such
that for every $u\in W^{1,2}(\Omega )$,
\begin{equation}  \label{sob-besov}
\Vert u\Vert _{\mathbb{B}_{d,s}^{2}(S,\mu )}\leq C\Vert u\Vert
_{W^{1,2}(\Omega )}.
\end{equation}
\end{enumerate}
\end{theorem}

We notice that the estimates \eqref{sob-meas} and \eqref{sob-besov} remain
valid with $W^{1,2}(\Omega)$ replaced by $\widetilde W^{1,2}(\Omega)$ if $%
\Omega$ is assumed to have the $\widetilde W^{1,2}$-extension property.

Next, we introduce the notion of Dirichlet form on an $L^{2}$-type space
(see \cite[Chapter 1]{Fuk}). To this end, let $X$ be a locally compact
metric space and $\eta$ a Radon measure on $X$. Let $L^2(X,\eta)$ be the
real Hilbert space with inner product $\left( \cdot ,\cdot \right) $ and let
$\mathcal{E}$ with domain $D(\mathcal{E})$ be a bilinear form on $%
L^2(X,\eta) $.

\begin{definition}
\label{Diri-form}The form $\mathcal{E}$ is said to be a Dirichlet form if
the following conditions hold:

\begin{enumerate}
\item $\mathcal{E}:D\left( \mathcal{E}\right) \times D\left( \mathcal{E}%
\right) \rightarrow \mathbb{R}$, where the domain $D\left( \mathcal{E}%
\right) $ of the form is a dense linear subspace of $L^2(X,\eta).$

\item $\mathcal{E}\left( u,v\right) =\mathcal{E}\left( v,u\right) $, $%
\mathcal{E}\left( u+v,w\right) =\mathcal{E}\left( u,w\right) +\mathcal{E}%
\left( v,w\right) $, $\mathcal{E}\left( au,v\right) =a\mathcal{E}\left(
u,v\right) $ and $\mathcal{E}\left( u,u\right) \geq 0$, for all $u,v,w\in
D\left( \mathcal{E}\right) $ and $a\in \mathbb{R}$.

\item Let $\lambda >0$ and define $\mathcal{E}_{\lambda }\left( u,v\right) =%
\mathcal{E}\left( u,v\right) \mathcal{+\lambda }\left( u,v\right) ,$ for $%
u,v\in D\left( \mathcal{E}_{\lambda }\right) =D\left( \mathcal{E}\right) $.
The form $\mathcal{E}$ is closed, that is, if $u_{n}\in D\left( \mathcal{E}%
\right) $ with $\mathcal{E}_{\lambda }\left( u_{n}-u_{m},u_{n}-u_{m}\right)
\rightarrow 0$ as $n,m\rightarrow \infty $, then there exists $u\in D\left(
\mathcal{E}\right) $ such that $\mathcal{E}_{\lambda }\left(
u_{n}-u,u_{n}-u\right) \rightarrow 0$ as $n\rightarrow \infty .$

\item For each $\epsilon >0$ there exists a function $\phi _{\epsilon }:%
\mathbb{R}\rightarrow \mathbb{R}$, such that $\phi_{\epsilon}\in C^\infty({%
\mathbb{R}})$, $\phi _{\epsilon }\left( t\right) =t,$ for $t\in \left[ 0,1%
\right] ,$ $-\epsilon \leq \phi _{\epsilon }\left( t\right) \leq 1+\epsilon $%
, for all $t\in \mathbb{R}$, $0\leq \phi _{\epsilon }\left( t\right) -\phi
_{\epsilon }\left( \tau\right) \leq t-\tau$, whenever $\tau<t$, such that $%
u\in D\left( \mathcal{E}\right) $ implies $\phi _{\epsilon }\left( u\right)
\in D\left( \mathcal{E}\right) $ and $\mathcal{E}\left( \phi _{\epsilon
}\left( u\right) ,\phi _{\epsilon }\left( u\right) \right) \leq \mathcal{E}%
\left( u,u\right) .$
\end{enumerate}
\end{definition}

\begin{remark}
\emph{Clearly, $D\left( \mathcal{E}\right) $ is a real Hilbert space with
inner product $\mathcal{E}_{\lambda }\left( u,u\right) $ for each $\lambda
>0 $. We recall that a form $\mathcal{E}$ which satisfies (a)-(c) is closed
and symmetric. If $\mathcal{E}$ also satisfies (d), then it is said to be a
Markovian form.}
\end{remark}

Let $T=(T(t))_{t\ge 0}$ be a semigroup on $L^2(X,\eta)$. We say that $T(t)$
is positively-preserving, if $T(t)u\ge 0$ $\eta$-a.e. and for all $t\ge 0$,
whenever $u\in L^2(X,\eta)$ and $u\ge 0$ $\eta$-a.e. If $T(t)$ is
positively-preserving and $T(t)$ is a contraction on $L^\infty(X,\eta)$,
that is,
\begin{align*}
\|T(t)u\|_{L^\infty(X,\eta)}\le \|u\|_{L^\infty(X,\eta)},\;\forall\;t\geq
0,\;\forall\;u\in L^2(X,\eta)\cap L^\infty(X,\eta),
\end{align*}
then we will say that $T=(T(t))_{t\ge 0}$ is a Markovian (or submarkovian)
semigroup. It turns out that if $A$ is the closed linear self-adjoint
operator in $L^2(X,\mu)$ associated with the Dirichlet form $\mathcal{E}$,
then $-A$ generates a strongly continuous Markovian semigroup on $%
L^2(X,\eta) $. Conversely, the generator of every symmetric, strongly
continuous Markovain semigroup on $L^2(X,\eta)$ is given by a Dirichlet form
on $L^2(X,\eta)$. For more details on this topic we refer to Chapter 1 of
the monographs \cite{Dav,Fuk}.

\subsection{The non-local dynamic boundary conditions on the interface}

\label{sub-22}

Throughout the remainder of the article the sets $\Omega$, $\partial\Omega$,
$\Gamma_N$, $\Gamma_D$ and $\Sigma$ are as defined in the introduction.
Recall that $\Sigma \subset \Omega$ is a relatively closed set with
Hausdorff dimension $d\in (N-2,N)\cap (0,N)$ and that $\mathcal{H}%
^{d}(\Sigma )<\infty $ where $\mathcal{H}^{d}$ denotes the $d$-dimensional
Hausdorff measure. We still denote by $\mu _{\Sigma }$ the restriction of $%
\mathcal{H}^{d}$ to the set $\Sigma $. In this case, we have that $\mu
_{\Sigma }$ is an upper $d$-Ahlfors measure on $\Sigma$ (see e.g. \cite%
{BW,JW}). Recall also that
\begin{equation*}
\widetilde{W}^{1,2}(\Omega \backslash \Sigma ):=\overline{W^{1,2}(\Omega
\backslash \Sigma )\cap C(\overline{\Omega })}^{W^{1,2}(\Omega \backslash
\Sigma )}.
\end{equation*}

\begin{remark}
\emph{In general, even if $\mathcal{H}^{N}(\Sigma )=0$, the space $%
\widetilde{W}^{1,2}(\Omega \backslash \Sigma )$ is \textbf{not} equal to $%
\widetilde{W}^{1,2}(\Omega )=W^{1,2}(\Omega )$. But if $d=N-1$, then it
follows from \cite[Proposition 3.6]{BW} that
\begin{equation*}
W^{1,2}(\Omega )=\widetilde{W}^{1,2}(\Omega )=\widetilde{W}^{1,2}(\Omega
\backslash \Sigma ).
\end{equation*}
}
\end{remark}

Let
\begin{equation*}
{W}_{0,\Gamma _{D}}^{1,2}(\Omega \backslash \Sigma )=\overline{\{u\in {W}%
^{1,2}(\Omega \backslash \Sigma )\cap C(\overline{\Omega }):\;u=0\;%
\mbox{ on
}\Gamma _{D}\}}^{{W}^{1,2}(\Omega \backslash \Sigma )}.
\end{equation*}%
By definition, we have that ${W}_{0,\Gamma _{D}}^{1,2}(\Omega \backslash
\Sigma )$ is a closed subspace of $\widetilde{W}^{1,2}(\Omega \backslash
\Sigma )$.

For $r,q\in \lbrack 1,\infty ]$ with $1\leq r,q<\infty $ or $r=q=\infty $ we
endow the Banach space
\begin{equation*}
\mathbb{X}^{r,q}(\Omega ,\Sigma ):=L^{r}(\Omega )\times L^{q}(\Sigma ,\mu
_{\Sigma })=\{(f,g),\;f\in L^{r}(\Omega ),\;g\in L^{q}(\Sigma ,\mu _{\Sigma
})\}
\end{equation*}%
with the norm
\begin{equation*}
\Vert (f,g)\Vert _{\mathbb{X}^{r,q}(\Omega ,\Sigma )}:=\Vert f\Vert
_{r,\Omega }+\Vert g\Vert _{q,\Sigma }\;\mbox{ if }\;r\neq q,\;\Vert
(f,g)\Vert _{\mathbb{X}^{r,r}(\Omega ,\Sigma )}:=(\Vert f\Vert _{r,\Omega
}^{r}+\Vert g\Vert _{r,\Sigma }^{r})^{\frac{1}{r}}
\end{equation*}%
if $1\leq r,q<\infty $ and
\begin{equation*}
\Vert (f,g)\Vert _{\mathbb{X}^{\infty ,\infty }(\Omega ,\Sigma )}:=\max
\{\Vert f\Vert _{\infty ,\Omega },\Vert g\Vert _{\infty ,\Sigma }\}.
\end{equation*}%
We will simple write $\mathbb{X}^{r}(\Omega ,\Sigma ):=\mathbb{X}%
^{r,r}(\Omega ,\Sigma )$. If $\Omega \backslash \Sigma $ has the $\widetilde{%
W}^{1,2}$-extension property, then identifying each function $u\in
\widetilde{W}^{1,2}(\Omega \backslash \Sigma )$ with $(u,u|_{\Sigma })$
(recall that in this case, by Theorem \ref{Theo-Bie}, every $u\in \widetilde{%
W}^{1,2}(\Omega \backslash \Sigma )$ has a well-defined trace $u|_{\Sigma}$
which belongs to some $L^q(\Sigma,\mu_{\Sigma})$), we get from \eqref{sobo}
and \eqref{sob-meas} that if $N>2$, then
\begin{equation}
\widetilde{W}^{1,2}(\Omega \backslash \Sigma )\hookrightarrow \mathbb{X}%
^{r,q}(\Omega,\Sigma),\;\;\forall \;r\in \lbrack 1,2^{\star }],\;2^\star:=%
\frac{2N}{N-2},\;\forall \;q\in \lbrack 1,2_{\star }],\;\;2_{\star }:=\frac{%
2d}{N-2}.  \label{cont-em}
\end{equation}

\begin{remark}
\label{rem-29} \emph{Recall that by definition, $W_{0,\Gamma
_{D}}^{1,2}(\Omega \backslash \Sigma )$ is a closed subspace of $\widetilde{W%
}^{1,2}(\Omega \backslash \Sigma )$. Let }$Cap$\emph{$_{\Omega \setminus
\Sigma }$ be the relative capacity defined with on subsets of $\overline{%
\Omega \setminus \Sigma }=\overline{\Omega }$, with the regular (in these
sense of \cite[p.6]{Fuk}) Dirichlet space $\widetilde{W}^{1,2}(\Omega
\backslash \Sigma )$. More precisely, for a subset $A$ of $\bOm$,} \emph{let}
\begin{equation*}
Cap_{_{\Omega \setminus \Sigma }}\left( A\right) :=\inf \left\{ \Vert u\Vert
_{\widetilde{W}^{1,2}(\Omega \backslash \Sigma )}^{2}\;:\;%
\begin{array}{l}
u\in \widetilde{W}^{1,2}(\Omega \backslash \Sigma ),\;\exists \;O\subset {%
\mathbb{R}}^{N}\text{ open,} \\
A\subset O\text{ and }u\geq 1\text{ a.e. on }\overline{\Omega \setminus
\Sigma }\cap O%
\end{array}%
\right\} .
\end{equation*}%
\emph{With respect to the capacity }$Cap$\emph{$_{\Omega \setminus \Sigma }$%
, every function $u\in W_{0,\Gamma _{D}}^{1,2}(\Omega \backslash \Sigma )$
has a unique (relatively quasi-everywhere) relatively quasi-continuous
version $\tilde{u}$ on $\overline{\Omega }$. Throughout the rest of the
paper, if $u\in W_{0,\Gamma _{D}}^{1,2}(\Omega \backslash \Sigma )$, by $%
u|_{\Sigma }$, we mean $\tilde{u}|_{\Sigma }$. If $\Omega \backslash \Sigma $
has the $\widetilde{W}^{1,2}$-extension property then $u|_{\Sigma }=\tilde{u}%
|_{\Sigma }$ coincides with the trace of $u$ on $\Sigma ,$ which exists by
Theorem \ref{Theo-Bie} and belongs to $L^{2_{\star }}(\Sigma ,\mu _{\Sigma
}) $.}
\end{remark}

Next, let $0<s<1$ and define the bilinear symmetric form ${\mathcal{A}}%
_{\Theta ,{\Sigma }}$ on $\mathbb{X}^{2}(\Omega ,\Sigma )$ with domain
\begin{equation}
D({\mathcal{A}}_{\Theta ,{\Sigma }})=\{U:=(u,u|_{\Sigma }),\;u\in
W_{0,\Gamma _{D}}^{1,2}(\Omega \backslash \Sigma ),\;u|_{\Sigma }\in \mathbb{%
B}_{d,s}^{2}(\Sigma ,\mu _{\Sigma })\}  \label{form-def-dy}
\end{equation}%
and given for $U:=(u,u|_{\Sigma }),\Phi :=(\varphi ,\varphi |_{\Sigma })\in
D({\mathcal{A}}_{\Theta ,{\Sigma }})$ by%
\begin{align}
{\mathcal{A}}_{\Theta ,{\Sigma }}(U,\Phi )=& \int_{\Omega \backslash \Sigma }%
\mathbf{D}\nabla u\cdot \nabla \varphi dx+\int_{\Sigma }\beta \left(
x\right) u\varphi d\mu _{\Sigma }  \label{form-dy} \\
\ & +\int_{\Sigma }\int_{\Sigma }K(x,y)(u(x)-u(y))(\varphi (x)-\varphi
(y))d\mu _{\Sigma }\left( x\right) d\mu _{\Sigma }\left( y\right) ,  \notag
\end{align}%
where we recall $\mathbf{D}=(d_{ij}(x))$ is symmetric, bounded, measurable
and non-degenerate such that
\begin{equation}
\left\langle \mathbf{Dv,v}\right\rangle _{\mathbb{R}^{N}}\geq
d_{0}\left\vert \mathbf{v}\right\vert _{\mathbb{R}^{N}}^{2},\;%
\mbox{ for any
}\;\mathbf{v}\in {\mathbb{R}}^{N}\;\mbox{ with some
constant }\;d_{0}>0,  \label{D}
\end{equation}%
and the symmetric kernel $K$ is such that there exist two constants $%
0<c_{0}\leq c_{1}$ satisfying
\begin{equation*}
c_{0}\leq K(x,y)|x-y|^{d+2s}\leq c_{1},\;\;\forall \;x,y\in \Sigma ,\;x\neq
y.
\end{equation*}%
The function $\beta \in L^{\infty }(\Sigma ,\mu _{\Sigma })$ and there
exists a constant $\beta _{0}$ such that
\begin{equation}
\beta (x)\geq \beta _{0}>0\;\mbox{ for }\;\mu _{\Sigma }-\mbox{a.e. }\;x\in
\Sigma .  \label{beta}
\end{equation}%
We notice that $D({\mathcal{A}}_{\Theta ,{\Sigma }})$ is not empty, since it
contains the set $\{(u,u|_{\Sigma }):\;u\in C^{1}(\overline{\Omega }):\;u=0\;%
\mbox{ on }\;\Gamma _{D}\}$ and in particular it contains $\{(u,u|_{\Sigma
})=(u,0):\;u\in \mathcal{D}(\Omega \setminus \Sigma )\}$.

\begin{remark}
\emph{We mention that if $0<s<1-\frac{N-d}{2}$ and $\Omega \backslash \Sigma
$ has the $\widetilde{W}^{1,2}$-extension property, then it follows from
Theorem \ref{Theo-Bie}-(b) that%
\begin{equation*}
D(\mathcal{A}_{\Theta ,{\Sigma }})=\{U=(u,u|_{\Sigma }):\;u\in W_{0,\Gamma
_{D}}^{1,2}(\Omega \backslash \Sigma )\}.
\end{equation*}
}
\end{remark}

Throughout this section, in order to apply the abstract result on Dirichlet
forms given in Definition \ref{Diri-form}, we notice that $\mathbb{X}%
^{2}(\Omega ,\Sigma )$ can be identified with $L^{2}(\Omega ,\eta )$ where
the measure $\eta :=\mathcal{H}^{N}\oplus \mu _{\Sigma }$ is given for every
measurable set $B\subset \Omega $ by $\eta (B)=\mathcal{H}^{N}(B\cap (\Omega
\setminus \Sigma ))+\mu _{\Sigma }(B\cap \Sigma )$, so that for every $u\in
L^{2}(\Omega ,\eta )$, we have
\begin{equation*}
\int_{\Omega }ud\eta =\int_{\Omega \setminus \Sigma }udx+\int_{\Sigma }ud\mu
_{\Sigma }.
\end{equation*}

We have the following result.

\begin{proposition}
\label{prop-dform} Assume that $\Sigma $ is such that $\mu _{\Sigma }$ is
absolutely continuous with respect to $Cap_{\Omega \setminus \Sigma }$, that
is,
\begin{equation}
Cap_{\Omega \setminus \Sigma }(B)=0\;\Longrightarrow \;\mu _{\Sigma }(B)=0\;%
\mbox{ for all Borel set }\;B\subset \Sigma .  \label{cap}
\end{equation}%
The bilinear symmetric form ${\mathcal{A}}_{\Theta ,{\Sigma }}$ with domain $%
D({\mathcal{A}}_{\Theta ,{\Sigma }})$ is a Dirichlet form in the space $%
\mathbb{X}^{2}(\Omega ,\Sigma )$, that is, it is closed and Markovian.
\end{proposition}

\begin{proof}
Let ${\mathcal{A}}_{\Theta ,{\Sigma }}$ with domain $D({\mathcal{A}}_{\Theta
,{\Sigma }})$ be the bilinear symmetric form in $\mathbb{X}^{2}(\Omega
,\Sigma )$ defined in \eqref{form-dy}. First we show that the form ${%
\mathcal{A}}_{\Theta ,{\Sigma }}$ is closed in $\mathbb{X}^{2}(\Omega
,\Sigma )$. Indeed, let $U_{n}=(u_{n},u_{n}|_{\Sigma })\in D(\mathcal{A}%
_{\Theta ,{\Sigma }})$ be a sequence such that
\begin{equation}
\lim_{n,m\rightarrow \infty }\left(\mathcal{A}_{\Theta ,{\Sigma }%
}(U_{n}-U_{m},U_{n}-U_{m})+\Vert u_{n}-u_{m}\Vert _{L^{2}(\Omega
)}^{2}+\Vert u_{n}-u_{m}\Vert _{L^{2}(\Sigma ,\mu _{\Sigma })}^{2}\right)=0.
\label{conv-1}
\end{equation}%
It follows from \eqref{conv-1} that $\lim_{n,m\rightarrow \infty }\Vert
u_{n}-u_{m}\Vert _{W^{1,2}(\Omega \backslash \Sigma )}=0$. This implies that
$u_{n}$ converges strongly to some function $u\in \widetilde{W}^{1,2}(\Omega
\backslash \Sigma )$. Since $u_{n}\in W_{0,\Gamma _{D}}^{1,2}(\Omega
\backslash \Sigma )$ and $W_{0,\Gamma _{D}}^{1,2}(\Omega \backslash \Sigma )$
is a closed subspace of $\widetilde{W}^{1,2}(\Omega \backslash \Sigma )$ we
have that $u\in W_{0,\Gamma _{D}}^{1,2}(\Omega \backslash \Sigma )$.
Moreover taking a subsequence if necessary, we have that $u_n|_{\Sigma}$
converges relatively quasi-everywhere to $u|_{\Sigma}$ and hence by %
\eqref{cap}, $u_n|_{\Sigma}$ converges to $u|_{\Sigma}$, $\mu_{\Sigma}$-a.e.
on $\Sigma$. It also follows from \eqref{conv-1} that $u_{n}|_{\Sigma }$ is
a Cauchy sequence in the Banach space $\mathbb{B}_{d,s}^{2}(\Sigma ,\mu
_{\Sigma })$; hence, it converges in $\mathbb{B}_{d,s}^{2}(\Sigma ,\mu
_{\Sigma })$ to some function $v$ and also $\mu_{\Sigma}$-a.e. on $\Sigma$
(after taking a subsequence if necessary). By uniqueness of the limit, we
have that $v=u|_{\Sigma }\in \mathbb{B}_{d,s}^{2}(\Sigma ,\mu _{\Sigma })$.
Let $U=(u,u|_{\Sigma })$. We have shown that
\begin{equation*}
\lim_{n\rightarrow \infty }\mathcal{A}_{\Theta ,{\Sigma }}(U_{n}-U,U_{n}-U)+%
\Vert U_{n}-U\Vert _{\mathbb{X}^{2}(\Omega ,\Sigma )}^{2}=0
\end{equation*}%
and this implies that the form $\mathcal{A}_{\Theta ,{\Sigma }}$ is closed
in $\mathbb{X}^{2}(\Omega ,\Sigma )$.

Next, we show that the form $\mathcal{A}_{\Theta ,{\Sigma }}$ is Markovian.
Indeed, let $\varepsilon >0$ and $\phi _{\varepsilon }\in C^{\infty }({%
\mathbb{R}})$ be such that
\begin{equation}
\begin{cases}
\phi _{\varepsilon }(t)=t,\;\;\forall \;t\in \lbrack 0,1],\;-\varepsilon
\leq \phi _{\varepsilon }(t)\leq 1+\varepsilon ,\;\;\forall \;t\in {\mathbb{R%
}}, \\
0\leq \phi _{\varepsilon }(t_{1})-\phi _{\varepsilon }(t_{2})\leq
t_{1}-t_{2}\;\mbox{ whenver }\;t_{2}<t_{1}.%
\end{cases}
\label{func-phi1}
\end{equation}%
An example of such a function $\phi _{\varepsilon }$ is contained in \cite[%
Exercise 1.2.1, pg. 8]{Fuk}. We notice that it follows from \eqref{func-phi1}
that
\begin{equation}
0\leq \phi _{\varepsilon }^{\prime }(t)\leq 1,\;\;|\phi _{\varepsilon
}(t_{1})-\phi _{\varepsilon }(t_{2})|\leq |t_{1}-t_{2}|\;\mbox{ and }\;|\phi
_{\varepsilon }(t)|\leq |t|.  \label{func-phi2}
\end{equation}%
Let $U:=(u,u|_{\Sigma })\in D(\mathcal{A}_{\Theta ,{\Sigma }})$. It follows
from the first and third inequalities in \eqref{func-phi2} that $\phi
_{\varepsilon }(u)\in W_{0,\Gamma _{D}}^{1,2}(\Omega \backslash \Sigma ) $
and
\begin{equation}
\int_{\Omega \backslash \Sigma }|\mathbf{D}\nabla \phi _{\varepsilon
}(u)|^{2}dx=\int_{\Omega \backslash \Sigma }|\phi _{\varepsilon }^{\prime
}(u(x))|^{2}|\mathbf{D}\nabla u|^{2}dx\leq \int_{\Omega \backslash \Sigma }|%
\mathbf{D}\nabla u|^{2}dx.  \label{A1}
\end{equation}%
The second and the third inequalities in \eqref{func-phi2} imply that $\phi
_{\varepsilon }(u|_{\Sigma })=\phi _{\varepsilon }(u)|_{\Sigma }\in \mathbb{B%
}_{d,s}^{2}(\Sigma ,\mu _{\Sigma })$ and%
\begin{align}
& \int_{\Sigma }\int_{\Sigma }K(x,y)|\phi _{\varepsilon }(u(x))-\phi
_{\varepsilon }(u(y))|^{2}d\mu _{\Sigma }(x)d\mu _{\Sigma }(y)+\int_{\Sigma
}\beta (x)|\phi _{\varepsilon }\left( u\right) |^{2}d\mu _{\Sigma }
\label{A2} \\
& \leq \int_{\Sigma }\int_{\Sigma }K(x,y)|u(x)-u(y)|^{2}d\mu _{\Sigma
}(x)d\mu _{\Sigma }(y)+\int_{\Sigma }\beta (x)|u|^{2}d\mu _{\Sigma }.  \notag
\end{align}%
We have shown that $\Phi _{\varepsilon }(U):=(\phi _{\varepsilon }(u),\phi
_{\varepsilon }(u)|_{\Sigma })\in D(\mathcal{A}_{\Theta ,{\Sigma }})$.
Moreover the estimates \eqref{A1} and \eqref{A2} imply that
\begin{equation*}
\mathcal{A}_{\Theta ,{\Sigma }}(\Phi _{\varepsilon }(U),\Phi _{\varepsilon
}(U))\leq \mathcal{A}_{\Theta ,{\Sigma }}(U,U).
\end{equation*}%
Hence, by Definition \ref{Diri-form}-(d),\ the form $\mathcal{A}_{\Theta ,{%
\Sigma }}$ is Markovian on $\mathbb{X}^{2}(\Omega ,\Sigma )$. We have shown
that $\mathcal{A}_{\Theta ,{\Sigma }}$ with domain $D(\mathcal{A}_{\Theta ,{%
\Sigma }})$ is a Dirichlet form on $\mathbb{X}^{2}(\Omega ,\Sigma )$. The
proof is finished.
\end{proof}

Let $A_{\Theta ,{\Sigma }}$ be the closed linear self-adjoint operator on $%
\mathbb{X}^{2}(\Omega,\Sigma)$ associated with the form ${\mathcal{A}}%
_{\Theta ,{\Sigma }}$, in the sense that,
\begin{equation}  \label{op-form-dy}
\begin{cases}
D(A_{\Theta ,{\Sigma }}):=\{U\in D({\mathcal{A}}_{\Theta ,{\Sigma }%
}),\;\exists\; W\in \mathbb{X}^{2}(\Omega,\Sigma),\; {\mathcal{A}}_{\Theta ,{%
\Sigma }}(U,\Phi)=(W,\Phi)_{\mathbb{X}^{2}(\Omega,\Sigma)}\;\forall\;\Phi\in
D({\mathcal{A}}_{\Theta ,{\Sigma }})\}, \\
A_{\Theta ,{\Sigma }} U=W.%
\end{cases}%
\end{equation}

We also have the following characterization of the operator $A_{\Theta ,{%
\Sigma }}$.

\begin{proposition}
\label{op-dyn}Assume (\ref{cap}) and let ${A}_{\Theta ,{\Sigma }}$ be the
operator defined in \eqref{op-form-dy}. Then
\begin{equation}
\begin{cases}
D(A_{\Theta ,{\Sigma }})=\{U=(u,u|_{\Sigma }):\;u\in W_{0,\Gamma
_{D}}^{1,2}(\Omega \backslash \Sigma ),\;u|_{\Sigma }\in \mathbb{B}%
_{d,s}^{2}(\Sigma ,\mu _{\Sigma }),\;\;\mbox{div}(\mathbf{D}\nabla u)\in
L^{2}(\Omega ), \\
\hfill \partial _{\nu }^{\mathbf{D}}u=0\;\mbox{ on }\;\Gamma _{N}\;%
\mbox{
and }\;B_{\Theta }(u|_{\Sigma })\in L^{2}(\Sigma ,\mu _{\Sigma })\} \\
A_{\Theta ,{\Sigma }}U=\bigg(-\text{div}(\mathbf{D}\nabla u),{B}_{\Theta
}(u|_{\Sigma })\bigg),%
\end{cases}%
\end{equation}%
where
\begin{equation*}
{B}_{\Theta }(u|_{\Sigma }):=\frac{dN_{\mathbf{D}}(u)}{d\mu _{\Sigma }}%
+\beta (x)u+\Theta _{\Sigma }(u),
\end{equation*}%
and $dN_{\mathbf{D}}(u)$ is to be understood in the sense of \eqref{NE}.
\end{proposition}

\begin{proof}
Let ${A}_{\Theta ,{\Sigma }}$ be the closed linear self-adjoint operator on $%
\mathbb{X}^{2}(\Omega ,\Sigma )$ defined in \eqref{op-form-dy}. Set
\begin{align*}
{D}& :=\left\{ U=(u,u|_{\Sigma }):\;u\in W_{0,\Gamma _{D}}^{1,2}(\Omega
\backslash \Sigma ),\;u|_{\Sigma }\in \mathbb{B}_{d,s}^{2}(\Sigma ,\mu
_{\Sigma }),\;\;\text{div}(\mathbf{D}\nabla u)\in L^{2}(\Omega )\right. , \\
& \qquad \qquad \qquad \qquad \partial _{\nu }^{\mathbf{D}}u=0\;\mbox{ on }%
\;\Gamma _{N}\;\mbox{and }\;B_{\Theta }(u|_{\Sigma })\in L^{2}(\Sigma ,\mu
_{\Sigma })\}
\end{align*}%
and let $D({A}_{\Theta ,{\Sigma }})$ be given by \eqref{op-form-dy}. Let $%
U=(u,u|_{\Sigma })\in D(A_{\Theta ,{\Sigma }})$. Then by definition, there
exists $W=(w_{1},w_{2})\in \mathbb{X}^{2}(\Omega ,\Sigma )$ such that for
every $\varphi \in C^{1}(\overline{\Omega })$ with $\varphi =0$ on $\Gamma
_{D}$, we have
\begin{align}
\int_{\Omega \setminus \Sigma }w_{1}\varphi dx+\int_{\Sigma }w_{2}\varphi
d\mu _{\Sigma }=& \int_{\Omega \backslash \Sigma }\mathbf{D}\nabla u\cdot
\nabla \varphi dx+\int_{\Sigma }\beta \left( x\right) u\varphi d\mu _{\Sigma
}  \label{op-A-dy} \\
& +\int \int_{\Sigma \times \Sigma }K(x,y)(u(x)-u(y))(\varphi (x)-\varphi
(y))d\mu _{\Sigma }\left( x\right) d\mu _{\Sigma }\left( y\right)  \notag \\
=& \int_{\Omega \backslash \Sigma }\mathbf{D}\nabla u\cdot \nabla \varphi
dx+\int_{\Sigma }\beta \left( x\right) u\varphi d\mu _{\Sigma }+\int_{\Sigma
}\Theta _{\Sigma }(u)\varphi d\mu _{\Sigma }.  \notag
\end{align}%
In particular we get from \eqref{op-A-dy} that for every $\varphi \in
\mathcal{D}(\Omega \backslash \Sigma )$, we have%
\begin{equation}
\int_{\Omega \backslash \Sigma }w_{1}\varphi dx=\int_{\Omega \backslash
\Sigma }\mathbf{D}\nabla u\cdot \nabla \varphi dx.  \label{op-A2-dy}
\end{equation}%
It follows from \eqref{op-A2-dy} that
\begin{equation}
-\mbox{div}(\mathbf{D}\nabla u)=w_{1}\;\;\mbox{ in }\;\mathcal{D}(\Omega
\backslash \Sigma )^{\star }.  \label{distri}
\end{equation}%
Since $w_{1}\in L^{2}(\Omega \backslash \Sigma )$, we have that $\mbox{div}(%
\mathbf{D}\nabla u)\in L^{2}(\Omega \backslash \Sigma )$. Using %
\eqref{op-A-dy}, \eqref{op-A2-dy}, \eqref{distri} and \eqref{NE}, we get
that $\partial _{\nu }^{\mathbf{D}}u=0\;\mbox{ on }\;\Gamma _{N}$ (in the
distributional sense) and
\begin{equation*}
dN_{\mathbf{D}}(u)=\left( w_{2}-\beta (x)u-\Theta _{\Sigma }(u)\right) d\mu
_{\Sigma },
\end{equation*}%
so that
\begin{equation*}
{B}_{\Theta }(u|_{\Sigma }):=\frac{dN_{\mathbf{D}}(u)}{d\mu _{\Sigma }}%
+\beta (x)u+\Theta _{\Sigma }(u)=w_{2}\;\;\;\mbox{ on }\;\Sigma ,
\end{equation*}%
in the sense that for every $\varphi \in C^{1}(\overline{\Omega })$, $%
\varphi =0$ on $\Gamma _{D}$,
\begin{equation}
\int_{\Sigma }{B}_{\Theta }(u|_{\Sigma })\varphi d\mu _{\Sigma
}=\int_{\Sigma }w_{2}\varphi d\mu _{\Sigma }.  \label{form-equa}
\end{equation}%
Since $w_{2}\in L^{2}(\Sigma ,\mu _{\Sigma })$, it follows from %
\eqref{form-equa} that ${B}_{\Theta }(u|_{\Sigma })\in L^{2}(\Sigma ,\mu
_{\Sigma })$. Hence, $U\in {D}$ and we have shown that $D({A}_{\Theta ,{%
\Sigma }})\subset {D}$.

Conversely, let $U\in {D}$ and set $w_{1}:=-\text{div}(\mathbf{D}\nabla u)$
and $w_{2}={B}_{\Theta }(u|_{\Sigma }):=\frac{dN_{\mathbf{D}}(u)}{d\mu
_{\Sigma }}+\beta (x)u+\Theta _{\Sigma }(u)$. Then by hypothesis, $%
W=(w_{1},w_{2})\in \mathbb{X}^{2}(\Omega ,\Sigma )$. Moreover,
\begin{equation*}
w_{2}d\mu _{\Sigma }=dN_{\mathbf{D}}(u)+\beta (x)ud\mu _{\Sigma }+\Theta
_{\Sigma }(u)d\mu _{\Sigma },
\end{equation*}%
in the sense that for every $\varphi \in C^{1}(\overline{\Omega })$, $%
\varphi =0$ on $\Gamma _{D}$,
\begin{align}
\int_{\Sigma }\varphi w_{2}d\mu _{\Sigma }=& \int_{\Sigma }\varphi dN_{%
\mathbf{D}}(u)+\int_{\Sigma }\beta (x)\varphi ud\mu _{\Sigma }  \label{eq-mu}
\\
& +\int_{\Sigma }\int_{\Sigma }K(x,y)(\varphi (x)-\varphi
(y))(u(x)-u(y))\;d\mu _{\Sigma }(x)d\mu _{\Sigma }(y).  \notag
\end{align}

Let $\varphi \in C^{1}(\overline{\Omega })$ with $\varphi =0$ on $\Gamma
_{D} $ and set $\Phi :=(\varphi ,\varphi |_{\Sigma })$. Integrating by parts
(in the sense of the generalized Green type identity \eqref{NE}), and using %
\eqref{eq-mu} we infer
\begin{align*}
\int_{\Omega \setminus \Sigma }w_{1}\varphi dx+\int_{\Sigma }w_{2}\varphi
d\mu _{\Sigma }=& -\int_{\Omega \backslash \Sigma }\text{div}(\mathbf{D}%
\nabla u)\varphi dx+\int_{\Sigma }w_{2}\varphi d\mu _{\Sigma } \\
=& \int_{\Omega \backslash \Sigma }\mathbf{D}\nabla u\cdot \nabla \varphi
dx-\int_{\Gamma _{N}}\partial _{\nu }^{\mathbf{D}}u\varphi d\sigma
-\int_{\Sigma }\varphi dN_{\mathbf{D}}(u)+\int_{\Sigma }w_{2}\varphi d\mu
_{\Sigma } \\
=& \int_{\Omega \backslash \Sigma }\mathbf{D}\nabla u\cdot \nabla \varphi
dx-\int_{\Sigma }\varphi dN_{\mathbf{D}}(u)+\int_{\Sigma }\varphi dN_{%
\mathbf{D}}(u)+\int_{\Sigma }\beta (x)u\varphi d\mu _{\Sigma } \\
& +\int_{\Sigma }\int_{\Sigma }K(x,y)(u(x)-u(y))(\varphi (x)-\varphi
(y))d\mu _{\Sigma }(x)d\mu _{\Sigma }(y) \\
=& \int_{\Omega \backslash \Sigma }\mathbf{D}\nabla u\cdot \nabla \varphi
dx+\int_{\Sigma }\beta (x)u\varphi d\mu _{\Sigma } \\
& +\int_{\Sigma }\int_{\Sigma }K(x,y)(u(x)-u(y))(\varphi (x)-\varphi
(y))d\mu _{\Sigma }(x)d\mu _{\Sigma }(y) \\
=& {\mathcal{A}}_{\Theta ,{\Sigma }}(U,\Phi ).
\end{align*}%
We have shown that ${D}\subset D({A}_{\Theta ,{\Sigma }})$ and the proof is
finished.
\end{proof}

\begin{remark}
\emph{Note that if $\Omega \backslash \Sigma $ has the $\widetilde{W}^{1,2}$%
-extension property then $\Sigma $ satisfies \eqref{cap} (see \cite{BW,Wa}).
For more details on this subject we also refer the reader to \cite%
{AW1,AW2,BW,Fuk,Wa} and their references. In order to keep the exposition of
our main results in the subsequent sections more simple, we shall always
assume that $\Omega \backslash \Sigma $ satisfies the $\widetilde{W}^{1,2}$%
-extension property in Sections \ref{strong-sol}, \ref{gl}, \ref{bl}.}
\end{remark}

We have the following result of generation of semigroup.

\begin{theorem}
\label{theo-sg-dy} Let $A_{\Theta ,{\Sigma }}$ be the operator defined in %
\eqref{op-form-dy}. Then the following assertions hold.

\begin{enumerate}
\item Assume (\ref{cap}). The operator $-A_{\Theta ,{\Sigma }}$ generates a
Markovian semigroup $(e^{-tA_{\Theta ,{\Sigma }}})_{t\geq 0}$ on $\mathbb{X}%
^{2}(\Omega ,\Sigma )$. The semigroup can be extended to contraction
semigroups on $\mathbb{X}^{p}(\Omega ,\Sigma )$ for every $p\in \lbrack
1,\infty ]$, and each semigroup is strongly continuous if $p\in \lbrack
1,\infty )$ and bounded analytic if $p\in (1,\infty )$.

\item Assume that $\Omega \backslash \Sigma $ has the $\widetilde{W}^{1,2}$%
-extension property. Then the semigroup $(e^{-tA_{\Theta ,{\Sigma }%
}})_{t\geq 0}$ is ultracontractive in the sense that it maps $\mathbb{X}%
^{2}(\Omega ,\Sigma )$ into $\mathbb{X}^{\infty }(\Omega ,\Sigma )$ and each
semigroup on $\mathbb{X}^{p}(\Omega ,\Sigma )$ is compact for every $p\in
\lbrack 1,\infty ]$.

\item Assume that $\Omega \backslash \Sigma $ has the $\widetilde{W}^{1,2}$%
-extension property. Then the operator $A_{\Theta ,{\Sigma }}$ has a compact
resolvent, and hence has a discrete spectrum. The spectrum of $A_{\Theta ,{%
\Sigma }}$ is an increasing sequence of real numbers $0<\lambda _{1}\leq
\lambda _{2}\leq \cdots \leq \lambda _{n}\leq \dots ,$ that converges to $%
+\infty $. Moreover, if $U_{n}$ is an eigenfunction associated with $%
\lambda_{n}$, then $U_{n}\in D\left( A_{\Theta ,{\Sigma }}\right) \cap
\mathbb{X}^{\infty }(\Omega ,\Sigma )$.

\item Assume that $\Omega \backslash \Sigma $ has the $\widetilde{W}^{1,2}$%
-extension property. Then for each $\theta \in (0,1]$, the embedding $%
D(A_{\Theta ,\Sigma }^{\theta })\hookrightarrow \mathbb{X}^{\infty }\left(
\Omega ,\Sigma \right) $ is continuous provided that $\theta >\frac{\gamma}{4%
}$ with $\gamma =\frac{2d}{d-N+2}$.
\end{enumerate}
\end{theorem}

\begin{proof}
Let $A_{\Theta ,{\Sigma }}$ be the operator defined in \eqref{op-form-dy}.

(a) We have shown in Proposition \ref{prop-dform} that ${\mathcal{A}}%
_{\Theta ,{\Sigma }}$ is a Dirichlet form on $\mathbb{X}^{2}(\Omega ,\Sigma
) $. Hence, by \cite[Theorem 1.4.1]{Fuk} the operator $-{A}_{\Theta ,{\Sigma
}} $ generates a Markovian semigroup $(e^{-t{A}_{\Theta ,{\Sigma }}})_{t\geq
0}$ on $\mathbb{X}^{2}(\Omega ,\Sigma )$. It follows from \cite[Theorem 1.4.1%
]{Dav} that the semigroup can be extended to contraction semigroups on $%
\mathbb{X}^{p}(\Omega ,\Sigma )$ for every $p\in \lbrack 1,\infty ]$, and
each semigroup is strongly continuous if $p\in \lbrack 1,\infty )$ and
bounded analytic if $p\in (1,\infty )$.

(b) Now assume in the remainder of the proof that $\Omega \backslash \Sigma $
has the $\widetilde{W}^{1,2}$-extension property. Then exploiting %
\eqref{cont-em}, we ascertain the existence of a constant $C>0$ such that
for every $U=(u,u|_{\Sigma })\in D({\mathcal{A}}_{\Theta ,{\Sigma }})$,
\begin{equation}
\Vert |U\Vert |_{\mathbb{X}^{2_{\star }}(\Omega ,\Sigma )}^2\leq C{\mathcal{A%
}}_{\Theta ,{\Sigma }}(U,U)\;\mbox{ with }\;2_{\star }=\frac{2d}{N-2}=\frac{%
2\gamma }{\gamma -2},\;\;\gamma :=\frac{2d}{d-N+2}.  \label{sob-ul}
\end{equation}%
By \cite[Theorem 2.4.2]{Dav}, the estimate \eqref{sob-ul} implies that the
semigroup $(e^{-t{A}_{\Theta ,{\Sigma }}})_{t\geq 0}$ is ultracontractive.
More precisely, we have that there exists a constant $C>0$ such that for
every $F=(f_{1},f_{2})\in \mathbb{X}^{2}(\Omega ,\Sigma )$ and for every $%
t>0 $, we have
\begin{equation}
\Vert |e^{-t{A}_{\Theta ,{\Sigma }}}F\Vert |_{\mathbb{X}^{\infty }(\Omega
,\Sigma )}\leq Ct^{-\frac{\gamma }{4}}\Vert |F\Vert |_{\mathbb{X}^{2}(\Omega
,\Sigma )}.  \label{uultra}
\end{equation}%
The estimate \eqref{cont-em} also implies that the embedding $D({\mathcal{A}}%
_{\Theta ,{\Sigma }})\hookrightarrow \mathbb{X}^{2}(\Omega ,\Sigma )$ is
compact and this implies that the semigroup $(e^{-t{A}_{\Theta ,{\Sigma }%
}})_{t\geq 0}$ on $\mathbb{X}^{2}(\Omega ,\Sigma )$ is compact. Since $%
\Omega $ is bounded and $\mu_{\Sigma} (\Sigma )<\infty $, then the
compactness of the semigroup on $\mathbb{X}^{2}(\Omega ,\Sigma )$ together
with the ultracontractivity imply that the semigroup on $\mathbb{X}%
^{p}(\Omega ,\Sigma )$ is compact for every $p\in \lbrack 1,\infty ]$ (see,
e.g \cite[Theorem 1.6.4]{Dav}).

(c) The first part is an immediate consequence of (b) since $A_{\Theta ,{%
\Sigma }}$ is a\textit{\ positive} self-adjoint operator with compact
resolvent owing to the fact that $\beta \left( x\right) \geq \beta _{0}>0$
on $\Sigma $. Now let $U_{n}$ be an eigenfunction associated with $%
\lambda_{n}$. Then by definition, $U_{n}\in D\left( A_{\Theta ,{\Sigma }%
}\right)$. Since the semigroup $(e^{-t{A}_{\Theta ,{\Sigma }}})_{t\geq 0}$
is ultracontractive and $|\Omega|<\infty$, $\mu_{\Sigma}(\Sigma)<\infty$, it
follows from \cite[Theorem 2.1.4]{Dav} that $U_{n}\in\mathbb{X}^{\infty
}(\Omega ,\Sigma )$ and this completes the proof of this part.

(d) Since the operator $I+A_{\Theta ,\Sigma }$ is invertible we have that
the $\mathbb{X}^{2}\left( \Omega ,\Sigma \right) $-norm of $\left(
I+A_{\Theta ,\Sigma }\right) ^{\theta }$ defines an equivalent norm on $%
D(A_{\Theta ,\Sigma }^{\theta })$. Moreover for every $F\in \mathbb{X}%
^2(\Omega,\Sigma)$,
\begin{equation*}
\left( I+A_{\Theta ,\Sigma }\right) ^{-\theta }F=\frac{1}{\Gamma \left(
\theta \right) }\int_{0}^{\infty }t^{\theta -1}e^{-t}e^{-tA_{\Theta ,\Sigma
}}Fdt.
\end{equation*}%
Using (\ref{uultra}) for $t\in \left( 0,1\right) $ and the contractivity of $%
e^{-tA_{\Theta ,\Sigma }}$ for $t\ge 1$, for $u\in D(A_{\Theta ,\Sigma
}^{\theta })$, we deduce that there exists a constant $C>0$ such that
\begin{equation*}
\left\Vert u\right\Vert _{\mathbb{X}^{\infty }\left( \Omega ,\Sigma \right)
}\leq C\left\Vert u\right\Vert _{D(A_{\Theta ,\Sigma }^{\theta
})}\int_{0}^{1}t^{\theta -1-\frac{\gamma }{4}}dt+C\left\Vert u\right\Vert
_{D(A_{\Theta ,\Sigma }^{\theta })}\int_{1}^{\infty }e^{-t}dt.
\end{equation*}%
The first integral is finite if and only if $\gamma <4\theta $. This
completes the proof of the theorem.
\end{proof}

\begin{remark}
\emph{If $\theta \in (0,1]$ and $\Sigma$ is a Lipschitz hypersurface of
dimension $d=N-1$, hence, $\mu _{\Sigma }=\sigma _{\Sigma }$, then the
embedding $D(A_{\Theta ,\Sigma }^{\theta })\hookrightarrow \mathbb{X}%
^{\infty }\left( \Omega ,\Sigma \right) $ holds provided that $2\theta +1>N$.%
}
\end{remark}

\section{Well-posedness results}

\label{strong-sol}

We recall that the initial value problem associated with \eqref{p1}, %
\eqref{p2} and \eqref{m2} is the transmission problem%
\begin{equation}
\partial _{t}u-\text{div}\left( \mathbf{D}\nabla u\right) +f\left( u\right)
=0\text{, in }\,J\times (\Omega \backslash \Sigma ),  \label{p1b}
\end{equation}%
subject to the boundary conditions of the form%
\begin{equation}
u=0\text{ on }J\times \Gamma _{D},\;\;\;\text{ }(\mathbf{D}\nabla u)\cdot
\nu =0\text{ on }J\times \Gamma _{N},  \label{p2b}
\end{equation}%
the \emph{interfacial} boundary condition is given by
\begin{equation}
dN_{\mathbf{D}}(u)+\left( \partial _{t}u+\beta \left( x\right) u+\Theta _{{%
\Sigma }}\left( u\right) \right)d\mu_{\Sigma} =h\left(
u\right)d\mu_{\Sigma}, \text{ on }\,\;\;J\times \Sigma ,  \label{p3b}
\end{equation}%
and the initial conditions%
\begin{equation}
u\left( 0\right) =u_{0}\text{ in }\Omega \backslash \Sigma ,\text{ }u\left(
0\right) =v_{0}\text{ on }\Sigma ,  \label{p4b}
\end{equation}%
where $J=\left( 0,T\right)$, for some $T>0$ and some given functions $u_0$
and $v_{0}$. We emphasize that $v_{0}$ needs not necessarily be the trace of
$u_{0}$ to $\Sigma$, since $u_0$ will not be assumed to have a trace on $%
\Sigma$. But if $u_0$ has a well defined trace on $\Sigma$, then $v_0$ will
coincide with $u_0|_{\Sigma}$. %even when a trace operator exists.

In what follows we shall use classical (linear/nonlinear semigroup)
definitions of generalized solutions to \eqref{p1b}-\eqref{p4b}.
\textquotedblleft Generalized\textquotedblright\ solutions are defined via
nonlinear semigroup theory for bounded initial data and satisfy the
differential equations almost everywhere in $t>0$.

\begin{definition}
\label{solu} Let $\left( u_{0},v_{0}\right) \in \mathbb{X}^{\infty }(\Omega
,\Sigma )$. The function $u$ is said to be a \emph{strong} solution of %
\eqref{p1b}-\eqref{p4b} if, for a.e. $t\in \left( 0,T\right) ,$ for any $T>0$%
, the following properties are valid:

\begin{itemize}
\item Regularity:%
\begin{equation}
U=(u,u|_{\Sigma })\in W_{\text{loc}}^{1,\infty }((0,T];\mathbb{X}^{2}(\Omega
,\Sigma ))\cap C(\left[ 0,T\right] ;\mathbb{X}^{\infty }(\Omega ,\Sigma
))\cap C_{\text{loc}}((0,T];D(\mathcal{A}_{\Theta ,\Sigma }))
\label{reg_weak}
\end{equation}%
such that $U\left( t\right) \in D(A_{\Theta ,\Sigma }),$ a.e. $t\in \left(
0,T\right) ,$ for any $T>0.$

\item The following variational identity%
\begin{align}
& \int_{\Omega \backslash \Sigma }\partial _{t}u\left( t\right) \xi
dx+\int_{\Sigma }\partial _{t}u\left( t\right) \xi |_{\Sigma }d\mu _{\Sigma
}+\mathcal{A}_{\Theta ,\Sigma }(U\left( t\right) ,\xi )+\int_{\Omega
\backslash \Sigma }f\left( u\left( t\right) \right) \xi dx  \label{de_form}
\\
& =\int_{\Sigma }h\left( u\left( t\right) |_{\Sigma }\right) \xi |_{\Sigma
}d\mu _{\Sigma }  \notag
\end{align}%
holds for all $\xi =(\xi ,\xi |_{\Sigma })\in D({\mathcal{A}}_{\Theta ,{%
\Sigma }}),$ a.e. $t\in \left( 0,T\right) $.

\item We have $(u\left( t\right) ,u\left( t\right) |_{\Sigma })\rightarrow
\left( u_{0},v_{0}\right) $ strongly in $\mathbb{X}^{\infty}(\Omega ,\Sigma
) $ as $t\rightarrow 0^{+}$.
\end{itemize}
\end{definition}

Throughout the remainder of the article, we will always say that $%
U(t)=(u(t),u(t)|_{\Sigma})$ is a strong solution to the transmission problem %
\eqref{p1b}-\eqref{p4b}.

We will now recall some results for a non-homogeneous Cauchy problem%
\begin{equation}
\left\{
\begin{array}{ll}
u^{\prime }(t)+A\left( u\right) \ni \mathcal{G}\left( t\right) , & t\in
\left(0,T\right) , \\
u(0)=u_{0}\text{.} &
\end{array}%
\right.  \label{nonhom}
\end{equation}

\begin{theorem}
\label{m11}\cite[Chapter IV, Theorem 4.3]{Scho} Let $H$ be a Hilbert space, $%
\varphi :H\rightarrow (-\infty ,+\infty ]$ a proper, convex, and
lower-semicontinuous functional on $H$ and set $A:=\partial \varphi$, the
subdifferential of $\varphi$. Let $u$ be the generalized solution of %
\eqref{nonhom} with $\mathcal{G}\in L^{2}\left((0,T);H\right) $ and $%
u_{0}\in \overline{D\left( A\right) }.$ Then $\varphi \left( u\right) \in
L^{1}\left( 0,T\right) ,$ $\sqrt{t}u^{\prime }(t)\in
L^{2}\left((0,T);H\right) $ and $u\left( t\right) \in D\left( A\right) $ for
a.e. $t\in \left( 0,T\right).$
\end{theorem}

In Theorem \ref{m11} and below, by a generalized solution $u$ of %
\eqref{nonhom}, we mean a function $u\in C([0,T];H)$ for which there exists
a sequence of (absolutely continuous) solutions $u_n$ of
\begin{align*}
u_n^{\prime }(t)+A\left( u_n\right) \ni \mathcal{G}_n\left(
t\right),\;\;n\ge 1,
\end{align*}
with $\mathcal{G}_n\to \mathcal{G}$ in $L^1((0,T);H)$ and $u_n\to u$ in $%
C([0,T];H)$ as $n\to\infty$.

The second one is a more general version of \cite[Chapter IV, Proposition 3.2%
]{Scho} and was proved in our recent work \cite[Theorem 6.3 and Corollary 6.4%
]{GW}.

\begin{theorem}
\label{ap_reg_thm}Let the assumptions of Theorem \ref{m11} be satisfied.
Assume that $A$ is \emph{strongly} accretive in $H$, that is, $A-\omega I$
is accretive for some $\omega \ge 0$ and, in addition,%
\begin{equation}
\mathcal{G}\in L^{\infty }\left( (\tau ,\infty );H\right) \cap W^{1,2}\left(
(\tau ,\infty );H\right) ,  \label{ext}
\end{equation}%
for every $\tau >0$. Let $u$ be the unique generalized solution of \eqref
{nonhom} for $u_{0}\in \overline{D\left( A\right) }$. Then
\begin{equation}
u\in L^{\infty }\left( (\tau ,\infty );D\left( A\right) \right) \cap
W^{1,\infty }\left( (\tau ,\infty );H\right) .  \label{ext2}
\end{equation}
\end{theorem}

We need a Poincar\'{e}-type inequality in the space $\widetilde
W^{1,1}\left( \Omega \backslash \Sigma \right) .$

\begin{lemma}
\label{poincare-weak}Assume $\Omega \backslash \Sigma $ has the $\widetilde{W%
}^{1,1}$-extension property and that there exists a trace operator%
\begin{equation*}
T:\widetilde{W}^{1,1}\left( \Omega \backslash \Sigma \right) \rightarrow
L^{1}\left( \Sigma ,d\mu _{\Sigma }\right) ,\text{ }Tu=u|_{\Sigma }
\end{equation*}%
which is linear and bounded. Then there is a constant $C_{\Sigma ,\Omega
}=C\left( \mu _{\Sigma }\left( \Sigma \right) ,\left\vert \Omega \right\vert
\right) >0$ independent of $u$ such that
\begin{equation}  \label{poinc-1}
\left\Vert u-\frac{1}{\mu _{\Sigma }\left( \Sigma \right) }\int_{\Sigma
}ud\mu _{\Sigma }\right\Vert _{L^{1}\left( \Omega \backslash \Sigma \right)
}\leq C_{\Sigma ,\Omega }\left\Vert \nabla u\right\Vert _{L^{1}\left( \Omega
\backslash \Sigma \right) },
\end{equation}
for all $u\in \widetilde{W}^{1,1}\left( \Omega \backslash \Sigma \right) .$
\end{lemma}

\begin{proof}
We notice that since $\Omega \backslash \Sigma $ has the $\widetilde W^{1,1}$%
-extension property, we have that the classical Sobolev embedding yields%
\begin{equation}
\widetilde W^{1,1}\left( \Omega \backslash \Sigma \right) \hookrightarrow L^{%
\frac{N}{N-1}}(\Omega\setminus\Sigma)  \label{cont-1}
\end{equation}%
with continuous inclusion. To prove \eqref{poinc-1}, it suffices to show
that there exists a constant $C>0$ such that for every $u\in \widetilde
W^{1,1}(\Omega \backslash \Sigma )$ with $\int_{\Sigma }ud\mu _{\Sigma }=0$
and $\Vert u\Vert _{L^{1}(\Omega \setminus\Sigma)}=1 $, we have that $1\leq
C\Vert \nabla u\Vert _{L^{1}(\Omega\setminus\Sigma )}$. Indeed, assume to
the contrary that there exists a sequence $u_{n}\in \widetilde
W^{1,1}(\Omega \backslash \Sigma )$ such that
\begin{equation*}
\int_{\Sigma }u_{n}d\mu _{\Sigma }=0,\;\;\Vert u_{n}\Vert
_{L^{1}(\Omega\setminus\Sigma )}=1\;\;\mbox{ and }\;\Vert \nabla u_{n}\Vert
_{L^{1}(\Omega\setminus\Sigma )}\leq \frac{1}{n},\;\;\forall\,n\in\mathbb{N}.
\end{equation*}%
Then, $u_{n}$ is a bounded sequence in $\widetilde W^{1,1}(\Omega \backslash
\Sigma )$. Since the embedding $\widetilde W^{1,1}(\Omega \backslash \Sigma
)\hookrightarrow L^{1}(\Omega\setminus\Sigma )$ is compact (this follows
from \eqref{cont-1}), then taking a subsequence if necessary, we have that $%
u_{n}$ converges strongly to some function $u$ in $L^{1}(\Omega\setminus%
\Sigma )$. Moreover, for every $\varphi \in \mathcal{D}(\Omega \backslash
\Sigma )$ and $i=1,\ldots ,N$, we have that
\begin{equation*}
\int_{\Omega\setminus\Sigma }uD_{i}\varphi dx=\lim_{n\rightarrow \infty
}\int_{\Omega\setminus\Sigma }u_{n}D_{i}\varphi dx=\lim_{n\rightarrow \infty
}\int_{\Omega\setminus\Sigma }-\varphi D_{i}u_{n}\;dx=0.
\end{equation*}%
Therefore, $\int_{\Omega\setminus\Sigma }uD_{i}\varphi dx=0$ for all $%
\varphi \in \mathcal{D}(\Omega \backslash \Sigma )$ and $i=1,\ldots ,N$.
This implies that $\nabla u=0$ on $\Omega \backslash
\Sigma=\Omega_1\cup\Omega_2$. Hence, $\nabla u=0$ on $\Omega _{1}$ and $%
\nabla u=0$ on $\Omega _{2}$. Since $\Omega _{1}$ and $\Omega _{2}$ are
connected, we have that $u=C_{1}$ on $\Omega _{1}$ and $u=C_{2}$ on $\Omega
_{2}$ and $C_{1}|\Omega _{1}|+C_{2}|\Omega _{2}|=1$ (this last equality
follows from the fact that $\|u\|_{L^1(\Omega)}=\|u\|_{L^1(\Omega\setminus%
\Sigma)}=1$). Consequently, $u_{n}\rightarrow u$ strongly in $\widetilde
W^{1,1}(\Omega \backslash \Sigma )$ as $n\to\infty$. Since by assumption
there exists a trace operator $T:\widetilde W^{1,1}\left( \Omega \backslash
\Sigma \right) \rightarrow L^{1}\left( \Sigma ,d\mu _{\Sigma }\right) ,$ $%
Tu=u|_{\Sigma }$ which is linear and bounded, we have that $u_{n}\rightarrow
u$ strongly in $L^{1}(\Sigma ,\mu _{\Sigma })$ as $n\to\infty$, and the
uniqueness of the trace operator shows that $C_{1}=C_{2}=C$ on $\Sigma$.
Finally, we have that
\begin{equation*}
0=\lim_{n\to\infty}\int_{\Sigma }u_{n}d\mu _{\Sigma }= \int_{\Sigma }ud\mu
_{\Sigma }=C\mu _{\Sigma }(\Sigma )=\frac{\mu _{\Sigma }(\Sigma )}{|\Omega_1
|+|\Omega_2|}=\frac{\mu _{\Sigma }(\Sigma )}{|\Omega |}\neq 0.
\end{equation*}%
This is a contradiction and the proof is finished.
\end{proof}

We also need the following Poincar\'e type inequality.

\begin{lemma}
\label{poincare-forms}Assume that $\Omega \backslash \Sigma $ has the $%
\widetilde W^{1,2}$-extension property. Then, for every $\varepsilon \in
\left( 0,1\right) $ there exists $\zeta >0$ such that for all $U\in D({%
\mathcal{A}}_{\Theta ,{\Sigma }})$,
\begin{equation}
\left\Vert U\right\Vert _{\mathbb{X}^{2}\left( \Omega ,\Sigma \right)
}^{2}\leq \varepsilon {\mathcal{A}}_{\Theta ,{\Sigma }}\left( U,U\right)
+\varepsilon ^{-\zeta }\left\Vert U\right\Vert _{\mathbb{X}^{1}\left( \Omega
,\Sigma \right) }^{2}.  \label{PDBC}
\end{equation}
\end{lemma}

\begin{proof}
First, we observe that%
\begin{equation*}
\left( \mathcal{A}_{\Theta ,{\Sigma }}\left( U,U\right) \right)
^{1/2}+\left\Vert U\right\Vert _{\mathbb{X}^{1}\left( \Omega ,\Sigma \right)
}
\end{equation*}%
defines an equivalent norm on $D(\mathcal{A}_{\Theta ,{\Sigma }})\cap\mathbb{%
X}^1(\Omega,\Sigma)=D(\mathcal{A}_{\Theta ,{\Sigma }})$. Second, by dividing %
\eqref{PDBC} by $\left\Vert U\right\Vert _{\mathbb{X}^{2}\left( \Omega
,\Sigma \right) }^{2}$ if necessary, it suffices to prove \eqref{PDBC} for $%
\Vert U\Vert _{\mathbb{X}^{2}\left( \Omega ,\Sigma \right) }=1$. Suppose
that there is no $\zeta >0$ such that \eqref{PDBC} holds for a given $%
\varepsilon \in (0,1)$. Then for every $k\in \mathbb{N}$ there is a sequence
$U_{k}\in D(\mathcal{A}_{\Theta ,{\Sigma }})$ such that%
\begin{equation}  \label{INE}
\Vert U_{k}\Vert _{\mathbb{X}^{2}\left( \Omega ,\Sigma \right) }^{2}=1\geq
\varepsilon \mathcal{A}_{\Theta ,{\Sigma }}\left( U_{k},U_{k}\right)
+\varepsilon ^{-\zeta }\left\Vert U_{k}\right\Vert _{\mathbb{X}^{1}\left(
\Omega ,\Sigma \right) }^{2}.
\end{equation}%
The inequality \eqref{INE} implies that the resulting sequence $(U_{k})$ is
bounded in $D(\mathcal{A}_{\Theta ,{\Sigma }})$. Hence, after a subsequence
if necessary, we have that $(U_k)$ converges weakly to some $U\in D(\mathcal{%
A}_{\Theta ,{\Sigma }}$. Since the embeddings $D(\mathcal{A}_{\Theta ,{%
\Sigma }})\hookrightarrow\mathbb{X}^{2}(\Omega ,\Sigma )$ and $D(\mathcal{A}%
_{\Theta ,{\Sigma }})\hookrightarrow\mathbb{X}^{1}\left( \Omega ,\Sigma
\right) $ are compact, we find a subsequence, again denoted by $(U_{k})$,
that converges strongly in $\mathbb{X}^{2}(\Omega ,\Sigma )$ and in $\mathbb{%
X}^{1}(\Omega ,\Sigma )$ to the function $U\in D(\mathcal{A}_{\Theta ,{%
\Sigma }})$. By assumption we have $\Vert U\Vert _{\mathbb{X}^{2}\left(
\Omega ,\Sigma \right) }=1$. On the other hand, \eqref{INE} shows that $%
\Vert U_{k}\Vert _{\mathbb{X}^{1}\left( \Omega ,\Sigma \right) }^{2}\leq
\varepsilon ^{k}$ for all $k$. Therefore $\Vert U\Vert _{\mathbb{X}%
^{1}\left( \Omega ,\Sigma \right) }=0$ and thus $u=0$ a.e. in $%
\Omega\setminus\Sigma $ and $u|_{\Sigma }=0,$ $\mu _{\Sigma }$ a.e. on $%
\Sigma $. This is a contradiction which altogether completes the proof of
the lemma.
\end{proof}

First, we have the local existence result.

\begin{theorem}
\label{SG}Assume that $f$, $h\in C_{loc}^1({\mathbb{R}})$ and that $\Omega
\backslash \Sigma $ has the $\widetilde{W}^{1,2}$-extension property. Then
for every $\left( u_{0},v_{0}\right) \in \mathbb{X}^{\infty }(\Omega ,\Sigma
),$ there exists a unique strong solution $U$ of \eqref{p1b}-\eqref{p4b} on $%
\left( 0,T_{\max }\right) $ in the sense of Definition \ref{solu}, for some $%
T_{\max }>0$. Moreover, if $T_{\max }<\infty $, then%
\begin{equation*}
\lim_{t\uparrow T_{\max }}\Vert U(t)\Vert _{\mathbb{X}^{\infty }(\Omega
,\Sigma )}=\infty .
\end{equation*}
\end{theorem}

\begin{proof}
Let $\left( u_{0},v_{0}\right) \in \mathbb{X}^{\infty }(\Omega ,\Sigma
)\subset \mathbb{X}^{2}(\Omega ,\Sigma )=\overline{D({\mathcal{A}}_{\Theta
,\Sigma })}^{\mathbb{X}^{2}(\Omega ,\Sigma )}$. From Theorem \ref{theo-sg-dy}%
\ we know that $-{A}_{\Theta ,{\Sigma }}$ generates a submarkovian (linear)
semigroup $(e^{-t{A}_{\Theta ,{\Sigma }}})_{t\geq 0}$ on $\mathbb{X}%
^{2}(\Omega ,\Sigma )$. Hence, the operator $e^{-t{A}_{\Theta ,{\Sigma }}}$
is non-expansive on $\mathbb{X}^{\infty }(\Omega ,\Sigma )$, that is,
\begin{equation}
||e^{-tA_{\Theta ,{\Sigma }}}U_{0}||_{\mathbb{X}^{\infty }(\Omega ,\Sigma
)}\leq \left\Vert U_{0}\right\Vert _{\mathbb{X}^{\infty }(\Omega ,\Sigma
)},\;\;\forall \;\text{ }t\geq 0\text{ and }U_{0}=\left( u_{0},v_{0}\right)
\in \mathbb{X}^{\infty }(\Omega ,\Sigma ).  \label{nonexp}
\end{equation}%
In addition, we have that ${A}_{\Theta ,{\Sigma }}$ is strongly accretive on
$\mathbb{X}^{2}(\Omega ,\Sigma )$. That is, ${\mathcal{A}}_{\Theta ,{\Sigma }%
}(U,U)\geq C\left\Vert U\right\Vert _{\mathbb{X}^{2}(\Omega ,\Sigma )}^{2}$,
for some $C>0$ and for every $U\in D({\mathcal{A}}_{\Theta ,{\Sigma }})$,
where we have used \eqref{beta}. Thus, the operator version of problem %
\eqref{p1b}-\eqref{p4b} reads
\begin{equation}
\partial _{t}U=-{A}_{\Theta ,{\Sigma }}U-F\left( U\right) ,\text{ }U\overset{%
\text{def}}{=}\left( u,u|_{\Sigma }\right) ,\text{ }U(0)=\left(
u_{0},v_{0}\right) ,  \label{op_v}
\end{equation}%
where we have set%
\begin{equation*}
F\left( U\right) =\left( f\left( u\right) ,-h\left( u\right) |_{\Sigma
}\right) .
\end{equation*}%
We construct the (locally-defined) strong solution by a fixed point
argument. To this end, fix $0<T^{\ast }\leq T$, $\Vert U_{0}\Vert _{\mathbb{X%
}^{\infty }(\Omega ,\Sigma )}\leq R^{\star }$, consider the space
\begin{equation*}
\mathcal{X}_{T^{\ast },R^{\ast }}\equiv \left\{ V\in C\left( \left[
0,T^{\ast }\right] ;\mathbb{X}^{\infty }(\Omega ,\Sigma )\right) :\left\Vert
V\left( t\right) \right\Vert _{\mathbb{X}^{\infty }(\Omega ,\Sigma )}\leq
R^{\ast },\;\;V(0)=U_{0}:=(u_{0},v_{0})\right\}
\end{equation*}%
and define the following mapping%
\begin{equation}
S\left( V\right) \left( t\right) =e^{-t{A}_{\Theta ,{\Sigma }%
}}U_{0}-\int_{0}^{t}e^{-\left( t-s\right) {A}_{\Theta ,{\Sigma }}}F\left(
V\left( s\right) \right) ds,\text{ }t\in \left[ 0,T^{\ast }\right] .
\label{mapping}
\end{equation}%
We mention that the space $\mathcal{X}_{T^{\ast },R^{\ast }}$ is not empty,
since it contains at least the function $U_{0}$. We notice that $\mathcal{X}%
_{T^{\ast },R^{\ast }}$, when endowed with the norm of $C\left( \left[
0,T^{\ast }\right] ;\mathbb{X}^{\infty }(\Omega ,\Sigma )\right) $, is a
closed subset of the space $C\left( \left[ 0,T^{\ast }\right] ;\mathbb{X}%
^{\infty }\left( \Omega ,\Sigma \right) \right) $, and since $f,h$ are
locally Lipschitz we have that $S\left( V\right) \left( t\right) $ is
continuous on $\left[ 0,T^{\ast }\right] $. We will show that, by properly
choosing $T^{\ast },R^{\ast }>0$, we get that $S:\mathcal{X}_{T^{\ast
},R^{\ast }}\rightarrow \mathcal{X}_{T^{\ast },R^{\ast }}$ is a contraction
mapping with respect to the metric induced by the norm of $C\left( \left[
0,T^{\ast }\right] ;\mathbb{X}^{\infty }(\Omega ,\Sigma )\right) .$ The
appropriate choices for $T^{\ast },R^{\ast }>0$ will be specified below.
First, we show that if $V\in \mathcal{X}_{T^{\ast },R^{\ast }}$ then $%
S\left( V\right) \in \mathcal{X}_{T^{\ast },R^{\ast }}$, that is, $S$ maps $%
\mathcal{X}_{T^{\ast },R^{\ast }}$ to itself. From (\ref{nonexp}), the fact
that $F\in C_{\text{loc}}^{1}\left( \mathbb{R},\mathbb{R}^{2}\right) $
together with the fact that $F(x)-F(y)=F^{\prime }(\xi )(x-y)$ for some $\xi
$ on the line segment from $x$ to $y$ (by the mean value theorem), we
observe that the mapping $S$ satisfies the following estimate%
\begin{align*}
\left\Vert S\left( V\left( t\right) \right) \right\Vert _{\mathbb{X}^{\infty
}(\Omega ,\Sigma )}& \leq \left\Vert U_{0}\right\Vert _{\mathbb{X}^{\infty
}(\Omega ,\Sigma )}+\int_{0}^{t}\left\Vert e^{-\left( t-\tau \right) {A}%
_{\Theta ,\Sigma }}\left( F\left( 0\right) +\left( F\left( V\left( \tau
\right) \right) -F\left( 0\right) \right) \right) \right\Vert _{\mathbb{X}%
^{\infty }(\Omega ,\Sigma )}d\tau \\
& \leq \left\Vert U_{0}\right\Vert _{\mathbb{X}^{\infty }(\Omega ,\Sigma
)}+t\left( \left\vert F\left( 0\right) \right\vert +Q_{f,h}\left( R^{\ast
}\right) R^{\ast }\right) ,
\end{align*}%
for some positive continuous function $Q_{f,h}$ which depends only on the
size of the nonlinearities $f^{^{\prime }},h^{^{\prime }}$ and by $|F(0)|$
we mean $|F(0)|=|f(0)|+|h(0)|$. Thus, provided that we set $R^{\ast }\geq
2\left\Vert U_{0}\right\Vert _{\mathbb{X}^{\infty }(\Omega ,\Sigma )}$, we
can find a sufficiently small time $T^{\ast }>0$ such that
\begin{equation}
2T^{\ast }\left( \left\vert F\left( 0\right) \right\vert +Q_{f,h}\left(
R^{\ast }\right) R^{\ast }\right) \leq R^{\ast }.  \label{3.21time}
\end{equation}%
Since $S(V)(0)=U_{0}$ (by \eqref{mapping}), we have shown that $S\left(
V\right) \in \mathcal{X}_{T^{\ast },R^{\ast }}$, for any $V\in \mathcal{X}%
_{T^{\ast },R^{\ast }}.$ Next, we show that by possibly choosing $T^{\ast
}>0 $ smaller, $S:\mathcal{X}_{T^{\ast },R^{\ast }}\rightarrow \mathcal{X}%
_{T^{\ast },R^{\ast }}$ is also a contraction. Indeed, for any $%
V_{1},V_{2}\in \mathcal{X}_{T^{\ast },R^{\ast }}$, exploiting again (\ref%
{nonexp}), we estimate%
\begin{align}
\left\Vert S\left( V_{1}\left( t\right) \right) -S\left( V_{2}\left(
t\right) \right) \right\Vert _{\mathbb{X}^{\infty }(\Omega ,\Sigma )}& \leq
Q_{f,h}\left( R^{\ast }\right) \int_{0}^{t}\left\Vert e^{-\left( t-\tau
\right) {A}_{\Theta ,\Sigma }}(V_{1}\left( \tau \right) -V_{2})\left( \tau
\right) \right\Vert _{\mathbb{X}^{\infty }(\Omega ,\Sigma )}d\tau
\label{3.21} \\
& \leq tQ_{f,h}\left( R^{\ast }\right) \left\Vert V_{1}-V_{2}\right\Vert
_{C\left( \left[ 0,T^{\ast }\right] ;\mathbb{X}^{\infty }(\Omega ,\Sigma
)\right) }.  \notag
\end{align}%
This shows that $S$ is a contraction on $\mathcal{X}_{T^{\ast },R^{\ast }}$
(compare with \eqref{3.21time}) provided that we choose a time $T^{\ast }>0$
which satisfies (\ref{3.21time}) and $T^{\ast }Q_{f,h}\left( R^{\ast
}\right) <1$. Therefore, owing to the contraction mapping principle, we
conclude that problem (\ref{op_v}) has a unique local solution $U=\left(
u,u|_{\Sigma }\right) \in \mathcal{X}_{T^{\ast },R^{\ast }}$. Using
semigroup properties, we get that this solution can certainly be (uniquely)
extended on a right maximal time interval $[0,T_{\max })$, with $T_{\max }>0$
depending on $\left\Vert U_{0}\right\Vert _{\mathbb{X}^{\infty }(\Omega
,\Sigma )},$ such that, either $T_{\max }=\infty $ or $T_{\max }<\infty $,
in which case $\lim_{t\uparrow T_{\max }}\left\Vert U\left( t\right)
\right\Vert _{\mathbb{X}^{\infty }(\Omega ,\Sigma )}=\infty .$ Indeed, if $%
T_{\max }<\infty $ and the latter condition does not hold, we can find a
sequence $t_{n}\uparrow T_{\max }$ as $n\rightarrow \infty $ such that $%
\left\Vert U\left( t_{n}\right) \right\Vert _{\mathbb{X}^{\infty }(\Omega
,\Sigma )}\leq C$ for all $n\in \mathbb{N}$. This would allow us to extend $%
U $ as a solution to equation (\ref{op_v}) to an interval $[0,t_{n}+\delta ),
$ for some $\delta >0$ independent of $n$. Hence $U$ can be extended beyond $%
T_{\max }$ which contradicts the construction of $T_{\max }>0$. To conclude
that the solution $U$ belongs to the class in Definition \ref{solu}, let us
further set $\mathcal{G}\left( t\right) :=-F\left( U\left( t\right) \right)
, $ for $U\in C\left( [0,T_{\max }\right) ;\mathbb{X}^{\infty }(\Omega
,\Sigma ))$ and notice that $U$ is the "generalized" solution of%
\begin{equation}
\partial _{t}U+{A}_{\Theta ,\Sigma }U=\mathcal{G}\left( t\right) ,\text{ }%
t\in \lbrack 0,T_{\max }),  \label{3.44}
\end{equation}%
such that $U\left( 0\right) =U_{0}\in \mathbb{X}^{\infty }(\Omega ,\Sigma
)\subset \mathbb{X}^{2}(\Omega ,\Sigma )=\overline{D({A}_{\Theta ,\Sigma })}%
. $ By Theorem \ref{m11}, the "generalized" solution $U$ has the additional
regularity $\partial _{t}U\in L^{2}\left( (\tau ,T_{\max });\mathbb{X}%
^{2}(\Omega ,\Sigma )\right) ,$ and since $U$ is continuous on $[0,T_{\max
}) $ with values in $\mathbb{X}^{\infty }(\Omega ,\Sigma )$ and $f,h\in C_{%
\text{loc}}^{1}\left( \mathbb{R}\right) $, there readily holds
\begin{equation}
\mathcal{G}\in W^{1,2}\left( (\tau ,T_{\max });\mathbb{X}^{2}(\Omega ,\Sigma
)\right) \cap L^{\infty }\left( (\tau ,T_{\max });\mathbb{X}^{\infty
}(\Omega ,\Sigma )\right) ,  \label{3.45}
\end{equation}%
owing to the fact that $\partial _{t}\mathcal{G}=(f^{^{\prime }}\left(
u\right) \partial _{t}u,-(h^{^{\prime }}\left( u\right) \partial
_{t}u)|_{\Sigma })$ a.e. on $[\tau ,T_{\max })$. Thus, we can apply Theorem %
\ref{ap_reg_thm} to deduce%
\begin{equation}
U\in L^{\infty }((\tau ,T_{\max });D({A}_{\Theta ,\Sigma }))\cap W^{1,\infty
}\left( (\tau ,T_{\max });\mathbb{X}^{2}(\Omega ,\Sigma )\right) ,
\label{3.46}
\end{equation}%
such that the solution $U$ is Lipschitz continuous on $[\tau ,T_{\max })$,
for every $\tau >0.$ Thus, we have obtained a locally-defined strong
solution in the sense of Definition \ref{solu}. Multiplying \eqref{p1b} by a
test function $\xi =(\xi ,\xi |_{\Sigma })\in D(\mathcal{A}_{\Theta ,\Sigma
})$, using \eqref{p3b} and Proposition \ref{op-dyn} we get the variational
equality in \eqref{de_form} and we note that this identity is satisfied
pointwise (in time $t\in \left( 0,T_{\max }\right) $) by the local strong
solution. The proof is finished.
\end{proof}

Every locally-defined \emph{bounded} solution of problem \eqref{p1b}-%
\eqref{p4b} remains bounded for all times provided that the following holds.

\begin{theorem}
\label{global-dyn}Let the assumptions of Theorem \ref{SG} and Lemma \ref%
{poincare-weak} be satisfied. Assume that there exists $\tau _{0}>0$, such
that for any $m\geq 1$ and $\left\vert \tau \right\vert \geq \tau _{0},$ it
holds%
\begin{align}
& -f\left( \tau \right) \left\vert \tau \right\vert ^{m-1}\tau +\frac{\mu
_{\Sigma }\left( \Sigma \right) }{\left\vert \Omega \right\vert }h\left(
\tau \right) \left\vert \tau \right\vert ^{m-1}\tau +\frac{\left( C_{\Omega
,\Sigma }^{\ast }\right) ^{2}}{4m\varepsilon }\left\vert \tau \right\vert
^{m-1}\left( h^{^{\prime }}\left( \tau \right) \tau +mh\left( \tau \right)
\right) ^{2}  \label{balance} \\
& \leq L_{\lambda }\left( m\right) (\left\vert \tau \right\vert ^{m+1}+1),
\notag
\end{align}%
for some $\varepsilon \in (0,d_{0}),$ and some positive function $L_{\lambda
}:\mathbb{R}_{+}\rightarrow \mathbb{R}_{+},$ $L_{\lambda }\left( m\right)
\sim cm^{\lambda }$, for some constants $\lambda ,c>0$, as $m\rightarrow
\infty .$ Here
\begin{equation}
C_{\Omega ,\Sigma }^{\ast }=C_{\Omega ,\Sigma }\frac{\mu _{\Sigma }\left(
\Sigma \right) }{\left\vert \Omega \right\vert }  \label{P-C}
\end{equation}%
and $C_{\Omega ,\Sigma }>0$ is the Poincar\'{e} constant in Lemma \ref%
{poincare-weak} and $d_{0}$ is the constant in (\ref{D}). Then the solution
of problem \eqref{p1b}-\eqref{p4b} is global.
\end{theorem}

\begin{proof}
We have to show that the maximal time $T_{\max }=\infty $ (see \eqref{3.46})
because of the condition \eqref{balance} on the nonlinearities. This ensures
that the solution constructed in the proof of Theorem \ref{SG} is also
global. We shall perform a Moser-type iteration argument. In this step, $C>0$
will denote a constant that is independent of $t,$ $T_{\max }$, $m,$ $k$ and
initial data, which only depends on the other structural parameters of the
problem. Such a constant may vary even from line to line. Moreover, we shall
denote by $L_{\varepsilon }\left( m\right) $ a monotone nondecreasing
function in $m$ of order $\varepsilon ,$ for some nonnegative constant $%
\varepsilon $ independent of $m.$ More precisely, $L_{\varepsilon }\left(
m\right) \sim cm^{\varepsilon }$ as $m\rightarrow \infty $, for some
constant $c>0$.

Let $U(t)=(u(t),u(t)|_{\Sigma})$ be the local strong solution of problem %
\eqref{p1b} -\eqref{p4b} on $(0,T_{\max})$ given by Theorem \ref{SG}. Let $%
m\geq 1$ and consider the function $E_{m}:\;(0,\infty )\rightarrow \lbrack
0,\infty )$ defined by%
\begin{equation*}
E_{m}(t):=\Vert U(t)\Vert _{\mathbb{X}^{m+1}(\Omega ,\Sigma )}^{m+1}=\Vert
u(t)\Vert _{L^{m+1}\left( \Omega \backslash \Sigma \right) }^{m+1}+\Vert
u(t)\Vert _{L^{m+1}\left( \Sigma ,\mu _{\Sigma }\right) }^{m+1}.
\end{equation*}%
Notice that $E_{m}$ is well-defined on $\left( 0,T_{\max }\right) $ because $%
U=\left( u,u|_{\Sigma }\right) $ is bounded in $\Omega \times (0,T_{\max })$%
, $\left\vert \Omega \right\vert <\infty $ (i.e., the $N$-dimensional
Lebesgue measure of $\Omega $ is finite) and $\mu _{\Sigma }\left( \Sigma
\right) <\infty $. Since $U$ is a strong solution on $\left( 0,T_{\max
}\right) ,$ see Definition \ref{solu}, $U$ (as function of $t$) is
differentiable a.e. on $\left( 0,T_{\max }\right) $, whence, the function $%
E_{m}(t)$ is also differentiable for a.e. $t\in \left( 0,T_{\max }\right) $.%
\newline

\noindent \noindent \textit{Step 1 (Recursive relation)}. We begin by
showing that $E_{m}\left( t\right) $ satisfies a local recursive relation
which can be used to perform an iterative argument. Let $\xi =(\left\vert
u\right\vert ^{m-1}u, \left\vert u\right\vert ^{m-1}u|_{\Sigma}),$ $m\geq 1$%
. The boundedness of $u$ mentioned above together with the the fact that $%
U(t)\in D(\mathcal{A}_{\Theta,\Sigma})$ imply that $\xi\in D(\mathcal{A}%
_{\Theta,\Sigma})$. Testing the variational equation (\ref{de_form}) on $%
\left( 0,T_{\max }\right) $ with $\xi =(\left\vert u\right\vert ^{m-1}u,
\left\vert u\right\vert ^{m-1}u|_{\Sigma}),$ $m\geq 1$ gives
\begin{align}
& \frac{1}{m+1}\frac{d}{dt}E_{m}\left( t\right) +{\mathcal{A}}_{\Theta
,\Sigma }(U\left( t\right) ,\xi\left( t\right) )+\int_{\Omega \backslash
\Sigma }f\left( u\left( t\right) \right) \left\vert u\left( t\right)
\right\vert ^{m-1}u\left( t\right) dx  \label{bala-0} \\
& =\int_{\Sigma }h\left( u\left( t\right) \right) \left\vert u\left(
t\right) \right\vert ^{m-1}u\left( t\right) d\mu _{\Sigma }.  \notag
\end{align}%
Now since $\left\vert \Omega \right\vert =\left\vert \Omega \backslash
\Sigma \right\vert <\infty $, $\mu _{\Sigma }\left( \Sigma \right) <\infty ,$
we write%
\begin{align}
& \int_{\Omega \backslash \Sigma }f\left( u\right) \left\vert u\right\vert
^{m-1}udx-\int_{\Sigma }h\left( u\right) \left\vert u\right\vert ^{m-1}ud\mu
_{\Sigma }  \label{bala-1} \\
& =\int_{\Omega \backslash \Sigma }\left[ f\left( u\right) \left\vert
u\right\vert ^{m-1}u-\frac{\mu _{\Sigma }\left( \Sigma \right) }{\left\vert
\Omega \right\vert }h\left( u\right) \left\vert u\right\vert ^{m-1}u\right]
dx  \notag \\
& +\frac{\mu _{\Sigma }\left( \Sigma \right) }{\left\vert \Omega \right\vert
}\int_{\Omega \backslash \Sigma }\left( h\left( u\right) \left\vert
u\right\vert ^{m-1}u-\frac{1}{\mu _{\Sigma }\left( \Sigma \right) }%
\int_{\Sigma }h\left( u\right) \left\vert u\right\vert ^{m-1}ud\mu _{\Sigma
}\right) dx.  \notag
\end{align}%
Following a similar argument applied in \cite[Proposition 3.4]{Gal0}, we now
apply the Poincar\'e inequality (see Lemma \ref{poincare-weak}) to the last
term on the right-hand side of \eqref{bala-1}. We deduce%
\begin{align}
& \frac{\mu _{\Sigma }\left( \Sigma \right) }{\left\vert \Omega \right\vert }%
\left\vert \int_{\Omega \backslash \Sigma }\left( h\left( u\right)
\left\vert u\right\vert ^{m-1}u-\frac{1}{\mu _{\Sigma }\left( \Sigma \right)
}\int_{\Sigma }h\left( u\right) \left\vert u\right\vert ^{m-1}ud\mu _{\Sigma
}\right) dx\right\vert  \label{bala-2} \\
& \leq \frac{C_{\Omega ,\Sigma }\mu _{\Sigma }\left( \Sigma \right) }{%
\left\vert \Omega \right\vert }\left\Vert \nabla \left( h\left( u\right)
u\left\vert u\right\vert ^{m-1}\right) \right\Vert _{L^{1}\left( \Omega
\backslash \Sigma \right) }  \notag \\
& =C_{\Omega ,\Sigma }^{\ast }\left\Vert \left( h^{^{\prime }}\left(
u\right) u+mh\left( u\right) \right) \left\vert u\right\vert ^{m-1}\nabla
u\right\Vert _{L^{1}\left( \Omega \backslash \Sigma \right) }  \notag \\
& =C_{\Omega ,\Sigma }^{\ast }\int_{\Omega \backslash \Sigma }\left\vert
\left( \left\vert u\right\vert ^{\frac{m-1}{2}}\nabla u\right) \left\vert
u\right\vert ^{\frac{m-1}{2}}\left( h^{^{\prime }}\left( u\right) u+mh\left(
u\right) \right) \right\vert dx.  \notag
\end{align}%
By application of H\"{o}lder and Young inequalities, we can estimate the
last term in \eqref{bala-2} as follows:
\begin{align}
& C_{\Omega ,\Sigma }^{\ast }\left( \int_{\Omega \backslash \Sigma
}\left\vert u\right\vert ^{m-1}\left\vert \nabla u\right\vert ^{2}dx\right)
^{1/2}\left( \int_{\Omega \backslash \Sigma }\left\vert u\right\vert
^{m-1}\left( h^{^{\prime }}\left( u\right) u+mh\left( u\right) \right)
^{2}dx\right) ^{1/2}  \label{bala-3} \\
& =C_{\Omega ,\Sigma }^{\ast }\left( \frac{2}{m+1}\right) \left(
m\int_{\Omega \backslash \Sigma }\left\vert \nabla \left\vert u\right\vert ^{%
\frac{m+1}{2}}\right\vert ^{2}dx\right) ^{1/2}  \notag \\
& \times \left( \int_{\Omega \backslash \Sigma }\left\vert u\right\vert
^{m-1}\left( h^{^{\prime }}\left( u\right) u+mh\left( u\right) \right)
^{2}dx\right) ^{1/2}m^{-1/2}  \notag \\
& \leq \frac{4m\varepsilon }{\left( m+1\right) ^{2}}\int_{\Omega \backslash
\Sigma }\left\vert \nabla \left\vert u\right\vert ^{\frac{m+1}{2}%
}\right\vert ^{2}dx+\frac{\left( C_{\Omega ,\Sigma }^{\ast }\right)
^{2}m^{-1}}{4\varepsilon }\int_{\Omega \backslash \Sigma }\left\vert
u\right\vert ^{m-1}\left( h^{^{\prime }}\left( u\right) u+mh\left( u\right)
\right) ^{2}dx,  \notag
\end{align}
for every $\varepsilon>0$, where we have also used that
\begin{align*}
\left(\frac{m+1}{2}\right)^2|u|^{m-1}|\nabla u|^2=|\nabla |u|^{\frac{m+1}{2}%
}|^2.
\end{align*}
Recalling \eqref{bala-1}, owing to \eqref{bala-3} we can estimate%
\begin{align}
& \int_{\{x\in \Sigma :\left\vert u(x)\right\vert \geq s_{0}\}}h\left(
u\right) \left\vert u\right\vert ^{m-1}ud\mu _{\Sigma }-\int_{\{x\in\Omega
\backslash \Sigma :\left\vert u(x)\right\vert \geq s_{0}\}}f\left( u\right)
\left\vert u\right\vert ^{m-1}udx  \label{bala-4} \\
& \leq \int_{\{x\in\Omega \backslash \Sigma :\left\vert u(x)\right\vert \geq
s_{0}\}} \left[ -f\left( u\right) \left\vert u\right\vert ^{m-1}u+\frac{\mu
_{\Sigma }\left( \Sigma \right) }{\left\vert \Omega \right\vert }h\left(
u\right) \left\vert u\right\vert ^{m-1}u\right] dx  \notag \\
& +\frac{\left( C_{\Omega ,\Sigma }^{\ast }\right) ^{2}m^{-1}}{4\varepsilon }%
\int_{\{x\in\Omega \backslash \Sigma :\left\vert u(x)\right\vert \geq
s_{0}\}}\left\vert u\right\vert ^{m-1}\left( h^{^{\prime }}\left( u\right)
u+mh\left( u\right) \right) ^{2}dx+\frac{4m\varepsilon }{\left( m+1\right)
^{2}}\int_{\Omega \backslash \Sigma }\left\vert \nabla \left\vert
u\right\vert ^{\frac{m+1}{2}}\right\vert ^{2}dx.  \notag
\end{align}%
Let $\xi:=(\left\vert u\right\vert ^{m-1}u,\left\vert u\right\vert
^{m-1}u|_{\Sigma})$. Let us now observe that%
\begin{align}
{\mathcal{A}}_{\Theta ,\Sigma }(U,\xi) \geq& \frac{4md_{0}}{\left(
m+1\right) ^{2}}\int_{\Omega \backslash \Sigma }\left\vert \nabla \left\vert
u\right\vert ^{\frac{m+1}{2}}\right\vert ^{2}dx  \label{bala-5} \\
& +\frac{4m}{\left( m+1\right) ^{2}}\int_{\Sigma }\int_{\Sigma
}K(x,y)\left|\left\vert u\left( x\right) \right\vert ^{\frac{m+1}{2}%
}-\left\vert u\left( y\right) \right\vert ^{\frac{m+1}{2}}\right|^{2}d\mu
_{\Sigma }\left( x\right) d\mu _{\Sigma }\left( y\right)  \notag \\
& +\int_{\Sigma }\beta \left( x\right) \left\vert u\right\vert ^{m+1}d\mu
_{\Sigma },  \notag
\end{align}%
owing to \eqref{D} and the fact that
\begin{align*}
& \int_{\Sigma }\int_{\Sigma }K(x,y)(u(x)-u(y))(\left\vert u\right\vert
^{m-1}u(x)-\left\vert u\right\vert ^{m-1}u(y))d\mu _{\Sigma }\left( x\right)
d\mu _{\Sigma }\left( y\right) \\
& \geq \frac{4m}{\left( m+1\right) ^{2}}\int_{\Sigma }\int_{\Sigma
}K(x,y)\left|\left\vert u\left( x\right) \right\vert ^{\frac{m+1}{2}%
}-\left\vert u\left( y\right) \right\vert ^{\frac{m+1}{2}}\right|^{2}d\mu
_{\Sigma }\left( x\right) d\mu _{\Sigma }\left( y\right)
\end{align*}%
which follows from \cite[Lemma 3.4]{GW2}. Combining \eqref{bala-5} together
with \eqref{bala-4} and \eqref{bala-0} and the fact that on the sets $%
\left\{ x\in \Omega \backslash \Sigma :\left\vert u\right\vert \leq
s_{0}\right\} ,$ $\left\{ x\in \Sigma :\left\vert u\right\vert \leq
s_{0}\right\} $, the nonlinearities $f,h$ are bounded, and setting $|U|^{%
\frac{m+1}{2}}=(|u|^{\frac{m+1}{2}}, |u|^{\frac{m+1}{2}}|_{\Sigma})$, we
obtain%
\begin{equation}
\frac{d}{dt}E_{m}\left( t\right) +\gamma {\mathcal{A}}_{\Theta ,\Sigma
}(\left\vert U\left( t\right) \right\vert ^{\frac{m+1}{2}},\left\vert
U\left( t\right) \right\vert ^{\frac{m+1}{2}})\leq L_{\lambda }\left(
m+1\right) \left( \int_{\Omega \backslash \Sigma }\left\vert u\right\vert
^{m+1}dx+1\right) ,  \label{ee1}
\end{equation}%
for all $t\in \left( 0,T_{\max }\right) ,$ for some $\gamma =\gamma \left(
d_{0}\right) >0$ independent of $m,$ and $T_{\max }$. Next, set $%
m_{k}+1=2^{k},$ $k\in \mathbb{N}$, and define%
\begin{equation}
M_{k}:=\sup_{t\in \left( 0,T_{\max }\right) }\left(\int_{\Omega \backslash
\Sigma }\left\vert u\left( t,x\right) \right\vert ^{2^{k}}dx+\int_{\Sigma
}\left\vert u\left( t,x\right) \right\vert ^{2^{k}}d\mu _{\Sigma
}\right)=\sup_{t\in \left( 0,T_{\max }\right) }E_{m_{k}}\left( t\right) .
\label{def}
\end{equation}%
Our goal is to derive a recursive inequality for $M_{k}$ using (\ref{ee1}).
In order to do so, we define%
\begin{equation*}
\overline{p}_{k}:=\frac{m_{k}-m_{k-1}}{q\left( 1+m_{k}\right) -\left(
1+m_{k-1}\right) }=\frac{1}{2q-1}<1,\text{ }\overline{q}_{k}:=1-\overline{p}%
_{k}=2\left( \frac{q-1}{2q-1}\right)
\end{equation*}%
where $q>1$ is such that $D\left( {\mathcal{A}}_{\Theta ,\Sigma }\right)
\subset \mathbb{X}^{2q}\left( \Omega \backslash \Sigma \right) $ (here, $q=%
\frac{d}{N-2}$ with $d\in \left( N-2,N\right) \cap \left( 0,N\right) $, (see %
\eqref{cont-em}). We aim to estimate the term on the right-hand side of (\ref%
{ee1}) in terms of the $L^{1+m_{k-1}}\left( \Omega \backslash \Sigma \right)
$-norm of $u.$ First, we have (using the H\"older inequality and the
embedding $D\left( {\mathcal{A}}_{\Theta ,\Sigma }\right) \subset \mathbb{X}%
^{2q}\left( \Omega \backslash \Sigma \right) $)
\begin{align}
\int_{\Omega \backslash \Sigma }\left\vert u\right\vert ^{1+m_{k}}dx& \leq
\left( \int_{\Omega \backslash \Sigma }\left\vert u\right\vert ^{\left(
1+m_{k}\right) q}dx\right) ^{\overline{p}_{k}}\left( \int_{\Omega \backslash
\Sigma }\left\vert u\right\vert ^{1+m_{k-1}}dx\right) ^{\overline{q}_{k}}
\label{ee5} \\
& \leq C\left( {\mathcal{A}}_{\Theta ,\Sigma }(\left\vert U\right\vert ^{%
\frac{m_{k}+1}{2}},\left\vert U\right\vert ^{\frac{m_{k}+1}{2}})\right) ^{%
\overline{s}_{k}}\left( \int_{\Omega \backslash \Sigma }\left\vert
u\right\vert ^{1+m_{k-1}}dx\right) ^{\overline{q}_{k}},  \notag
\end{align}%
with $\overline{s}_{k}=\overline{p}_{k}q\equiv q/\left( 2q-1\right) \in
\left( 0,1\right) $. Applying now Young's inequality on the right-hand side
of (\ref{ee5}), we get for every $\varepsilon >0,$%
\begin{equation}
L_{\lambda }\left( m_{k}+1\right) \int_{\Omega \backslash \Sigma }\left\vert
u\right\vert ^{1+m_{k}}dx\leq \varepsilon {\mathcal{A}}_{\Theta ,\Sigma
}(\left\vert U\right\vert ^{\frac{m_{k}+1}{2}},\left\vert U\right\vert ^{%
\frac{m_{k}+1}{2}})+L_{\alpha }\left( m_{k}+1\right) \left( \int_{\Omega
\backslash \Sigma }\left\vert u\right\vert ^{1+m_{k-1}}dx\right) ^{2},
\label{ee5t}
\end{equation}%
for some $\alpha =\alpha \left( \varepsilon ,\lambda \right) >0$ independent
of $k$ since $\overline{q}_{k}/\left( 1-\overline{s}_{k}\right) \equiv 2$.
Hence, inserting \eqref{ee5t} into \eqref{ee1}, choosing a sufficiently
small $\varepsilon =\varepsilon _{0}<\min \left( \gamma /2,1\right) $, and
simplifying, we obtain for $t\in \left( 0,T_{\max }\right) ,$%
\begin{equation}
\frac{d}{dt}E_{m_{k}}\left( t\right) +\varepsilon _{0}{\mathcal{A}}_{\Theta
,\Sigma }(\left\vert U\left( t\right) \right\vert ^{\frac{m_{k}+1}{2}%
},\left\vert U\left( t\right) \right\vert ^{\frac{m_{k}+1}{2}})\leq
L_{\alpha }\left( m_{k}+1\right) \left( E_{m_{k-1}}\left( t\right) \right)
^{2}.  \label{ee6b}
\end{equation}%
Next, since $U\left( t\right) \in D\left(\mathcal{A}_{\Theta ,\Sigma}\right)
\cap \mathbb{X}^{\infty }\left( \Omega ,\Sigma \right) $, we have $|U(t)|^{%
\frac{1+m_k}{2}}:=(\left\vert u(t)\right\vert ^{\frac{1+m_{k}}{2}},
\left\vert u(t)\right\vert ^{\frac{1+m_{k}}{2}}|_{\Sigma})\in D({\mathcal{A}}%
_{\Theta ,\Sigma })$ for a.e. $t\in \left( 0,T_{\max }\right) $. Thus, we
can apply Lemma \ref{poincare-forms} (see (\ref{PDBC})) to infer that%
\begin{equation}
\varepsilon _{0}{\mathcal{A}}_{\Theta ,\Sigma }(\left\vert U\left( t\right)
\right\vert ^{\frac{m_{k}+1}{2}},\left\vert U\left( t\right) \right\vert ^{%
\frac{m_{k}+1}{2}})\geq E_{m_{k}}\left( t\right) -\varepsilon _{0}^{-\zeta
}\left( E_{m_{k-1}}\left( t\right) \right) ^{2}.  \label{ee6bis}
\end{equation}%
We can now combine \eqref{ee6b} with \eqref{ee6bis} to deduce%
\begin{equation}
\frac{d}{dt}E_{m_{k}}\left( t\right) +E_{m_{k}}\left( t\right) \leq
L_{\alpha }\left( 2^{k}\right) M_{k-1}^{2},  \label{ee7bis}
\end{equation}%
for $t\in \left( 0,T_{\max }\right) .$ Integrating \eqref{ee7bis} over $%
\left( 0,t\right) $, we infer from Gronwall-Bernoulli's inequality \cite[%
Lemma 1.2.4]{CD} that there exists yet another constant $C>0,$ independent
of $k$, such that%
\begin{equation}
M_{k}\leq \max \left\{ E_{m_{k}}\left( 0\right) ,C2^{k\alpha
}M_{k-1}^{2}\right\} ,\text{ for all }k\geq 2.  \label{claim2}
\end{equation}%
On the other hand, let us observe that there exists a positive constant $%
C_{\infty }=C_{\infty }(\left\Vert U_{0}\right\Vert _{\mathbb{X}^{\infty
}(\Omega ,\Sigma )})\geq 1,$ independent of $k$, such that $E_{m_{k}}\left(
0\right) ^{1/2^{k}}\leq C_{\infty }$. Taking the $2^{k}$-th root on both
sides of (\ref{claim2}), and defining $X_{k}:=\sup_{t\in \left( 0,T_{\max
}\right) }\left(E_{m_{k}}\left( t\right)\right) ^{1/2^{k}},$ we easily
arrive at%
\begin{equation}
X_{k}\leq \max \left\{ C_{\infty },\left( C2^{\alpha k}\right) ^{\frac{1}{%
2^{k}}}X_{k-1}\right\} ,\text{ for all }k\geq 2.  \label{ee7}
\end{equation}%
By straightforward induction in (\ref{ee7}) (see \cite[Lemma 3.2]{Ali}; cf.
also \cite[Lemma 9.3.1]{CD}), we finally obtain the estimate%
\begin{equation}
\sup_{t\in \left( 0,T_{\max }\right) }\left\Vert U\left( t\right)
\right\Vert _{\mathbb{X}^{\infty }(\Omega ,\Sigma )}\leq \lim_{k\rightarrow
+\infty }X_{k}\leq C\max \left\{ C_{\infty },\sup_{t\in \left( 0,T_{\max
}\right) }\left\Vert U\left( t\right) \right\Vert _{\mathbb{X}^{2}(\Omega
,\Sigma )}\right\} .  \label{ee8}
\end{equation}

\noindent \textit{Step 2 (The }$\mathbb{X}^{2}(\Omega ,\Sigma )$\textit{%
-bound)}. It remains to derive a global $L^{2}$-bound on the right-hand side
of (\ref{ee8}) in order to get full control of the $L^{\infty }$-bound. From %
\eqref{ee1} we readily see that%
\begin{equation}
\frac{d}{dt}E_{1}\left( t\right) +\gamma {\mathcal{A}}_{\Theta ,\Sigma
}(U\left( t\right) ,U\left( t\right) )\leq C\left( E_{1}\left( t\right)
+1\right) .  \label{3.1}
\end{equation}%
Integrating \eqref{3.1} over $\left( 0,t\right) $ with $t\in \left(
0,T\right) $ for any $T>0$ yields%
\begin{equation}
\left\Vert U\left( t\right) \right\Vert _{\mathbb{X}^{2}(\Omega ,\Sigma
)}^{2}+\gamma \int_{0}^{t}{\mathcal{A}}_{\Theta ,\Sigma }(U\left( \tau
\right) ,U\left( \tau \right) )d\tau \leq \left( \left\Vert U_{0}\right\Vert
_{\mathbb{X}^{2}(\Omega ,\Sigma )}^{2}+1\right) e^{Ct}.  \label{3.3}
\end{equation}%
Thus, we have derived a bound for $U=\left( u,u|_{\Sigma }\right) \in
L^{\infty }\left( (0,T);\newline
\mathbb{X}^{2}(\Omega ,\Sigma )\right) $, for any $T>0$. Finally, \eqref{ee8}
together with the global bound (\ref{3.3}) shows that $\left\Vert U\left(
t\right) \right\Vert _{\mathbb{X}^{\infty }(\Omega ,\Sigma )}$ is bounded
for all times $t>0$ with a bound, independent of $T_{\max },$ depending only
on $\left\Vert U_{0}\right\Vert _{\mathbb{X}^{\infty }(\Omega ,\Sigma )}$, $%
\left\vert \Omega \right\vert ,$ $\mu _{\Sigma }\left( \Sigma \right) ,$ $%
T>0 $ and the growth of the nonlinear functions $f,h.$ This gives $T_{\max
}=+\infty $ so that the (local) strong solution given by Theorem \ref{SG} is
in fact global. This completes the proof of the theorem.
\end{proof}

Consequently, we have the following general result in the case of polynomial
nonlinearities with a \emph{bad source} $h$ of \emph{arbitrary growth}
satisfying (\ref{bad-h}) for as long as the polynomial nonlinearity $f$\
acting in $\Omega \backslash \Sigma $ is strong enough to overcome it.

\begin{corollary}
\label{example-C}Let the assumptions of Theorem \ref{SG} and Lemma \ref%
{poincare-weak} be satisfied. Suppose that
\begin{equation*}
\lim_{\left\vert \tau\right\vert \rightarrow \infty }\frac{h^{^{\prime
}}\left( \tau\right) }{\left\vert \tau\right\vert ^{p}}=\left( p+1\right)
c_{h}\text{ and }\lim_{\left\vert \tau\right\vert \rightarrow \infty }\frac{%
f^{^{\prime }}\left( \tau\right) }{\left\vert \tau\right\vert ^{q}}=\left(
q+1\right) c_{f}
\end{equation*}%
with $c_{h}>0,c_{f}>0$ for some $p,q\geq 0$. Then the conclusion of Theorem %
\ref{global-dyn} holds provided that $q>2p$. In particular, problem \eqref
{p1b}-\eqref{p4b} possesses a unique global bounded solution in the sense of
Definition \ref{solu}.
\end{corollary}

\begin{proof}
We begin by noting that for large $\left\vert \tau\right\vert \geq \tau_{0}$%
, we have
\begin{align*}
h\left( \tau\right) \sim c_{h}\left\vert \tau\right\vert
^{p}\tau,\;\;h\left( \tau\right) \tau\sim c_{h}\left\vert \tau\right\vert
^{p+2}\;\mbox{ and }\;\;f\left( \tau\right) \sim c_{f}\left\vert
\tau\right\vert ^{q}\tau, \;\;f\left( \tau\right) \tau\sim c_{f}\left\vert
\tau\right\vert ^{q+2}.
\end{align*}
Therefore as $\left\vert \tau\right\vert \rightarrow \infty $, the leading
terms on the left-hand side of \eqref{balance} are%
\begin{equation}
-c_{f}\left\vert \tau\right\vert ^{m+q+1}+\frac{\mu _{\Sigma }\left( \Sigma
\right) }{\left\vert \Omega \right\vert }c_{h}\left\vert \tau\right\vert
^{m+p+1}+\frac{\left( C_{\Omega ,\Sigma }^{\ast }\right) ^{2}\left(
m+p+1\right) ^{2}}{4m\varepsilon }c_{h}^{2}\left\vert \tau\right\vert
^{m+1+2p}  \label{h-term}
\end{equation}%
for any $m\geq 1$. By assumption $q>2p$ so that the coefficient of the
highest-order term in (\ref{h-term}) is $-c_{f}<0$, whence \eqref{balance}
is satisfied and the proof is finished.
\end{proof}

A close investigation of the proof of Theorems \ref{global-dyn} shows that
one can derive another global result when only \emph{minimal} geometrical
assumptions on the interface $\Sigma $ are required and when the function $h$
is still of bad sign but grows at most linearly at infinity.

\begin{corollary}
\label{cor-weaker}Assume that $\Omega \backslash \Sigma $ has the $%
\widetilde{W}^{1,2}$-extension property and that
\begin{align*}
f\left( \tau\right) \geq-c_{f}\tau^{2}\;\mbox{ and }\;h\left( \tau\right)
\leq c_{h}\tau^{2}\;\mbox{ for }\; \left\vert \tau\right\vert \geq \tau_{0},
\end{align*}
for some sufficiently large $\tau_{0}>0$ and some $c_{f},c_{h}>0$. Then for
every $\left( u_{0},v_{0}\right) \in \mathbb{X}^{\infty }(\Omega ,\Sigma ),$
there exists a unique global solution $U$ of \eqref{p1b}-\eqref{p4b}.
\end{corollary}

\begin{proof}
By assumption, one can find $C_{f}\geq 0$ and $C_{h}\geq 0$ such that $%
f\left( \tau\right) \geq -c_{f}\tau^{2}-C_{f}$ and $h\left( \tau\right) \leq
c_{h}\tau^{2}+C_{h}$ for all $\tau\in \mathbb{R}$. These conditions yield $%
f\left( \tau\right) \left\vert \tau\right\vert ^{m-1}\tau\geq
-c_{f}\left\vert \tau\right\vert ^{m+1}$ and $h\left( \tau\right) \tau\leq
c_{h}\left\vert \tau\right\vert ^{m+1}$ for large enough $\left\vert
\tau\right\vert \geq \tau_{0}$. Henceforth, it is easy to see that
inequality \eqref{ee6b} with $m\geq 1$ still holds in this case owing to %
\eqref{ee5t}. It follows that $U\in L^{\infty }\left((0,T);\mathbb{X}%
^{\infty }(\Omega ,\Sigma )\right) $ owing to the \textit{Steps 1,2} of the
proof of Theorem \ref{global-dyn}. This completes the proof.
\end{proof}

\begin{remark}
\label{ex-lip} \emph{If $\Sigma $ is a Lipschitz hypersurface of dimension $%
N-1$, then $\mu _{\Sigma }=\sigma _{\Sigma }$, and all hypotheses of Lemma %
\ref{poincare-weak} are satisfied and $\Omega \backslash \Sigma $ has the $%
\widetilde{W}^{1,2}$-extension property. Thus, all the conclusions of
Theorem \ref{SG} and Theorem \ref{global-dyn} hold for the transmission
problem \eqref{p1b}-\eqref{p4b} in this case provided that the
nonlinearities satisfy the given assumptions. In particular, we recover the
existence results given in \cite{EMR2,EMR} for the linear transmission
problem with $\left( f,h\right) =\left( 0,0\right) $ and $\mathbf{D}$ is
non-degenerate and symmetric.}
\end{remark}

We finally conclude this section with the following result.

\begin{corollary}
\label{dyn_system}Let the assumptions of either Theorem \ref{global-dyn} or
Corollary \ref{cor-weaker} be satisfied. Then the transmission problem %
\eqref{p1b}-\eqref{p4b} defines a (nonlinear) continuous semigroup $\mathcal{%
S}\left( t\right) :\mathbb{X}^{\infty }(\Omega ,\Sigma )\rightarrow \mathbb{X%
}^{\infty }(\Omega ,\Sigma )$, given by
\begin{equation*}
\mathcal{S}\left( t\right) U_{0}=U\left( t\right) =\left( u\left( t\right)
,u\left( t\right) |_{\Sigma }\right) ,
\end{equation*}%
where $U$ is the (unique) strong solution of \eqref{p1b}-\eqref{p4b} in the
sense of Definition \ref{solu}.
\end{corollary}

\section{Finite dimensional attractors}

\label{gl}

The present section is focused on the long-term analysis of the transmission
problem \eqref{p1b}-\eqref{p4b}. We proceed to investigate its asymptotic
properties using the notion of an exponential attractor. We begin with the
following.

\begin{definition}
\label{gl_notion} Let $\mathcal{S}\left( t\right)$ be the semigroup on $%
\mathbb{X}^{\infty }(\Omega ,\Sigma )$ associated with \eqref{p1b}-%
\eqref{p4b} given in Corollary \ref{dyn_system}. A set $\mathcal{G}_{\Theta
,\Sigma }$ is an \textbf{exponential attractor} of the semigroup $\mathcal{S}%
\left( t\right) $ if the following assertions hold.

\begin{itemize}
\item $\mathcal{G}_{\Theta ,\Sigma }$ is compact in $\mathbb{X}^{2}\left(
\Omega ,\Sigma \right) $ and bounded in $\mathbb{X}^{\infty }\left( \Omega
,\Sigma \right) \cap D(\mathcal{A}_{\Theta ,{\Sigma }})$;

\item $\mathcal{G}_{\Theta ,\Sigma }$ is positively invariant, that is, $%
\mathcal{S}(t)\mathcal{G}_{\Theta ,\Sigma }\subseteq \mathcal{G}_{\Theta
,\Sigma },\;\;\;\forall \;t\geq 0$;

\item $\mathcal{G}_{\Theta ,\Sigma }$ attracts the images of all bounded
subsets of $\mathbb{X}^{\infty }\left( \Omega ,\Sigma \right) $ at an
exponential rate, namely, there exist two constants $\rho >0,C>0$ such that%
\begin{equation*}
dist_{\mathbb{X}^{\infty }\left( \Omega ,\Sigma \right) }\left( \mathcal{S}%
\left( t\right) B,\mathcal{G}_{\Theta ,\Sigma }\right) \leq Ce^{-\rho t},%
\text{ for all }t\geq 0,
\end{equation*}%
for every bounded subset $B$ of $\mathbb{X}^{\infty }\left( \Omega ,\Sigma
\right) $. Here, $dist_{\mathcal{H}}$ denotes the standard Hausdorff
semidistance between sets in a Banach space $\mathcal{H}$;

\item $\mathcal{G}_{\Theta ,\Sigma }$ has finite fractal dimension in $%
\mathbb{X}^{\infty }\left( \Omega ,\Sigma \right) $.
\end{itemize}
\end{definition}

The main result of this section gives the existence of such an attractor.

\begin{theorem}
\label{m1}Let the assumptions of Corollary \ref{dyn_system} be satisfied and
assume that $\Omega \backslash \Sigma $ has the $\widetilde{W}^{1,2}$%
-extension property. Furthermore, assume that for all $\tau \in \mathbb{R}$
it holds%
\begin{equation}
-f\left( \tau \right) \tau +\frac{\mu _{\Sigma }\left( \Sigma \right) }{%
\left\vert \Omega \right\vert }h\left( \tau \right) \tau +\frac{\left(
C_{\Omega ,\Sigma }^{\ast }\right) ^{2}}{4\varepsilon }\left( h^{^{\prime
}}\left( \tau \right) \tau +h\left( \tau \right) \right) ^{2}\leq \lambda
_{\ast }\tau ^{2}+C_{f,h}  \label{dissip-bala}
\end{equation}%
for some $\varepsilon \in (0,d_{0}),$ $C_{f,h}\geq 0$ and $\lambda _{\ast
}\in \lbrack 0,\overline{C})$ where $\overline{C}=C\left( \Omega ,\Sigma
,d_{0},\beta \right) >0$ is the best Sobolev-Poincar\'{e} constant in the
embedding (for $U\in D(\mathcal{A}_{\Theta ,\Sigma })$)
\begin{equation}
\overline{C}\left\Vert U\right\Vert _{\mathbb{X}^{2}\left( \Omega \backslash
\Sigma \right) }^{2}\leq \left( d_{0}-\varepsilon \right) \left\Vert \nabla
u\right\Vert _{L^{2}\left( \Omega \backslash \Sigma \right) }^{2}+\left\Vert
\beta ^{1/2}u\right\Vert _{L^{2}\left( \Sigma ,d\mu _{\Sigma }\right)
}^{2}+\int_{\Sigma }\int_{\Sigma }K(x,y)|u(x)-u(u)|^{2}\;d\mu _{\Sigma
}(x)d\mu _{\Sigma }(y).  \label{embed}
\end{equation}%
Then problem \eqref{p1b}-\eqref{p4b} has an exponential attractor $\mathcal{G%
}_{\Theta ,\Sigma }$\ in the sense of Definition \ref{gl_notion}.
\end{theorem}

\begin{remark}
\emph{Notice that \eqref{dissip-bala} is roughly the same as the general
balance condition \eqref{balance} when $m=1$ but one has explicit control of
the constant on the right-hand side of \eqref{balance} when $m=1$. We also
note that \eqref{embed} is always satisfied for $U\in D(\mathcal{A}_{\Theta ,%
{\Sigma }})$. Moreover, if $\Sigma $ is as in Remark \ref{ex-lip} or more
generally, $\Omega \backslash \Sigma $ has the $\widetilde{W}^{1,2}$%
-extension property, then the embedding $D(\mathcal{A}_{\Theta ,\Sigma
})\hookrightarrow \mathbb{X}^{2}(\Omega ,\Sigma )$ is also compact.}
\end{remark}

Since the exponential attractor always contains the\ global attractor, as a
consequence of Theorem \ref{m1} we immediately have the following.

\begin{theorem}
\label{m1b}Let the assumptions of Theorem \ref{m1} be satisfied. The
semigroup $\mathcal{S}\left( t\right) $ associated with the transmission
problem \eqref{p1b}-\eqref{p4b} possesses a global attractor $\mathbb{A}%
_{\Theta ,\Sigma },$ bounded in $\mathbb{X}^{\infty }\left( \Omega ,\Sigma
\right) \cap D(\mathcal{A}_{\Theta ,{\Sigma }})$, compact in $\mathbb{X}%
^{2}\left( \Omega ,\Sigma \right) $ and of finite fractal dimension in the $%
\mathbb{X}^{\infty }\left( \Omega ,\Sigma \right) $-topology. This attractor
is generated by all complete bounded trajectories of \eqref{p1b}-\eqref{p4b}%
, that is, $\mathbb{A}_{\Theta ,\Sigma }=\mathcal{K}_{\Theta ,\Sigma \mid
t=0}$, where $\mathcal{K}_{\Theta ,\Sigma }$ is the set of all strong
solutions $U=\left( u,u|_{\Sigma }\right) $ which are defined for all $t\in
\mathbb{R}_{+}$ and bounded in the $\mathbb{X}^{\infty }\left( \Omega
,\Sigma \right) \cap D(\mathcal{A}_{\Theta ,{\Sigma }})$-norm.
\end{theorem}

Our construction of an exponential attractor is based on the following
abstract result \cite[Proposition 4.1]{EZ}.

\begin{proposition}
\label{abstract}Let $\mathcal{H}$,$\mathcal{V}$,$\mathcal{V}_{1}$ be Banach
spaces such that the embedding $\mathcal{V}_{1}\hookrightarrow \mathcal{V}$
is compact. Let $\mathbb{B}$ be a closed bounded subset of $\mathcal{H}$ and
let $\mathbb{S}:\mathbb{B}\rightarrow \mathbb{B}$ be a map. Assume also that
there exists a uniformly Lipschitz continuous map $\mathbb{T}:\mathbb{B}%
\rightarrow \mathcal{V}_{1}$, i.e.,%
\begin{equation}
\left\Vert \mathbb{T}b_{1}-\mathbb{T}b_{2}\right\Vert _{\mathcal{V}_{1}}\leq
L\left\Vert b_{1}-b_{2}\right\Vert _{\mathcal{H}},\quad \forall\;
b_{1},b_{2}\in \mathbb{B},  \label{gl1}
\end{equation}%
for some $L\geq 0$, such that%
\begin{equation}
\left\Vert \mathbb{S}b_{1}-\mathbb{S}b_{2}\right\Vert _{\mathcal{H}}\leq
\gamma \left\Vert b_{1}-b_{2}\right\Vert _{\mathcal{H}}+K\left\Vert \mathbb{T%
}b_{1}-\mathbb{T}b_{2}\right\Vert _{\mathcal{V}},\quad \forall\;
b_{1},b_{2}\in \mathbb{B},  \label{gl2}
\end{equation}%
for some constant $0\le \gamma <\frac{1}{2}$ and $K\geq 0$. Then, there
exists a (discrete) exponential attractor $\mathcal{M}_{d}\subset \mathbb{B}$
of the semigroup $\{\mathbb{S}(n):=\mathbb{S}^{n},n\in \mathbb{Z}_{+}\}$
with discrete time in the phase space $\mathcal{H}$, which satisfies the
following properties:

\begin{itemize}
\item semi-invariance: $\mathbb{S}\left( \mathcal{M}_{d}\right) \subset
\mathcal{M}_{d}$;

\item compactness: $\mathcal{M}_{d}$ is compact in $\mathcal{H}$;

\item exponential attraction: $dist_{\mathcal{H}}(\mathbb{S}^{n}\mathbb{B},%
\mathcal{M}_{d})\leq Ce^{-\alpha n},$ for all $n\in \mathbb{N}$ and for some
$\alpha >0$ and $C\geq 0$, where $dist_{\mathcal{H}}$ denotes the standard
Hausdorff semidistance between sets in $\mathcal{H}$;

\item finite-dimensionality: $\mathcal{M}_{d}$ has finite fractal dimension
in $\mathcal{H}$.
\end{itemize}
\end{proposition}

\begin{remark}
\emph{The constants $C$ and $\alpha ,$ and the fractal dimension of $%
\mathcal{M}_{d}$ can be explicitly expressed in terms of $L$, $K$, $\gamma $%
, $\left\Vert \mathbb{B}\right\Vert _{\mathcal{H}}$ (and hence, in terms of
the Sobolev-Poincar\'e constants involved in the previous Poincar\'e
inequalities) and Kolmogorov's $\kappa $-entropy of the compact embedding $%
\mathcal{V}_{1}\hookrightarrow \mathcal{V},$ for some $\kappa =\kappa \left(
L,K,\gamma \right) $. We recall that the Kolmogorov $\kappa $-entropy of the
compact embedding $\mathcal{V}_{1}\hookrightarrow \mathcal{V}$ is the
logarithm of the minimum number of balls of radius $\kappa $ in $\mathcal{V}$
necessary to cover the unit ball of $\mathcal{V}_{1}$.}
\end{remark}

We will prove the main theorem by carrying first a sequence of dissipative
estimates for the strong solution and then applying Proposition \ref%
{abstract} to our situation at the end.

\begin{lemma}
\label{dissipation}Under the assumptions of Theorem \ref{m1}, there exists a
sufficiently large radius $R>0$ independent of time and the initial data,
such that the ball%
\begin{equation}  \label{B}
\mathcal{B}\overset{\text{def}}{=}\left\{ U\in \mathcal{X}=D(\mathcal{A}%
_{\Theta ,{\Sigma }})\cap \mathbb{X}^{\infty }\left( \Omega ,\Sigma \right)
:\left\Vert U\right\Vert _{\mathcal{X}}\leq R\right\} ,
\end{equation}%
is an absorbing set for $\mathcal{S}\left( t\right) $ in $\mathbb{X}^{\infty
}\left( \Omega ,\Sigma \right) $. More precisely, for any bounded set $%
B\subset \mathbb{X}^{\infty }\left( \Omega ,\Sigma \right) $, there exists a
time $t_{\ast }=t_{\ast }(B)>0$ such that $\mathcal{S}\left( t\right)
B\subset \mathcal{B}$, for all $t\geq t_{\ast }.$
\end{lemma}

\begin{proof}
Let $U(t)$ be the unique strong solution of \eqref{p1b}-\eqref{p4b}.
Consider any real numbers $\tau ^{^{\prime }}>\tau >0$ and fix $\mu :=\tau
^{^{\prime }}-\tau $. There exists a positive constant $C=C\left( \mu
\right) \sim \mu ^{-\eta }$ (for some $\eta >0$), independent of $t$ and the
initial data, such that%
\begin{equation}
\sup_{t\geq \tau ^{^{\prime }}}\left\Vert U\left( t\right) \right\Vert _{%
\mathbb{X}^{\infty }\left( \Omega ,\Sigma \right) }\leq C\sup_{\sigma \geq
\tau }\left\Vert U\left( \sigma \right) \right\Vert _{\mathbb{X}^{2}\left(
\Omega ,\Sigma \right) }.  \label{sup-b}
\end{equation}%
Following \cite[Theorem 2.3]{Gal0} (cf. also \cite{GalNS, GW}), (\ref{sup-b}%
) is a consequence of the same recursive inequality for $E_{m_{k}}\left(
t\right) $ from (\ref{ee6b}). Arguing in a similar fashion as in our recent
work \cite{GW}, (\ref{ee6b}) allows us to deduce the following stronger
inequality%
\begin{equation}
\sup_{t\geq t_{k-1}}E_{m_{k}}\left( t\right) \leq C\left( 2^{k}\right)
^{l}\left( \sup_{\sigma \geq t_{k}}E_{m_{k-1}}\left( \sigma \right) \right)
^{2},\text{ \ for all }k\geq 1,  \label{rec_rel}
\end{equation}%
where the sequence $\left\{ t_{k}\right\} _{k\in \mathbb{N}}$ is defined
recursively $t_{k}=t_{k-1}-\mu /2^{k},$ $k\geq 1$, $t_{0}=\tau ^{^{\prime
}}. $ Here we recall that $C=C\left( \mu \right) >0,$ $l>0$ are independent
of $k $ and $C\left( \mu \right) $ is uniformly bounded in $\mu $ if $\mu
\geq 1$ (see \cite[Theorem 2.3]{Gal0}). Iterating in (\ref{rec_rel}) with
respect to $k\geq 1$ we deduce (\ref{sup-b}). Thus, the existence of an
absorbing ball in $\mathbb{X}^{2}\left( \Omega ,\Sigma \right) $ together
with (\ref{sup-b}) gives an absorbing ball for $\mathcal{S}\left( t\right) $
in the space $\mathbb{X}^{\infty }\left( \Omega ,\Sigma \right) $. We now
show how to derive the property in $\mathbb{X}^{2}\left( \Omega ,\Sigma
\right) $ for the semigroup. As in \textit{Step 2 and \eqref{bala-0}-%
\eqref{bala-5}} of the proof of Theorem \ref{global-dyn}, we have
\begin{equation}
\frac{1}{2}\frac{d}{dt}E_{1}\left( t\right) +{\mathcal{A}}_{\Theta ,\Sigma
}(U\left( t\right) ,U\left( t\right) )+\int_{\Omega \backslash \Sigma
}f\left( u\left( t\right) \right) u\left( t\right) dx=\int_{\Sigma }h\left(
u\left( t\right) \right) u\left( t\right) d\mu _{\Sigma }  \label{4.1}
\end{equation}%
and we can estimate%
\begin{align}
& \int_{\Sigma }h\left( u\right) ud\mu _{\Sigma }-\int_{\Omega \backslash
\Sigma }f\left( u\right) udx  \label{4.2} \\
& \leq \int_{\Omega \backslash \Sigma }\left[ -f\left( u\right) u+\frac{\mu
_{\Sigma }\left( \Sigma \right) }{\left\vert \Omega \right\vert }h\left(
u\right) u\right] dx  \notag \\
& +\frac{\left( C_{\Omega ,\Sigma }^{\ast }\right) ^{2}}{4\varepsilon }%
\int_{\Omega \backslash \Sigma }\left( h^{^{\prime }}\left( u\right)
u+h\left( u\right) \right) ^{2}dx+\varepsilon \int_{\Omega \backslash \Sigma
}\left\vert \nabla u\right\vert ^{2}dx,  \notag
\end{align}%
for every $\varepsilon>0$. Moreover,%
\begin{align}
{\mathcal{A}}_{\Theta ,\Sigma }(U,U)& \geq d_{0}\int_{\Omega \backslash
\Sigma }\left\vert \nabla u\right\vert ^{2}dx+\int_{\Sigma }\int_{\Sigma
}K(x,y)(u\left( x\right) -u\left( y\right) )^{2}d\mu _{\Sigma }\left(
x\right) d\mu _{\Sigma }\left( y\right)  \notag  \label{DD} \\
& +\int_{\Sigma }\beta \left( x\right) \left\vert u\right\vert ^{2}d\mu
_{\Sigma },
\end{align}%
owing once again to \eqref{D}. Using \eqref{4.1}, \eqref{4.2}, \eqref{DD}
and recalling \eqref{dissip-bala}, we obtain%
\begin{equation}
\frac{1}{2}\frac{d}{dt}E_{1}\left( t\right) +{\mathcal{A}}_{\Theta ,\Sigma
}(U\left( t\right) ,U\left( t\right) )\leq \lambda _{\ast }\int_{\Omega
\backslash \Sigma }\left\vert u\left( t\right) \right\vert
^{2}dx+\varepsilon \int_{\Omega \backslash \Sigma }\left\vert \nabla u\left(
t\right) \right\vert ^{2}dx+C,  \label{4.3}
\end{equation}%
for some constant $C>0$ which depends only on $\Omega $ and $\,f,h.$ The
embedding \eqref{embed} then yields from \eqref{4.3} that
\begin{equation}
\frac{dE_{1}(t)}{dt}+2\left( \overline{C}-\lambda _{\ast }\right)
E_{1}\left( t\right) \leq C,\;\mbox{ for all }\;t\geq 0.  \label{AB}
\end{equation}%
Integrating \eqref{AB} over $\left( 0,t\right) $ gives that $E_{1}\left(
t\right) \leq E_{1}\left( 0\right) e^{-\eta t}+C,$ with $\eta =\overline{C}%
-\lambda _{\ast }>0$, for some $C>0$ independent of time and initial data.
Moreover, it holds%
\begin{equation}
\int_{t}^{t+1}{\mathcal{A}}_{\Theta ,\Sigma }(U\left( \tau \right) ,U\left(
\tau \right) )d\tau \leq C\left\Vert U_{0}\right\Vert _{\mathbb{X}^{2}\left(
\Omega ,\Sigma \right) }^{2}e^{-\eta t}+C,\text{ for all }t\geq 0.
\label{4.4}
\end{equation}%
Henceforth, the existence of a bounded absorbing ball for the semigroup $%
\mathcal{S}\left( t\right) $ in the space $\mathbb{X}^{2}\left( \Omega
,\Sigma \right) $ (and therefore, in the space $\mathbb{X}^{\infty }\left(
\Omega ,\Sigma \right) $) immediatelly follows. In order to get the
existence of a bounded absorbing set in $D\left( \mathcal{A}_{\Theta ,{%
\Sigma }}\right) $ we argue as follows. Testing (\ref{de_form}) with $\xi
=(\partial _{t}u\left( t\right) ,\partial _{t}u\left( t\right) |_{\Sigma })$
(note that such a test function is allowed by the regularity (\ref{reg_weak}%
) of the solution) we find%
\begin{align}
& \frac{d}{dt}\left( \left\Vert U\left( t\right) \right\Vert _{D\left( {%
\mathcal{A}}_{\Theta ,\Sigma }\right) }^{2}+2\left( \overline{f}\left(
u\left( t\right) \right) ,1\right) _{L^{2}\left( \Omega \backslash \Sigma
\right) }-2\left( \overline{h}\left( u\left( t\right) \right) ,1\right)
_{L^{2}\left( \Sigma ,\mu _{\Sigma }\right) }\right)  \label{en} \\
& =-2\left\Vert \partial _{t}u\left( t\right) \right\Vert _{L^{2}\left(
\Omega \backslash \Sigma \right) }^{2}-2\left\Vert \partial _{t}u\left(
t\right) \right\Vert _{L^{2}\left( \Sigma ,\mu _{\Sigma }\right) }^{2},
\notag
\end{align}%
for $t>0.$ Here and below, $\overline{f}$ and $\overline{h}$ denote the
primitives of $f$ and $h$, respectively, i.e., $\overline{f}\left( \tau
\right) =\int_{0}^{\tau }f\left( y\right) dy$ and $\overline{h}\left( \tau
\right) =\int_{0}^{\tau }h\left( y\right) dy.$ The application of the
uniform Gronwall's lemma (see, e.g., \cite[Lemma III.1.1]{T}) together with (%
\ref{4.4}) and the existence of an absorbing set for $\mathcal{S}\left(
t\right) $ in the space $\mathbb{X}^{\infty }\left( \Omega ,\Sigma \right) $
yields the existence of a time $t_{\ast }=t_{\ast }\left( B\right) $ ($B$ is
any bounded set of initial data contained in $\mathbb{X}^{\infty }\left(
\Omega ,\Sigma \right) $) such that%
\begin{equation}
\sup_{t\geq t_{\ast }}\left( \int_{t}^{t+1}\left( \left\Vert \partial
_{t}u\left( \tau \right) \right\Vert _{L^{2}\left( \Omega \backslash \Sigma
\right) }^{2}+\left\Vert \partial _{t}u\left( \tau \right) \right\Vert
_{L^{2}\left( \Sigma ,\mu _{\Sigma }\right) }^{2}\right) d\tau +{\mathcal{A}}%
_{\Theta ,\Sigma }(U\left( t\right) ,U\left( t\right) )\right) \leq C,
\label{4.4bis}
\end{equation}%
for some constant $C>0$ independent of time and the initial data. This final
estimate implies the existence of a bounded absorbing set in $\mathcal{X}$
and the claim follows.
\end{proof}

Next we carry some estimates for the difference of any two strong solutions,
estimates which will become crucial in the final proof of Theorem \ref{m1}.

\begin{lemma}
\label{L5}Let the assumptions of Theorem \ref{m1} hold, and let $%
U_{1}=\left( u_{1},u_{1}|_{\Sigma }\right) $ and $U_{2}=\left(
u_{2},u_{2}|_{\Sigma }\right) $ be two strong solutions of \eqref{p1b}-\eqref%
{p4b} such that $U_{i}\left( 0\right) \in \mathcal{B}$, $i=1,2$ Then the
following estimates are valid:%
\begin{equation}
\left\Vert U_{1}\left( t\right) -U_{2}\left( t\right) \right\Vert _{\mathbb{X%
}^{2}\left( \Omega ,\Sigma \right) }^{2}\leq M\left\Vert U_{1}\left(
0\right) -U_{2}\left( 0\right) \right\Vert _{\mathbb{X}^{2}\left( \Omega
,\Sigma \right) }^{2}e^{-\omega t}+K\left\Vert U_{1}-U_{2}\right\Vert
_{L^{2}\left( \left(0,t\right);\mathbb{X}^{2}\left( \Omega ,\Sigma \right)
\right) }^{2},  \label{diff1}
\end{equation}%
and%
\begin{align}
& \left\Vert \partial _{t}U_{1}-\partial _{t}U_{2}\right\Vert _{L^{2}(\left(
0,t\right) ;(D\left( {\mathcal{A}}_{\Theta ,\Sigma }\right) ))^{\ast
})}^{2}+\int_{0}^{t}{\mathcal{A}}_{\Theta ,\Sigma }\left( U_{1}\left(
\tau\right) -U_{2}\left( \tau\right) ,U_{1}\left( \tau\right) -U_{2}\left(
\tau\right) \right) d\tau  \label{diff2} \\
& \leq Ce^{\nu t}\left\Vert U_{1}\left( 0\right) -U_{2}\left( 0\right)
\right\Vert _{\mathbb{X}^{2}\left( \Omega ,\Sigma \right) }^{2},  \notag
\end{align}%
for some constants $\omega ,\nu >0,$ $M,K,C\geq 0,$ all independent of $t$
and $U_{i}.$
\end{lemma}

\begin{proof}
Recall that the injection $D\left( {\mathcal{A}}_{\Theta ,\Sigma }\right)
\hookrightarrow\mathbb{X}^{2}\left( \Omega ,\Sigma \right) $ is compact and
continuous. Owing to Lemma \ref{dissipation}, we also have%
\begin{equation}
\sup_{t\geq 0}\left( \left\Vert U_{i}\left( t\right) \right\Vert _{\mathbb{X}%
^{\infty }\left( \Omega ,\Sigma \right) }+{\mathcal{A}}_{\Theta ,\Sigma
}(U_{i}\left( t\right) ,U_{i}\left( t\right) )\right) \leq C=C\left(
\left\Vert U_{i}(0)\right\Vert _{\mathcal{B}}\right) ,\text{ }i=1,2.
\label{4.5}
\end{equation}%
Setting $U:=U_{1}-U_{2}$, in light of Definition \ref{solu} the identity%
\begin{align}
& \int_{\Omega \backslash \Sigma }\partial _{t}u\left( t\right) \xi
dx+\int_{\Sigma }\partial _{t}u\left( t\right) \xi |_{\Sigma }d\mu _{\Sigma
}+\mathcal{A}_{\Theta ,\Sigma }(U\left( t\right) ,\xi )+\int_{\Omega
\backslash \Sigma }\left( f\left( u_{1}\left( t\right) \right) -f\left(
u_{2}\left( t\right) \right) \right) \xi dx  \label{vari} \\
& =\int_{\Sigma }\left( h\left( u_{1}\left( t\right) \right) -h\left(
u_{2}\left( t\right) \right) \right) \xi |_{\Sigma }d\mu _{\Sigma }  \notag
\end{align}%
holds for all $\xi \in D(\mathcal{A}_{\Theta,\Sigma}),$ a.e. $t\in \left(
0,T\right) $. Choosing $\xi =U\left( t\right) $ into (\ref{vari}) and owing
to the uniform bound (\ref{4.5}), we deduce%
\begin{equation*}
\frac{d}{dt}\left\Vert U\left( t\right) \right\Vert _{\mathbb{X}^{2}\left(
\Omega ,\Sigma \right) }^{2}+2\mathcal{A}_{\Theta ,\Sigma }(U\left( t\right)
,U\left( t\right) )\leq C_{f,h}\left( \left\Vert U(0)\right\Vert _{\mathcal{B%
}}\right) \left\Vert U\left( t\right) \right\Vert _{\mathbb{X}^{2}\left(
\Omega ,\Sigma \right) }^{2}
\end{equation*}%
for some constant $C_{f,h}>0$ which depends only on $f,h$ and on the
constant from (\ref{4.5}). Integrating the foregoing inequality in time
entails the desired estimate (\ref{diff1}) and then the estimate%
\begin{equation}
\int_{0}^{t}{\mathcal{A}}_{\Theta ,\Sigma }\left( U\left( \tau\right)
,U\left( \tau\right) \right) d\tau\leq Ce^{\nu t}\left\Vert U_{1}\left(
0\right) -U_{2}\left( 0\right) \right\Vert _{\mathbb{X}^{2}\left( \Omega
,\Sigma \right) }^{2},  \label{diffbis2}
\end{equation}%
owing to the Gronwall inequality and the fact that $\left\Vert U \right\Vert
_{\mathbb{X}^{2}\left( \Omega ,\Sigma \right) }\leq \nu \left\Vert U
\right\Vert _{D\left( {\mathcal{A}}_{\Theta ,\Sigma }\right) }$, for some $%
\nu >0$. Finally, we observe that for any test function $\xi \in D(\mathcal{A%
}_{\Theta,\Sigma})$, the variational identity (\ref{vari}) (which actually
holds a.e. for $t>0$), there holds%
\begin{equation*}
\left(\partial _{t}U\left( t\right) ,\xi \right)_{\mathbb{X}%
^2(\Omega,\Sigma)} =-\mathcal{A}_{\Theta ,\Sigma }\left( U\left( t\right)
,\xi \right) -\left\langle F\left( U_{1}\left( t\right) \right) -F\left(
U_{2}\left( t\right) \right) ,\xi \right\rangle \leq C\left\Vert U\left(
t\right) \right\Vert _{D(\mathcal{A}_{\Theta ,{\Sigma }})}\left\Vert \xi
\right\Vert _{D(\mathcal{A}_{\Theta ,{\Sigma }})},
\end{equation*}%
since $f,h\in C_{\text{loc}}^{1}\left( \mathbb{R}\right) $, owing to (\ref%
{4.5}). This estimate together with (\ref{diffbis2}) gives the desired
control on the time derivative in (\ref{diff2}). The proof is finished.
\end{proof}

The last ingredient we need is the uniform H\"{o}lder continuity of the time
map $t\mapsto \mathcal{S}\left( t\right) U_{0}$ in the $\mathbb{X}^{\infty
}\left( \Omega ,\Sigma \right) $-norm, namely,

\begin{lemma}
\label{time_cont}Let the assumptions of Theorem \ref{m1} be satisfied.
Consider $U\left( t\right) =\mathcal{S}\left( t\right) U_{0}$ with $U_{0}\in
\mathcal{B}$ where $\mathcal{B}$ is given in \eqref{B}. Then the following
estimate holds:%
\begin{equation}
\left\Vert U\left( t\right) -U\left( \tau\right) \right\Vert _{\mathbb{X}%
^{\infty }\left( \Omega ,\Sigma \right) }\leq C\left\vert t-\tau\right\vert
^{\rho },\text{ for all }t,\tau\in (0,T],  \label{4.6}
\end{equation}%
where $\rho <1,C>0$ are independent of $t,\tau$, $U$.
\end{lemma}

\begin{proof}
Exploiting the bound (\ref{4.5}), by comparison in (\ref{de_form}), we have
as in the proof of Lemma \ref{L5} that%
\begin{equation*}
\int_{0}^{T}\left\Vert \partial _{t}U\left( t\right) \right\Vert _{(D(%
\mathcal{A}_{\Theta ,{\Sigma }}))^{\ast }}^{2}dt\leq C_{T},
\end{equation*}%
for any $T>0$ and $(D(\mathcal{A}_{\Theta,\Sigma}))^{\star}$ denotes the
dual of $D(A_{\Theta,\Sigma})$. This estimate entails the inequality%
\begin{equation}
\left\Vert U\left( t\right) -U\left( \tau\right) \right\Vert _{(D(\mathcal{A}%
_{\Theta ,{\Sigma }}))^{\ast }}\leq C_{T}\left\vert t-\tau\right\vert
^{\frac 12}\text{, for all }t,\tau\in \left[ 0,T\right] .  \label{4.6bis}
\end{equation}%
By a duality argument, (\ref{4.6bis}) and the uniform bound (\ref{4.5})
further yield%
\begin{equation}
\left\Vert U\left( t\right) -U\left( \tau\right) \right\Vert _{\mathbb{X}%
^{2}\left( \Omega ,\Sigma \right) }\leq C_{T}^{^{\prime }}\left\vert
t-\tau\right\vert ^{\frac 14},\text{ for all }t,\tau\in \left[ 0,T\right] .
\label{4.6tris}
\end{equation}%
Inequality (\ref{4.6}) is a consequence of (\ref{4.6tris}) and the $\mathbb{X%
}^{2}$-$\mathbb{X}^{\infty }$ smoothing property (\ref{sup-b}). Indeed, due
to the boundedness of $U\left( t\right) \in \mathbb{X}^{\infty }\left(
\Omega ,\Sigma \right) ,$ a.e. $t\geq 0$, the nonlinearities $f,h$ become
subordinated to the linear part of the equation (\ref{op_v}) no matter how
fast they grow. More precisely, obtaining the $\mathbb{X}^{2}$-$\mathbb{X}%
^{\infty }$\ continuous dependence estimate for the difference $U\left(
t\right) -U\left( \tau\right) $\ of any two strong solutions $U\left(
t\right) ,U\left( \tau\right) $ is actually reduced to the same iteration
procedure leading to (\ref{sup-b}) (cf. the proof of Lemma \ref{dissipation}%
). The proof is completed.
\end{proof}

We can now finish the proof of Theorem \ref{m1}, using the abstract scheme
of Proposition \ref{abstract}.

\begin{proof}[\textbf{Proof of Theorem \protect\ref{m1}}]
First, we construct the exponential attractor $\mathcal{M}_{d}$ of the
discrete map $\mathcal{S}\left( T^{\ast }\right) $ on $\mathcal{B}$ (the
above constructed absorbing ball in the space $\mathcal{X}$ given in %
\eqref{B}), for a sufficiently large $T^{\ast }$. Indeed, let $B_{1}=\left[
\cup _{t\geq T^{\ast }}\mathcal{S}\left( t\right) \mathcal{B}\right] _{%
\mathbb{X}^{2}(\Omega ,\Sigma )}$, where $\left[ \cdot \right] _{\mathbb{X}%
^{2}(\Omega ,\Sigma )}$ denotes the closure in the space $\mathbb{X}%
^{2}(\Omega ,\Sigma )$ and then set $\mathbb{B}:=\mathcal{S}\left( 1\right)
B_{1}$. Thus, $\mathbb{B}$ is a semi-invariant closed but also compact (for
the $\mathbb{X}^{2}(\Omega ,\Sigma )$ -metric) subset of the phase space $%
\mathbb{X}^{\infty }\left( \Omega ,\Sigma \right) $ and $\mathcal{S}\left(
T^{\ast }\right) :\mathbb{B}\rightarrow \mathbb{B}$, provided that $T^{\ast
} $ is large enough. Then, we apply Proposition \ref{abstract} on the set $%
\mathbb{B}$ with $\mathcal{H}=\mathbb{X}^{2}\left( \Omega ,\Sigma \right) $
and $\mathbb{S}=\mathcal{S}\left( T^{\ast }\right) ,$ with $T^{\ast }>0$
large enough so that $Me^{-\omega T^{\ast }}<\frac{1}{2}$ (see (\ref{diff1}%
)). Besides, letting%
\begin{align*}
\mathcal{V}_{1}& =L^{2}(\left( 0,T^{\ast }\right) ;D(\mathcal{A}_{\Theta ,{%
\Sigma }}))\cap W^{1,2}(\left( 0,T^{\ast }\right) ;(D(\mathcal{A}_{\Theta ,{%
\Sigma }}))^{\ast }), \\
\mathcal{V}& =L^{2}(\left( 0,T^{\ast }\right) ;\mathbb{X}^{2}\left( \Omega
,\Sigma \right) ),
\end{align*}%
we have that $\mathcal{V}_{1}\hookrightarrow \mathcal{V}$ is compact (owing
to the compactness of $D(\mathcal{A}_{\Theta ,{\Sigma }})\hookrightarrow
\mathbb{X}^{2}\left( \Omega ,\Sigma \right) \hookrightarrow (D(\mathcal{A}%
_{\Theta ,{\Sigma }}))^{\ast }$). Secondly, define $\mathbb{T}:\mathbb{B}%
\rightarrow \mathcal{V}_{1}$ to be the solving operator for (\ref{p1b})-(\ref%
{p4b}) on the time interval $\left[ 0,T^{\ast }\right] $ such that $\mathbb{T%
}U_{0}:=U\in \mathcal{V}_{1},$ with $U\left( 0\right) =U_{0}\in \mathbb{B}.$
Due to Lemma \ref{L5}, (\ref{diff2}), we have the global Lipschitz
continuity (\ref{gl1}) of $\mathbb{T}$ from $\mathbb{B}$ to $\mathcal{V}_{1}$%
, and (\ref{diff1}) gives us the basic estimate (\ref{gl2}) for the map $%
\mathbb{S}=\mathcal{S}\left( T^{\ast }\right) $. Therefore, the assumptions
of Proposition \ref{abstract} are verified and, consequently, the map $%
\mathbb{S}=\mathcal{S}\left( T^{\ast }\right) $ possesses an exponential
attractor $\mathcal{M}_{d}$ on $\mathbb{B}$. In order to construct the
exponential attractor $\mathcal{G}_{\Theta ,\Sigma }$ for the semigroup $%
\mathcal{S}(t)$ with continuous time, we note that this semigroup is
Lipschitz continuous with respect to the initial data in the topology of $%
\mathbb{X}^{2}(\Omega ,\Sigma )$ (in fact it is also Lipschitz continuous
with respect to the metric topology of $\mathbb{X}^{\infty }(\Omega ,\Sigma
) $, owing to the $\mathbb{X}^{2}$-$\mathbb{X}^{\infty }$ smoothing
property). Moreover, by Lemma \ref{time_cont} the map $\left( t,U_{0}\right)
\mapsto \mathcal{S}\left( t\right) U_{0}$ is also uniformly H\"{o}lder
continuous on $\left[ 0,T^{\ast }\right] \times \mathbb{B}$, where $\mathbb{B%
}$ is endowed with the metric topology of $\mathbb{X}^{\infty }(\Omega
,\Sigma )$. Hence, the desired exponential attractor $\mathcal{G}_{\Theta
,\Sigma }$ for the continuous semigroup $\mathcal{S}(t)$ can be obtained by
the standard formula%
\begin{equation}
\mathcal{G}_{\Theta ,\Sigma }=\bigcup_{t\in \left[ 0,T^{\ast }\right] }%
\mathcal{S}\left( t\right) \mathcal{M}_{d}.  \label{st}
\end{equation}%
Finally, the finite-dimensionality of $\mathcal{G}_{\Theta ,\Sigma }$ in $%
\mathbb{X}^{\infty }(\Omega ,\Sigma )$ follows from the finite
dimensionality of $\mathcal{M}_{d}$ in $\mathbb{X}^{2}(\Omega ,\Sigma )$ and
the $\mathbb{X}^{2}$-$\mathbb{X}^{\infty }$ smoothing property. The
remaining properties of $\mathcal{G}_{\Theta ,\Sigma }$ are also immediate.
Theorem \ref{m1} is now proved.
\end{proof}

\section{Blow-up results}

\label{bl}

The main results of this section deal with blow-up phenomena for the strong
solutions of \eqref{p1b}-\eqref{p4b}. To this end, we define the following
energy functional%
\begin{equation}
E\left( t\right) :=\frac{1}{2}\mathcal{A}_{\Theta ,\Sigma }(U\left( t\right)
,U\left( t\right) )+\left( \overline{f}\left( u\left( t\right) \right)
,1\right) _{L^{2}\left( \Omega \backslash \Sigma \right) }-\left( \overline{h%
}\left( u\left( t\right) \right) ,1\right) _{L^{2}\left( \Sigma ,\mu
_{\Sigma }\right) }  \label{energy}
\end{equation}%
and notice that%
\begin{equation}
E\left( t\right) +\int_{0}^{t}\left( \left\Vert \partial _{t}u\left(
s\right) \right\Vert _{L^{2}\left( \Omega \right) }^{2}+\left\Vert \partial
_{t}u\left( s\right) \right\Vert _{L^{2}\left( \Sigma ,\mu _{\Sigma }\right)
}^{2}\right) ds\leq E\left( 0\right) ,  \label{en-ineq}
\end{equation}%
for as long the strong solution $U$ exists (cf. \eqref{en}) provided that in
addition $U_{0}\in D\left( \mathcal{A}_{\Theta ,\Sigma }\right) $. Recall
that $\overline{f}$ and $\overline{h}$ denote the primitives of $f$ and $h$,
respectively, and that $C_{\Omega ,\Sigma }^{\ast }>0$ is the Poincar\'{e}
constant from \eqref{P-C}. The energy inequality is satisfied for instance
by any strong solution of Theorem \ref{SG} (on some interval $\left(
0,T_{\ast }\right) $) provided that in addition $U_{0}\in D\left( \mathcal{A}%
_{\Theta ,\Sigma }\right) \cap \mathbb{X}^{\infty }\left( \Omega ,\Sigma
\right) $. The validity of \eqref{en-ineq} on some interval on which the
strong solution $U$ exists can be easily checked by first verifying that the
energy identity (that is \eqref{en-ineq} with equality) holds for a sequence
of approximate solutions $U_{n}\in L^{\infty }\left((0,T_{\ast });D\left(
A_{\Theta ,\Sigma }\right) \right) \cap W^{1,2}\left((0,T_{\ast });D\left(
\mathcal{A}_{\Theta ,\Sigma }\right) \right) ,$ uniformly in $n\geq 1$,
associated with a given smooth initial datum $U_{0n}\in D\left( A_{\Theta ,{%
\Sigma }}\right) \cap \mathbb{X}^{\infty }\left( \Omega ,\Sigma \right) $
such that $U_{0n}\rightarrow U_{0}$ strongly in $D\left( \mathcal{A}_{\Theta
,\Sigma }\right) \cap \mathbb{X}^{\infty }\left( \Omega ,\Sigma \right) $.
Integrating the corresponding energy identity for $U_{n}\left( t\right) $ on
the time interval $\left( 0,t\right) $ and exploiting standard convergence
results for these approximate solutions, together with the weak
lower-semicontinuity of the form $\mathcal{A}_{\Theta ,\Sigma }$, we can
easily infer \eqref{en-ineq}. It turns out that the validity of %
\eqref{en-ineq} is sufficient for our goals below. In particular, we will
show that every strong solution $U$ of problem \eqref{p1b}-\eqref{p4b} that
obeys the energy inequality \eqref{en-ineq} must blow-up in finite time
under some general conditions on the nonlinearities.

\begin{theorem}
\label{theo2}Assume that $\Omega \backslash \Sigma $ has the $\widetilde{W}%
^{1,2}$-extension property and the hypotheses of Lemma \ref{poincare-weak}
are satisfied. Let $U$ be a (local) strong solution of \eqref{p1b}-%
\eqref{p4b} in the sense of Theorem \ref{SG} for some initial datum $%
U_{0}\in D\left( \mathcal{A}_{\Theta ,\Sigma }\right) \cap \mathbb{X}%
^{\infty }\left( \Omega ,\Sigma \right) $. Let $\alpha >2$ and define $%
g_{\alpha }\left( \tau \right) =-f\left( \tau \right) \tau +\alpha \overline{%
f}\left( \tau \right) $ and $l_{\alpha }\left( \tau \right) =-h\left( \tau
\right) \tau +\alpha \overline{h}\left( \tau \right) $ for $\tau \in \mathbb{%
R}$. Suppose there exist $\varepsilon \in \left( 0,\left( \alpha /2-1\right)
d_{0}\right) $ and constants $C_{1}>0,C_{2}\geq 0$ such that%
\begin{equation}
-g_{\alpha }\left( \tau \right) +\frac{\mu _{\Sigma }\left( \Sigma \right) }{%
\left\vert \Omega \right\vert }l_{\alpha }\left( \tau \right) -\frac{\left(
C_{\Omega ,\Sigma }^{\ast }\right) ^{2}}{4\varepsilon }\left( l_{\alpha
}^{^{\prime }}\left( \tau \right) \right) ^{2}\geq C_{1}\tau ^{2}-C_{2},%
\text{ for all }\tau \in \mathbb{R}\text{.}  \label{blow-bala}
\end{equation}%
Then there exist two constants $D_{1}>0,D_{2}>0$ (depending only on $\mu
_{\Sigma }\left( \Sigma \right) ,\left\vert \Omega \right\vert ,\varepsilon
,C_{1},C_{2},d_{0}$ and $\alpha $) such that for $U_{0}$ satisfying%
\begin{equation}
D_{1}\left\Vert U_{0}\right\Vert _{\mathbb{X}^{2}\left( \Omega ,\Sigma
\right) }^{2}>\alpha E\left( 0\right) +D_{2},  \label{ini}
\end{equation}%
the strong solution $U$ of \eqref{p1b}-\eqref{p4b} blows up in finite time.
\end{theorem}

\begin{proof}
Since the set $\Omega \backslash \Sigma $ is generally "rough", we will
obtain the result by exploiting an energy method and the concavity method
due to Levine and Payne \cite{LP}. With this technique we will prove that
some strong solutions of \eqref{p1b}-\eqref{p4b} must cease to exist in
finite time, since otherwise the $\mathbb{X}^{2}\left( \Omega ,\Sigma
\right) $-norm must become infinite in finite time. To this end let us
define
\begin{equation*}
G\left( t\right) =\frac{1}{2}\int_{\Omega \backslash \Sigma }\left\vert
u\left( t\right) \right\vert ^{2}dx+\frac{1}{2}\int_{\Sigma }\left\vert
u\left( t\right) \right\vert ^{2}d\mu _{\Sigma }.
\end{equation*}%
The starting point is the energy identity (\ref{4.1}), which can be
rewritten using the energy $E\left( t\right) $ (see (\ref{energy})) as
follows:%
\begin{align}
G^{^{\prime }}\left( t\right) =&-\alpha E\left( t\right) +\left( \frac{%
\alpha }{2}-1\right) \mathcal{A}_{\Theta ,\Sigma }(U,U)  \label{4.7} \\
& +\int_{\Sigma }\left( \alpha \overline{h}\left( u\right) -h\left( u\right)
u\right) d\mu _{\Sigma }-\int_{\Omega \backslash \Sigma }\left( \alpha
\overline{f}\left( u\right) -f\left( u\right) u\right) dx  \notag \\
\geq &\alpha \int_{0}^{t}\left\Vert \partial _{t}U\left( s\right)
\right\Vert _{\mathbb{X}^{2}\left( \Omega ,\Sigma \right) }^{2}ds+\left(
\frac{\alpha }{2}-1\right) \mathcal{A}_{\Theta ,\Sigma }(U,U)  \notag \\
& +\int_{\Sigma }l_{\alpha }\left( u\right) d\mu _{\Sigma }-\int_{\Omega
\backslash \Sigma }g_{\alpha }\left( u\right) dx-\alpha E\left( 0\right) .
\notag
\end{align}%
for $t\geq 0$, owing to (\ref{en-ineq}). Now we estimate the nonlinear terms
on the right-hand side of (\ref{4.7}). Exactly as in (\ref{bala-1}) we have%
\begin{align}
\int_{\Sigma }l_{\alpha }\left( u\right) d\mu _{\Sigma }-\int_{\Omega
\backslash \Sigma }g_{\alpha }\left( u\right) dx& =\int_{\Omega \backslash
\Sigma }\left( -g_{\alpha }\left( u\right) +\frac{\mu _{\Sigma }\left(
\Sigma \right) }{\left\vert \Omega \right\vert }l_{\alpha }\left( u\right)
\right) dx  \label{4.8} \\
& +\frac{\mu _{\Sigma }\left( \Sigma \right) }{\left\vert \Omega \right\vert
}\int_{\Omega \backslash \Sigma }\left( l_{\alpha }\left( u\right) -\frac{1}{%
\mu _{\Sigma }\left( \Sigma \right) }\int_{\Sigma }l_{\alpha }\left(
u\right) d\mu _{\Sigma }\right) dx.  \notag
\end{align}%
We can apply the Poincar\'{e} inequality of Lemma \ref{poincare-weak}
yielding%
\begin{align}
& \left\vert \frac{\mu _{\Sigma }\left( \Sigma \right) }{\left\vert \Omega
\right\vert }\int_{\Omega \backslash \Sigma }\left( l_{\alpha }\left(
u\right) -\frac{1}{\mu _{\Sigma }\left( \Sigma \right) }\int_{\Sigma
}l_{\alpha }\left( u\right) d\mu _{\Sigma }\right) dx\right\vert  \label{4.9}
\\
& \leq C_{\Omega ,\Sigma }^{\ast }\left\Vert \nabla \left( l_{\alpha }\left(
u\right) \right) \right\Vert _{L^{1}\left( \Omega \backslash \Sigma \right)
}=C_{\Omega ,\Sigma }^{\ast }\left\Vert l_{\alpha }^{^{\prime }}\left(
u\right) \nabla u\right\Vert _{L^{1}\left( \Omega \backslash \Sigma \right) }
\notag \\
& \leq \varepsilon \left\Vert \nabla u\right\Vert _{L^{2}\left( \Omega
\backslash \Sigma \right) }^{2}+\frac{\left( C_{\Omega ,\Sigma }^{\ast
}\right) ^{2}}{4\epsilon }\left\Vert l_{\alpha }^{^{\prime }}\left( u\right)
\right\Vert _{L^{2}\left( \Omega \backslash \Sigma \right) }^{2},  \notag
\end{align}%
for every $\varepsilon \in \left( 0,\left( \alpha /2-1\right) d_{0}\right) .$
Inserting \eqref{4.9} into \eqref{4.8} gives%
\begin{align*}
& \int_{\Sigma }l_{\alpha }\left( u\right) d\mu _{\Sigma }-\int_{\Omega
\backslash \Sigma }g_{\alpha }\left( u\right) dx \\
& \geq \int_{\Omega \backslash \Sigma }\left( -g_{\alpha }\left( u\right) +%
\frac{\mu _{\Sigma }\left( \Sigma \right) }{\left\vert \Omega \right\vert }%
l_{\alpha }\left( u\right) -\frac{\left( C_{\Omega ,\Sigma }^{\ast }\right)
^{2}}{4\varepsilon }\left( l_{\alpha }^{^{\prime }}\left( u\right) \right)
^{2}\right) dx-\frac{\varepsilon }{d_{0}}\int_{\Omega \backslash \Sigma
}\left\vert \mathbf{D}\nabla u\right\vert ^{2}dx
\end{align*}%
so that \eqref{4.7} now reads%
\begin{align}
G^{^{\prime }}\left( t\right) \geq & \alpha \int_{0}^{t}\left\Vert \partial
_{t}U\left( \tau\right) \right\Vert _{\mathbb{X}^{2}\left( \Omega ,\Sigma
\right) }^{2}d\tau+\frac{1}{d_{0}}\left( \left( \frac{\alpha }{2}-1\right)
d_{0}-\varepsilon \right) \mathcal{A}_{\Theta ,\Sigma }(U,U)  \label{4.10} \\
& +\int_{\Omega \backslash \Sigma }\left( -g_{\alpha }\left( u\right) +\frac{%
\mu _{\Sigma }\left( \Sigma \right) }{\left\vert \Omega \right\vert }%
l_{\alpha }\left( u\right) -\frac{\left( C_{\Omega ,\Sigma }^{\ast }\right)
^{2}}{4\varepsilon }\left( l_{\alpha }^{^{\prime }}\left( u\right) \right)
^{2}\right) dx-\alpha E\left( 0\right)  \notag
\end{align}%
for $t\geq 0$. From \eqref{blow-bala} we get%
\begin{equation*}
\int_{\Omega \backslash \Sigma }\left( -g_{\alpha }\left( u\right) +\frac{%
\mu _{\Sigma }\left( \Sigma \right) }{\left\vert \Omega \right\vert }%
l_{\alpha }\left( u\right) -\frac{\left( C_{\Omega ,\Sigma }^{\ast }\right)
^{2}}{4\varepsilon }\left( l_{\alpha }^{^{\prime }}\left( u\right) \right)
^{2}\right) dx\geq C_{1}\int_{\Omega \backslash \Sigma }\left\vert
u\right\vert ^{2}dx-C_{2}\left\vert \Omega \right\vert
\end{equation*}%
and using the fact that the injection $D\left( \mathcal{A}_{\Theta ,\Sigma
}\right) \hookrightarrow \mathbb{X}^{2}\left( \Omega ,\Sigma \right) $ is
continuous with a constant $\widetilde{C}_{\Omega ,\Sigma }>0$, from (\ref%
{4.10}) we obtain that $G\left( t\right) $ satisfies the initial value
problem for the differential inequality
\begin{equation}
\left\{
\begin{array}{ll}
G^{^{\prime }}\left( t\right) \geq D_{1}G\left( t\right) -\left(
D_{2}+\alpha E\left( 0\right) \right) , & \text{for }t\geq 0 \\
G\left( 0\right) =2\left\Vert U_{0}\right\Vert _{\mathbb{X}^{2}\left( \Omega
,\Sigma \right) }^{2}, &
\end{array}%
\right.  \label{4.11}
\end{equation}%
with%
\begin{equation*}
D_{1}=2\left( \frac{1}{d_{0}}\left( \left( \frac{\alpha }{2}-1\right)
d_{0}-\varepsilon \right) \widetilde{C}_{\Omega ,\Sigma }+C_{1}\right) >0,%
\text{ }D_{2}=C_{2}\left\vert \Omega \right\vert >0.
\end{equation*}%
Let us now define $H\left( U\left( t\right) \right) :=D_{1}G\left( t\right)
-\left( D_{2}+\alpha E\left( 0\right) \right) $ and observe that (\ref{ini})
is equivalent to $H\left( U_{0}\right) >0,$ in which case from (\ref{4.11})
we deduce that $H\left( U\left( t\right) \right) >0$ for all $t\geq 0$ (for
as long as the solution exists) and the function $G\left( t\right) $ grows
at least exponentially fast. Let us now employ a contradiction argument
similar to arguments used in \cite{LP}. Let us suppose that the strong
solution $U$ is defined for all times $t>0$. Then from (\ref{4.10}) we see
that%
\begin{equation}
G^{^{\prime }}\left( t\right) \geq \alpha \int_{0}^{t}\left\Vert \partial
_{t}U\left( \tau \right) \right\Vert _{\mathbb{X}^{2}\left( \Omega ,\Sigma
\right) }^{2}d\tau +H\left( U\left( t\right) \right) >\alpha
\int_{0}^{t}\left\Vert \partial _{t}U\left( \tau \right) \right\Vert _{%
\mathbb{X}^{2}\left( \Omega ,\Sigma \right) }^{2}d\tau .  \label{4.12}
\end{equation}%
Next, denote by $M\left( t\right) :=\int_{0}^{t}\left\Vert U\left( \tau
\right) \right\Vert _{\mathbb{X}^{2}\left( \Omega ,\Sigma \right) }^{2}d\tau
$ so that (\ref{4.12}) implies that%
\begin{equation}
M^{^{\prime \prime }}\left( t\right) >2\alpha \int_{0}^{t}\left\Vert
\partial _{t}U\left( \tau \right) \right\Vert _{\mathbb{X}^{2}\left( \Omega
,\Sigma \right) }^{2}d\tau .  \label{M}
\end{equation}%
Multiplying \eqref{M} by $M\left( t\right) $ and applying the Cauchy-Schwarz
inequality we derive%
\begin{equation*}
M^{^{\prime \prime }}\left( t\right) M\left( t\right) >\frac{\alpha }{2}%
\left( M^{^{\prime }}\left( t\right) -M\left( 0\right) \right) ^{2}.
\end{equation*}%
Now since $M^{^{\prime }}\left( t\right) =2G\left( t\right) \rightarrow
\infty $ as time $t\rightarrow \infty $ and $\alpha >2$, there exists a
sufficiently small $\varepsilon \in \left( 0,\alpha /2\right) $ such that
for large time $t>0$,
\begin{equation*}
M^{^{\prime \prime }}\left( t\right) M\left( t\right) >\left( \frac{\alpha }{%
2}-\varepsilon \right) \left( M^{^{\prime }}\left( t\right) \right) ^{2}.
\end{equation*}%
This inequality yields that $M^{-\varepsilon +1-\alpha /2}\left( t\right) >0$
is a concave function for large time $t>0$ but this is impossible since $%
M^{-\varepsilon +1-\alpha /2}\left( t\right) \rightarrow 0$ as $t\rightarrow
\infty $. Hence, the solution $U$ must blow-up in finite time. The proof is
finished.
\end{proof}

The following result shows that in the case of polynomial nonlinearities
with a \emph{good dissipative source} $h$ of \emph{arbitrary growth} along
the sharp interface $\Sigma $ but with a bulk source $f$ with \emph{bad sign}
at infinity, blow-up of some strong solutions still occurs provided that $h$
is dominated by $f$. In some sense, this result is in contrast to the result
obtained in Corollary \ref{example-C} which asserts the global existence of
solutions with a \emph{bad dissipative source} $h$ of \emph{arbitrary growth}
along the interface $\Sigma $ but with a bulk source $f$ with \emph{good sign%
} at infinity.

\begin{corollary}
Let the assumptions of Theorem \ref{SG} and Lemma \ref{poincare-weak} be
satisfied. Suppose that
\begin{equation*}
\lim_{\left\vert \tau \right\vert \rightarrow \infty }\frac{h^{^{\prime
}}\left( \tau \right) }{\left\vert \tau \right\vert ^{p}}=\left( p+1\right)
c_{h}\text{ and }\lim_{\left\vert \tau \right\vert \rightarrow \infty }\frac{%
f^{^{\prime }}\left( \tau \right) }{\left\vert \tau \right\vert ^{q}}=\left(
q+1\right) c_{f}
\end{equation*}%
with $c_{h}<0,c_{f}<0$ for some $p,q\geq 0$. We have the following cases:

\begin{enumerate}
\item Let $\alpha \in \left( p+2,q+2\right) $ and $U_{0}\in D\left( \mathcal{%
A}_{\Theta ,\Sigma }\right) \cap \mathbb{X}^{\infty }\left( \Omega ,\Sigma
\right) $ such that \eqref{ini} is satisfied, and assume $q>2p.$

\item Under the same assumption on the initial datum, assume $q=2p$ and%
\begin{equation*}
-c_{f}\left( 1-\frac{\alpha }{q+2}\right) >\frac{\left( C_{\Omega ,\Sigma
}^{\ast }\right) ^{2}c_{h}^{2}\left( p+2-\alpha \right) ^{2}}{4\varepsilon },
\end{equation*}%
for some $\varepsilon \in \left( 0,\left( \alpha /2-1\right) d_{0}\right) .$

\item Let $\alpha \in \left( 2,q+2\right) $ with $U_{0}\in D\left( \mathcal{A%
}_{\Theta ,\Sigma }\right) \cap \mathbb{X}^{\infty }\left( \Omega ,\Sigma
\right) $ satisfying \eqref{ini} and assume $q>2p.$
\end{enumerate}

Then, in each case, the strong solution $U\left( t\right) $ associated with
the corresponding initial datum $U_{0}$ blows up in finite time.
\end{corollary}

\begin{proof}
We begin by noting that for large $\left\vert \tau\right\vert \geq \tau_{0}$%
, we have
\begin{align*}
h\left( \tau\right) \sim c_{h}\left\vert \tau\right\vert
^{p}\tau,\;\;h\left( \tau\right) \tau\sim c_{h}\left\vert \tau\right\vert
^{p+2}\;\mbox{ and }\;f\left( \tau\right) \sim c_{f}\left\vert
\tau\right\vert ^{q}\tau, \;\;f\left( \tau\right) \tau\sim c_{f}\left\vert
\tau\right\vert ^{q+2}.
\end{align*}
Moreover, as $\left\vert \tau\right\vert \rightarrow \infty $, we have
\begin{align*}
\overline{f}\left( \tau\right) \sim \frac{c_{f}}{q+2}\left\vert
\tau\right\vert ^{q+2}\;\mbox{ and }\;\overline{h}\left( \tau\right) \sim
\frac{c_{h}}{p+2}\left\vert \tau\right\vert ^{p+2}
\end{align*}
which yields
\begin{align*}
g_{\alpha}\left( \tau\right) \sim c_{f}\left( 1-\alpha /\left( q+2\right)
\right) \left\vert \tau\right\vert ^{q+2}\;\mbox{ and }\;l_{\alpha }\left(
\tau\right) \sim c_{h}\left( 1-\alpha /\left( p+2\right) \right) \left\vert
\tau\right\vert ^{p+2}.
\end{align*}
Thus, for large enough $\left\vert \tau\right\vert \geq \tau_{0}$, the
highest-order terms on the right-hand side of \eqref{blow-bala} are%
\begin{equation}
-c_{f}\left( 1-\alpha /\left( q+2\right) \right) \left\vert \tau\right\vert
^{q+2}+\frac{\mu _{\Sigma }\left( \Sigma \right) }{\left\vert \Omega
\right\vert }c_{h}\left( 1-\alpha /\left( p+2\right) \right) \left\vert
\tau\right\vert ^{p+2}-\frac{\left( C_{\Omega ,\Sigma }^{\ast }\right)
^{2}c_{h}^{2}\left( p+2-\alpha \right) ^{2}}{4\varepsilon }\left\vert
\tau\right\vert ^{2p+2}.  \label{4.15}
\end{equation}%
In the first case (a), the highest-order term in (\ref{4.15}) is the first
one since $-c_{f}>0$. Hence, \eqref{blow-bala} is satisfied for some $%
\varepsilon \in \left( 0,\left( \alpha /2-1\right) d_{0}\right) $ and the
conclusion of Theorem \ref{theo2} applies. For the case (b), we notice that
the highest-order term is%
\begin{equation*}
\left( -c_{f}\left( 1-\frac{\alpha}{\left( q+2\right)} \right) -\frac{\left(
C_{\Omega ,\Sigma }^{\ast }\right) ^{2}c_{h}^{2}\left( p+2-\alpha \right)
^{2}}{4\varepsilon }\right) \left\vert \tau\right\vert ^{2p+2}
\end{equation*}%
so that \eqref{blow-bala} is once more satisfied if the coeffcient of this
term is positive. In the last case (c), we notice that since $p\leq 2p<q$
and the coefficients of the first and second terms in \eqref{4.15} are
positive and negative respectively, the highest-order term in \eqref{4.15}
is still the first one since $-c_{f}>0$. Therefore, \eqref{blow-bala} is
satisfied and the conclusion holds.
\end{proof}

The last result is of similar nature and roughly states that if both
nonlinearities have a bad sign at infinity in contrast to the conditions of
Corollary \ref{cor-weaker}, blow-up in finite time of some strong solutions
to the transmission problem still occurs.

\begin{theorem}
\label{theo3}Assume that $\Omega \backslash \Sigma $ has the $\widetilde{W}%
^{1,2}$-extension property and let $U$ be a (local) strong solution of %
\eqref{p1b}-\eqref{p4b} in the sense of Theorem \ref{SG}. Let $\alpha >2$
and suppose that there exist constants $C_{f},C_{h},C_{f}^{^{\prime
}},C_{h}^{^{\prime }}\geq 0$ such that%
\begin{equation}
\left\{
\begin{array}{l}
g_{\alpha }\left( \tau\right) =\alpha \overline{f}\left( \tau\right)
-f\left( \tau\right) \tau\leq -C_{f}\tau^{2}+C_{f}^{^{\prime }}, \\
l_{\alpha }\left( \tau\right) =\alpha \overline{h}\left( \tau\right)
-h\left( \tau\right) \tau\geq C_{h}\tau^{2}-C_{h}^{^{\prime }},%
\end{array}%
\right.  \label{4.16}
\end{equation}%
for all $\tau\in \mathbb{R}.$ Then there exist constants $D_{1}>0,D_{2}>0$
(depending only on $\mu _{\Sigma }\left( \Sigma \right) ,\left\vert \Omega
\right\vert ,$ the constants in \eqref{4.16} and $\alpha $) such that for
any initial datum $U_{0}\in D\left( \mathcal{A}_{\Theta ,\Sigma }\right)
\cap \mathbb{X}^{\infty }\left( \Omega ,\Sigma \right) $ satisfying%
\begin{equation}
D_{1}\left\Vert U_{0}\right\Vert _{\mathbb{X}^{2}\left( \Omega ,\Sigma
\right) }^{2}>\alpha E\left( 0\right) +D_{2},  \label{4.17}
\end{equation}%
the strong solution $U$ of \eqref{p1b}-\eqref{p4b} blows up in finite time.
\end{theorem}

\begin{proof}
In this case using \eqref{4.16} from \eqref{4.7} we get
\begin{align}
G^{^{\prime }}\left( t\right) & \geq \alpha \int_{0}^{t}\left\Vert \partial
_{t}U\left( \tau \right) \right\Vert _{\mathbb{X}^{2}\left( \Omega ,\Sigma
\right) }^{2}d\tau +\left( \frac{\alpha }{2}-1\right) \mathcal{A}_{\Theta
,\Sigma }(U,U)  \label{4.18} \\
& +C_{f}\int_{\Omega \backslash \Sigma }\left\vert u\right\vert
^{2}dx+C_{h}\int_{\Sigma }\left\vert u\right\vert ^{2}d\mu _{\Sigma
}-C_{f}^{^{\prime }}\left\vert \Omega \right\vert -C_{h}^{^{\prime }}\mu
_{\Sigma }\left( \Sigma \right) -\alpha E\left( 0\right) .  \notag
\end{align}%
Thus also in this case we deduce that $G\left( t\right) $ satisfies %
\eqref{4.11} with
\begin{equation*}
D_{1}=2\left( \left( \frac{\alpha }{2}-1\right) \widetilde{C}_{\Omega
,\Sigma }+\min \left( C_{f},C_{h}\right) \right) >0,\text{ }%
D_{2}=C_{f}^{^{\prime }}\left\vert \Omega \right\vert +C_{h}^{^{\prime }}\mu
_{\Sigma }\left( \Sigma \right) .
\end{equation*}%
As before all solutions of \eqref{4.11} with $H\left( U_{0}\right) >0$
(which is equivalent to \eqref{4.17}) must blow-up in finite time. The proof
is finished.
\end{proof}

\begin{acknowledgement}
The authors thank the anonymous referees for their careful reading and their
important remarks on an earlier version of the manuscript.
\end{acknowledgement}

\end{document}